\documentclass[12pt]{article}
\usepackage[utf8]{inputenc}

\usepackage{amsmath} 
\usepackage[colorlinks=true,bookmarks=false,linkcolor=blue,urlcolor=blue,citecolor=blue,breaklinks=true,hypertexnames=false]{hyperref}
\usepackage{breakcites}

\usepackage{doi}
\usepackage{graphicx}

\usepackage{amssymb} 
\usepackage{amsfonts} 
\usepackage{amsthm} 
\usepackage[shortlabels]{enumitem} 
\usepackage{comment} 
\usepackage{bbm}
\usepackage{mathtools}
\usepackage{tikz-cd}
\usepackage[margin=1in]{geometry} 
\usepackage{algorithm}
\usepackage{algorithmicx}
\usepackage{algpseudocode}
\usepackage{multirow} 
\usepackage{booktabs} 
\usepackage[nottoc]{tocbibind} 
\usepackage{multicol}

\usepackage{ifthen}
\newboolean{shortver}

\newtheorem{theorem}{Theorem}[section]
\newtheorem{lemma}[theorem]{Lemma}
\newtheorem{proposition}[theorem]{Proposition}
\newtheorem{corollary}[theorem]{Corollary}
\newtheorem{example}[theorem]{Example}
\newtheorem{definition}[theorem]{Definition}
\newtheorem{remark}[theorem]{Remark}

\newtheorem{assumption}[theorem]{Assumption}
\newtheorem{conjecture}[theorem]{Conjecture}

\numberwithin{equation}{section}

\makeatletter
\renewcommand{\paragraph}{%
  \@startsection{paragraph}{4}%
  {\z@}{1ex \@plus 1ex \@minus .2ex}{-.5em}%
  {\normalfont\normalsize\bfseries}%
}
\makeatother

\newcommand{\be}{\begin{equation}}
\newcommand{\ee}{\end{equation}}
\newcommand{\bthm}{\begin{theorem}}
\newcommand{\ethm}{\end{theorem}}
\newcommand{\blem}{\begin{lemma}}
\newcommand{\elem}{\end{lemma}}
\newcommand{\bpof}{\begin{proof}}
\newcommand{\epof}{\end{proof}}
\newcommand{\bcor}{\begin{corollary}}
\newcommand{\ecor}{\end{corollary}}
\newcommand{\bprop}{\begin{proposition}}
\newcommand{\eprop}{\end{proposition}}

\newcommand{\rank}{\mathrm{rank}}

\newcommand{\diag}{\mathrm{diag}}

\newcommand{\mc}[1]{\mathcal{#1}}
\newcommand{\mbb}[1]{\mathbb{#1}}

\newcommand{\transpose}{^\top\! }
\newcommand{\inner}[2]{\left\langle{#1},{#2}\right\rangle}

\newcommand{\Proj}{\mathrm{Proj}}

\newcommand{\T}{\mathrm{T}}

\newcommand{\Rmn}{{\mathbb{R}^{m\times n}}}

\newcommand{\D}{\mathrm{D}}

\newcommand{\calE}{\mathcal{E}}

\newcommand{\calU}{\mathcal{U}}

\newcommand{\calM}{\mathcal{M}}
\newcommand{\calN}{\mathcal{N}}

\newcommand{\calX}{\mathcal{X}}

\newcommand{\dist}{\mathrm{dist}}
\newcommand{\Lmap}{\mathbf{L}}
\newcommand{\Qmap}{\mathbf{Q}}
\newcommand{\oneimpliesone}{``1 $\Rightarrow\!$ 1''}
\newcommand{\twoimpliesone}{``2 $\Rightarrow\!$ 1''}
\newcommand{\kimpliesone}{``$k\! \Rightarrow\!$ 1''}
\newcommand{\kimpliesk}{``$k\! \Rightarrow\! k$''}
\newcommand{\localimplieslocal}{``local $\Rightarrow\!$ local''}

\newcommand{\reals}{\mathbb{R}} 

\newcommand{\im}{\operatorname{im}}

\newcommand{\Rmnlr}{\reals^{m \times n}_{\leq r}}

\newcommand{\RR}{\mathbb{R}}

\newcommand{\NN}{\mathbb{N}}

\setboolean{shortver}{false}

\providecommand{\keywords}[1]
{
  \small	
  \textbf{Keywords:} #1
}

\title{The effect of smooth parametrizations\\on nonconvex optimization landscapes}
\author{Eitan Levin\thanks{Dept.\ of Computing and Mathematical Sciences, California Institute of Technology (\href{mailto:eitanl@caltech.edu}{eitanl@caltech.edu})}
\and 
Joe Kileel\thanks{Dept.\ of Mathematics and Oden Institute for Computational Engineering and Sciences, University of Texas at Austin (\href{mailto:jkileel@math.utexas.edu}{jkileel@math.utexas.edu})}
\and 
Nicolas Boumal\thanks{Institute of Mathematics, \'Ecole polytechnique f\'ed\'erale de Lausanne (EPFL) (\href{mailto:nicolas.boumal@epfl.ch}{nicolas.boumal@epfl.ch})}
}
\date{\today}

\begin{document}

\maketitle

\begin{abstract}
We develop new tools to study landscapes in nonconvex optimization.
Given one optimization problem, we pair it with another by smoothly parametrizing the domain.
This is either for practical purposes (e.g., to use smooth optimization algorithms with good guarantees) or for theoretical purposes (e.g., to reveal that the landscape satisfies a strict saddle property).
In both cases, the central question is: how do the landscapes of the two problems relate?
More precisely: how do desirable points such as local minima and critical points in one problem relate to those in the other problem?
A key finding in this paper is that these relations are often determined by the parametrization itself, and are almost entirely independent of the cost function.
Accordingly, we introduce a general framework to study parametrizations by their effect on landscapes.
The framework enables us to obtain new guarantees for an array of problems, some of which were previously treated on a case-by-case basis in the literature.
Applications include: optimizing low-rank matrices and tensors through factorizations; solving semidefinite programs via the Burer--Monteiro approach; training neural networks by optimizing their weights and biases; and quotienting out symmetries.

\end{abstract}


\tableofcontents

\section{Introduction}\label{sec:intro}

We consider pairs of optimization problems~\eqref{eq:P} and~\eqref{eq:Q} as defined below, where $\calE$ is a linear space, $\calM$ is a smooth manifold\footnote{All linear spaces and manifolds are assumed to be finite-dimensional.}, and $\varphi \colon \calM \to \calE$ is a smooth (over)parametrization of the search space $\calX = \varphi(\calM)$ of~\eqref{eq:P}.\footnote{The two-headed arrow~$\twoheadrightarrow$ in the diagram denotes a surjection.} Their optimal values are equal:

\vspace{2mm}
\noindent%
\begin{minipage}[c]{.4\linewidth}
    \[\begin{tikzcd}
        {\calM} \\
        {\qquad\calX \subseteq \calE} & \reals
        \arrow["\varphi"', two heads, from=1-1, to=2-1]
        \arrow["f"', from=2-1, to=2-2]
        \arrow["{g=f\circ\varphi}", from=1-1, to=2-2]
    \end{tikzcd}\]
\end{minipage}
\noindent%
\begin{minipage}[c]{.6\linewidth}
    \begin{align}
        \min_{y \in \calM} g(y) \label{eq:Q} \tag{Q} \\
        \min_{x \in \calX} f(x) \label{eq:P} \tag{P} 
    \end{align}
    \vspace{3mm}
\end{minipage}
We usually assume $f \colon \calE \to \reals$ is smooth ($C^{\infty}$), hence so is $g = f \circ \varphi \colon \calM \to \reals$ by composition.



Such pairs of problems~\eqref{eq:P} and~\eqref{eq:Q} arise in \textbf{two scenarios} (concrete examples follow):
\begin{enumerate}[(a)]
    \item Our task is to minimize $f$ on $\calX$ as in~\eqref{eq:P}, but we lack good algorithms to do so, e.g., because $\calX$ lacks regularity. In this case, we choose a smooth parametrization $\varphi$ of $\calX$ and run algorithms on the smooth problem~\eqref{eq:Q} instead.
    
    \item Our task is to minimize $g$ on $\calM$ as in~\eqref{eq:Q}, but its landscape is complex (e.g., due to symmetries). In this case, we factor $g$ through a smooth map $\varphi$ in the hope of revealing a problem~\eqref{eq:P} whose landscape is simpler and can be leveraged to analyze that of~\eqref{eq:Q}.
\end{enumerate}
In both cases, we run an optimization algorithm on the smooth problem~\eqref{eq:Q}. This algorithm may find desirable points $y$ on $\calM$ for~\eqref{eq:Q} (global or local minima, stationary points). For example, certain trust-region algorithms are guaranteed to accumulate at second-order stationary points--see~\cite{curtis2018concise} and an extension to manifolds~\cite[\S3]{levin2021finding}--and many first- and even zeroth-order methods converge to second-order stationary points from almost all initializations~\cite{antonakopoulos2022adagrad,lee2019first,zerothSaddleEscape}. However, in general such points need not map to desirable points $\varphi(y)$ on $\calX$ for~\eqref{eq:P}. Indeed, nonlinear parametrizations may severely distort landscapes, and notably may introduce spurious critical points. Algorithms running on~\eqref{eq:Q} are liable to terminate at an approximately stationary point near such a spurious point, and return a point whose image through $\varphi$ is nowhere near any stationary point for~\eqref{eq:P}.

In this paper, we characterize the properties that the parametrization $\varphi$ needs to satisfy for desirable points of~\eqref{eq:Q} to map to desirable points of~\eqref{eq:P}, that is, we develop a general framework to relate the landscapes of pairs of problems of the above form. Importantly, we observe that these properties are often entirely independent of the cost function $f$ in~\eqref{eq:P}, since many parametrizations map desirable points for~\eqref{eq:Q} to those for~\eqref{eq:P} for any cost function. Our framework enables us to unify and strengthen the analysis of a wealth of parametrizations arising in applications, hitherto studied case-by-case and often only for specific costs.

Parametrizations are ubiquitous. They arise in semidefinite programming~\cite{BM_det_guarantees,cifuentes2021burer}, low-rank optimization~\cite{ha2020criticallowrank,eitan_thesis,mishra2014fixed}, computer vision~\cite{shohan_rots}, inverse kinematics and trajectory planning~\cite[Chaps.~1,4]{siciliano2009modelling}, algebraic geometry~\cite[Chap.~17]{harris1992algebraic}, training neural networks~\cite{top_nns,lezcano2021geometric}, and risk minimization~\cite{Trager2020Pure,bach2021gradient,gradflow_NNs1}. The following are two concrete examples that illustrate the above two scenarios.

For an example of \textbf{scenario (a)}, consider minimizing a cost $f$ over the set $\calX=\Rmnlr$ of all $m\times n$ matrices of rank at most $r$. Unfortunately, standard algorithms running on~\eqref{eq:P} may converge to a non-stationary point because of the nonsmooth geometry of $\calX$~\cite{levin2021finding,olikier2022apocalypse}. Instead of trying to solve~\eqref{eq:P} directly, it is common to parametrize $\calX$ by the linear space $\calM=\RR^{m\times r}\times\RR^{n\times r}$ using the rank factorization $\varphi(L,R)=LR^\top$, and to solve~\eqref{eq:Q} instead. The resulting problem~\eqref{eq:Q} requires minimizing a smooth cost function over a linear space; there are several algorithms that converge to a second-order stationary point for such problems. Furthermore, any second-order stationary point for~\eqref{eq:Q} maps under $\varphi$ to a stationary point for~\eqref{eq:P} by~\cite[Thm.~1]{ha2020criticallowrank}. Thus, parametrization of $\calX$ by $\varphi$ gives us an algorithm converging to a stationary point for~\eqref{eq:P} by running standard algorithms on~\eqref{eq:Q}, even though similarly reasonable algorithms may fail to produce a stationary point when applied directly to~\eqref{eq:P}.

For an illustration of \textbf{scenario (b)}, consider finding the smallest eigenvalue of a $d\times d$ symmetric matrix $A$, which can be written in the form~\eqref{eq:Q} with $\calM$ the unit sphere in $\RR^d$ and $g(y)=y^\top Ay$. This problem is not convex, hence it could have bad local minima. 
Here is one way to reason that it does not (as is well known).
If $\lambda\in\RR^d$ denotes the vector of eigenvalues of $A$ and $U\in \mathrm{O}(d)$ is an orthogonal matrix of eigenvectors satisfying $A=U\mathrm{diag}(\lambda)U^\top$ (both of which are unknown), define $\varphi(y)=\mathrm{diag}(U^\top yy^\top U)\in\RR^d$ and $f(x)=\lambda^\top x$. 
It is easy to check that $g = f\circ\varphi$, and that $\calX=\varphi(\calM)$ is the standard simplex in $\RR^d$. 
The resulting problem~\eqref{eq:P} is convex in this case, hence each of its stationary points is a global minimum.
A corollary of the theory we develop in this paper is that any second-order stationary point for~\eqref{eq:Q} with $\varphi$ as above maps to a stationary point for~\eqref{eq:P}, for any cost function $f$---see Example~\ref{ex:eigenval_example}. 
Thus, we recover the well-known fact that any second-order stationary point for the eigenvalue problem~\eqref{eq:Q} is globally optimal. 
Even though the problem~\eqref{eq:P} cannot be solved directly in this case because $f$ and $\varphi$ are unknown, their mere existence can be used to show that the nonconvexity of~\eqref{eq:Q} is ``benign''. From this perspective, problem~\eqref{eq:P} reveals hidden convexity in problem~\eqref{eq:Q}. 
This hidden convexity is present more generally in lifts arising from Kostant's convexity theorem, extending this example to optimization of certain linear functions over certain Lie group orbits~\cite{leake2021optimization}. 


We state our main definitions and results relating the landscapes of~\eqref{eq:P} and~\eqref{eq:Q} in general, and instantiate these results on a number of specific lifts arising in the literature, in Section~\ref{sec:main_results}. Table~\ref{tab:mans} collects the notations and definitions for several sets used throughout the paper. 
\ifthenelse{\boolean{shortver}}{}
    {
Other notation and basic definitions are introduced in the text as needed and are collected for the reader's convenience in Appendix~\ref{apdx:notation}. 
}
\begin{table}[h]
    \centering
    \begin{tabular}{@{}ll@{}}
        \toprule
        Bounded-rank matrices & $\Rmnlr = \{X\in\RR^{m\times n}:\rank(X)\leq r\}$\\ \midrule
        Fixed-rank matrices & $\RR_{=r}^{m\times n} = \{X\in\RR^{m\times n}:\rank(X) = r\}$ \\ \midrule
        Symmetric matrices & $\mbb S^m = \{X\in\RR^{m\times m}:X=X^\top\}$\\ \midrule
        Nonnegative orthant & $\RR^m_{\geq0} = \{x\in\RR^m:x_i\geq 0 \textrm{ for all } i\}$ \\ \midrule
        Positive orthant & $\RR^m_{>0} = \{x\in\RR^m:x_i > 0 \textrm{ for all } i\}$ \\ \midrule
        Positive-semidefinite (PSD) matrices & $\mbb S^m_{\succeq0} = \{X\in\mbb S^m:X\succeq0\}$ \\ \midrule
        Positive-definite (PD) matrices & $\mbb S^m_{\succ0} = \{X\in\mbb S^m:X\succ0\}$ \\ \midrule
        Standard simplex & $\Delta^{n-1} = \left\{x\in\RR^n_{\geq 0} : \sum_{i=1}^nx_i=1\right\}$ \\ \midrule
        Stiefel matrices & $\mathrm{St}(m,r) = \{X\in\RR^{m\times r}:X^\top X=I_r\}$\\ \midrule
        Unit sphere & $\mathrm{S}^{m-1} = \mathrm{St}(m,1)$ \\ \midrule
        Orthogonal matrices & $\mathrm{O}(m) = \mathrm{St}(m,m)$ \\ \midrule
        Special orthogonal matrices & $\mathrm{SO}(m) = \{X\in\mathrm{O}(m):\det(X)=1\}$ \\ \midrule
        Invertible matrices & $\mathrm{GL}(m) = \{X\in\RR^{m\times m}:\det(X)\neq 0\}$\\ \midrule
        Grassmannian & $\mathrm{Gr}(m,r) = \{\mc S\subseteq\RR^m:\mc S \textrm{ is a linear subspace}, \dim\mc S = r\}$ \\ \bottomrule
    \end{tabular}
    \caption{Sets used frequently in the paper. Other sets that appear include low rank tensors and functions representable by fixed neural network architectures.}
    \label{tab:mans}
\end{table}

\section{Lifts and their properties}\label{sec:main_results}
We call the parametrization in~\eqref{eq:Q} a (smooth) \emph{lift} of $\calX$: 
\begin{definition}\label{def:lift}
A \emph{smooth lift} of $\calX\subseteq\calE$ is a smooth manifold $\calM$ together with a smooth map $\varphi\colon\calM\to\calE$ such that $\varphi(\calM)=\calX$. 
\end{definition}
As the two scenarios in Section~\ref{sec:intro} illustrate, understanding when lifts map desirable points for~\eqref{eq:Q} to desirable points for~\eqref{eq:P} yields guarantees for algorithms running on~\eqref{eq:Q}. 
Here desirable points might be minimizers (global or local) and stationary points (of first, second, or higher order).
The relation between these two sets of desirable points has been studied for various specific lifts and cost functions.
In this paper, we study this relation in general and answer the following question:
\begin{quote}
    \textit{Which lifts have the property that desirable points of~\eqref{eq:Q} map to desirable points of~\eqref{eq:P}, for \emph{\textbf{all}} cost functions $f$?}
\end{quote}
Surprisingly, we find that many lifts arising in practice satisfy such properties, yielding guarantees for algorithms running on~\eqref{eq:Q} that are independent of the particular cost function involved, and only depend on the geometry of the lift. We further show that whenever a lift does not preserve desirable points for all cost functions, then it fails to do so already for quite simple costs. 
In this case our results identify obstructions to proofs of guarantees for algorithms, which must then exploit the structure of the particular cost function at hand. 

To begin answering the above question, we note that global minima of~\eqref{eq:Q} always map under $\varphi$ to global minima of~\eqref{eq:P}, for all cost functions $f$. This holds simply because $\varphi(\calM) = \calX$, see Proposition~\ref{prop:global_min_equiv}.
Global minima are hard to find in general, so we study other types of desirable points such as local minima and stationary points. In contrast to global minima, these types of desirable points are not guaranteed to map\footnote{The converse question is simple: preimages of local minima are local minima by continuity of $\varphi$ (Proposition~\ref{prop:preim_of_loc_is_loc}), and preimages of stationary points are stationary by differentiability of $\varphi$ (Proposition~\ref{prop:preimage_of_stationary_is_critical}).} to each other under smooth lifts. In fact, it is possible for a local minimum of~\eqref{eq:Q} to map under $\varphi$ to a non-stationary point for~\eqref{eq:P}, see Example~\ref{ex:nodal_cubic_bad}. 
Thus, we define the following properties of smooth lifts that, when satisfied, yield a connection between desirable points for~\eqref{eq:Q} and those for~\eqref{eq:P}.
\begin{definition}[Desirable properties of lifts]\label{def:desirable_lifts}
    Suppose $\varphi\colon\calM\to\calX$ is a smooth lift.
    \begin{enumerate}[(a)]
        \item The lift $\varphi$ satisfies the {\localimplieslocal} property at $y \in \calM$ if, for all continuous $f \colon \calX \to \reals$, if $y$ is a local minimum for~\eqref{eq:Q} then $x = \varphi(y)$ is a local minimum for~\eqref{eq:P}.
	    We say $\varphi$ satisfies the {\localimplieslocal} property if it does so at all $y \in \calM$.
	    
	    \item The lift $\varphi$ satisfies the {\kimpliesone} property at $y$ for $k=1,2$ if for all $k$-times differentiable $f \colon \calX \to \reals$, if $y$ is a $k$th-order stationary point (``$k$-critical'' for short) for~\eqref{eq:Q} then $x = \varphi(y)$ is stationary for~\eqref{eq:P}.
	    We say $\varphi$ satisfies the {\kimpliesone} property if it does so at all $y \in \calM$.
    \end{enumerate}
\end{definition}
The precise definitions of each type of desirable points above is given in Section~\ref{sec:lifts}.
We fully characterize these properties and explain how to check them on specific examples. We then apply our results to study lifts arising in applications ranging from low-rank matrices and tensors to neural networks.

Note that {\twoimpliesone} at $y$ implies {\oneimpliesone} at $y$ since 2-critical points are 1-critical, but no other implication between the different properties holds in general---see Remark~\ref{rmk:rels_between_props}.
We also mention that $C^{\infty}$ smoothness is not necessary for the above properties or for their characterizations. For example, it suffices for the manifold $\calM$ and the lift $\varphi$ to be of class $C^k$ for the definition of {\kimpliesone} and its characterization to apply. For {\localimplieslocal} it suffices for $\calM$ to be a topological space satisfying certain properties (see Appendix~\ref{apdx:liftsopen}) and for $\varphi$ to be continuous.

Our characterizations of {\localimplieslocal} and {\oneimpliesone} are easy to state as follows.
\begin{theorem} \label{thm:localimplocalcharact}
	The lift $\varphi \colon \calM \to \calX$ satisfies {\localimplieslocal} at $y\in\calM$ if and only if it is open at $y$.
	If $\varphi$ does not satisfy {\localimplieslocal} at $y$, there is a \emph{smooth} cost $f$ such that $y$ is a local minimum for~\eqref{eq:Q} but $\varphi(y)$ is not a local minimum for~\eqref{eq:P}.
\end{theorem}
By definition, the map $\varphi$ is open at $y\in\calM$ if it maps neighborhoods of $y$ (that is, sets containing $y$ in their interior) to neighborhoods of $\varphi(y)$ (in the subspace topology on $\calX$ from $\calE$)---a purely topological property.
Proving that openness is sufficient for {\localimplieslocal} is easy.
Proof of its necessity requires substantial work, deferred to Appendix~\ref{apdx:liftsopen}.
Our proof in the appendix provides the result in a more general, topological setting without using smoothness.
It also provides (possibly new) conditions which are equivalent to openness and may be easier to check for some lifts. 

Our characterization for {\oneimpliesone} involves the image of the differential of the lift map $\varphi$, and is proved in Section~\ref{sec:1imp1}.
\begin{theorem}\label{thm:oneimpliesone_char}
The lift $\varphi\colon\calM\to\calX$ satisfies {\oneimpliesone} at $y\in\calM$ if and only if $\im\D\varphi(y) = \T_x\calX$, where $x=\varphi(y)$. If $\varphi$ does not satisfy {\oneimpliesone}, there is a \emph{linear} cost $f$ such that $y$ is 1-critical for~\eqref{eq:Q} but $\varphi(y)$ is not stationary for~\eqref{eq:P}.
\end{theorem}
Here $\varphi$ is viewed as a smooth map between smooth manifolds $\calM\to\calE$, and its differential $\D\varphi(y)$ maps the tangent space $\T_y\calM$ to (in general, a subset of) the tangent \emph{cone} $\T_x\calX$, see Definition~\ref{def:tangentcone}. 
Since $\D\varphi(y)$ is a linear map and $\T_y\calM$ is a linear space, {\oneimpliesone} is rarely satisfied: unless all tangent cones of $\calX$ are linear subspaces, for \emph{every} smooth lift $\varphi$, there exists a (linear) $f$ such that some stationary point for~\eqref{eq:Q} does not map to a stationary point for~\eqref{eq:P}.

Our characterization for {\twoimpliesone} is more complicated, involving the second derivative of $\varphi$ as well. We state an equivalent condition for {\twoimpliesone}, as well as sufficient and necessary conditions for it that are easier to check in some applications, in Theorem~\ref{thm:2implies1_chain}. If {\twoimpliesone} fails at $y$, we show in Corollary~\ref{cor:2imp1_violating_cost} that there exists a convex quadratic cost $f$ such that $y$ is 2-critical for~\eqref{eq:Q} but $\varphi(y)$ is not stationary for~\eqref{eq:P}.
Note also that if {\oneimpliesone} holds at $y$ then so does {\twoimpliesone} by definition.

Understanding stationarity on $\calX$ requires knowledge of its tangent cones. These can be hard to characterize. We show that it is sometimes possible to obtain an explicit expression for the tangent cones simultaneously with proving {\oneimpliesone} and {\twoimpliesone} for some lift of $\calX$, see Sections~\ref{sec:low_rk_psd} and~\ref{sec:fiber_prod_lifts}. This is somewhat surprising since the tangent cones to $\calX$ are defined independently of any lift.
	
Given a set $\calX$, it is also natural to seek \emph{constructions} of a smooth lift $\varphi\colon\calM\to\calX$ satisfying desirable properties. We give a systematic construction of a map $\varphi\colon\calM\to\calX$ in Section~\ref{sec:fiber_prod_lifts} which maps a set $\calM$ surjectively onto $\calX$. When the set $\calM$ constructed in this way is a smooth manifold, we obtain a smooth lift and give conditions under which $\varphi$ satisfies each of the above properties. 

We now proceed to apply our results to study various lifts arising in the literature. 
\subsection{The sphere-to-simplex Hadamard lift} \label{sec:sphere_to_simplex}
There is growing interest in optimizing over the probability simplex $\calX=\Delta^{n-1}\subseteq \calE=\RR^n$ by lifting it to the sphere via the Hadamard lift
\begin{equation}
    \calM = \mathrm{S}^{n-1}, \qquad \varphi(y) = y \odot y,
    \label{eq:sphere_to_simplex}
    \tag{Had}
\end{equation}
where $\odot$ denotes the entrywise (Hadamard) product. Using this lift leads to fast algorithms for high-dimensional problems~\eqref{eq:Q}, see~\cite{NEURIPS2019_5cf21ce3,li2021simplex,chok2023convex}. This is also essentially the lift that appears in the eigenvalue example (scenario (b)) in Section~\ref{sec:intro}, see Example~\ref{ex:eigenval_example}.
This lift is particularly natural in applications involving probabilities since the push-forward under $\varphi$ of the standard metric on the sphere is the Fisher-Rao metric on the simplex~\cite[Prop.~2.1]{ay2017information}. 

We can characterize precisely where each of our desirable properties holds.
\begin{proposition}\label{prop:sphere_to_simplex}
    The lift~\eqref{eq:sphere_to_simplex} satisfies {\localimplieslocal} everywhere, {\oneimpliesone} at $y$ if and only if $y_i\neq 0$ for all $i$ (i.e., at preimages of the relative interior of the simplex), and {\twoimpliesone} everywhere.
\end{proposition}
We prove this proposition in Corollary~\ref{cor:fiber_prod_exs_revisit} by showing that the lift~\eqref{eq:sphere_to_simplex} is a special case of our construction of lifts in Section~\ref{sec:fiber_prod_lifts} and can be analyzed using the general results we prove there. 

The relation between desirable points for~\eqref{eq:Q} and for~\eqref{eq:P} have been previously studied in~\cite{li2021simplex}, where the authors show that 2-critical points for~\eqref{eq:Q} map to 2nd-order KKT points for~\eqref{eq:P}, viewed as a nonlinear program, for any twice-differentiable cost. 
This is a strengthening of {\twoimpliesone}. 
The authors of~\cite{ding2023squared} prove similar relations between first- and second-order optimality conditions for problems~\eqref{eq:P} over $\RR^n_{\geq0}$, and for their lifts~\eqref{eq:Q} to $\RR^n$ via the entrywise-squaring lift in~\eqref{eq:sphere_to_simplex}.

The Hadamard lift also induces a lift of the set of column-stochastic matrices $\calX = \{X\in\RR^{n\times m}_{\geq0}: X^\top\mathbbm{1}_n=\mathbbm{1}_n\}$ to the product of spheres (called the oblique manifold~\cite[\S7.2]{optimOnMans}) via
\begin{equation}\label{eq:HadProd}\tag{HadProd}
    \calM = \left(\mathrm{S}^{n-1}\right)^m,\qquad \varphi(y_1,\ldots,y_m) = [y_1\odot y_1,\ldots,y_m\odot y_m].
\end{equation}
Here $\mathbbm{1}_n$ is the all-1's vector of length $n$ and $[x_1,\ldots,x_m]$ denotes horizontal concatenation of $m$ vectors $x_i$ of length $n$ to form an $n\times m$ matrix.
By studying products of lifts in Proposition~\ref{prop:prod_lifts}, we characterize the properties of~\eqref{eq:HadProd} in Example~\ref{ex:stochastic_mats} and obtain the following result.
\begin{proposition}\label{prop:stochastic_mats}
    The lift~\eqref{eq:HadProd} satisfies {\localimplieslocal} everywhere, {\oneimpliesone} at $(y_1,\ldots,y_k)$ if and only if $(y_i)_j\neq 0$ for all $i,j$, and {\twoimpliesone} everywhere.
\end{proposition}
Optimization over stochastic matrices has been applied to nonnegative matrix factorization~\cite{ahmed2020nonnegative}. Such optimization also arises in information theory~\cite{ben1988role}, where stochastic matrices represent transition probabilities of channels. 

\subsection{Smooth semidefinite programs via Burer--Monteiro} \label{sec:smth_sdps}
Consider the domain $\calX$ of a rank-constrained semidefinite program (SDP),
\begin{align*}
    \calX = \left\{X\in\mbb S_{\succeq0}^n: \rank(X)\leq r,\ \langle A_i,X\rangle=b_i \textrm{ for } i=1,\ldots,m\right\}\subseteq \calE=\mbb S^n, 
\end{align*}
where $\inner{U}{V}=\mathrm{Tr}(U^\top V)$ is the (Frobenius) inner product on $\calE$.
The Burer--Monteiro approach~\cite{burer2003nonlinear} to optimizing over $\calX$ consists of optimizing over the following parametrization instead:
\begin{equation}
    \calM = \{R\in \RR^{n\times r}:h_i(R)\coloneqq\langle A_iR,R\rangle - b_i = 0 \textrm{ for } i=1,\ldots,m\},\qquad \varphi(R)=RR^\top. 
    \label{eq:BM_lift}
    \tag{BM}
\end{equation}
Burer and Monteiro prove in~\cite[Prop.~2.3]{burer2005local} that local minima for~\eqref{eq:Q} map under $\varphi$ to local minima for~\eqref{eq:P} for linear costs $f$. 

Under some conditions on the $A_i, b_i$ (which are satisfied generically~\cite[Prop.~1]{cifuentes2021burer} as well as for several applications of interest~\cite{BM_det_guarantees}), the constraints $h_i(R) = 0$ constitute (constant-rank) local defining functions (in the sense of~\cite[Thm.~5.12]{lee_smooth}) for $\calM$, which is then an embedded submanifold of $\reals^{n \times r}$. In that case, $\calM$ and $\varphi$ constitute a smooth lift of $\calX$.
In~\cite{BM_det_guarantees}, assuming $f$ is linear (as is typical for SDPs), the authors use the assumption that the $h_i$ are local defining functions to prove that (in our terminology) rank-deficient 2-critical points for~\eqref{eq:Q} map under $\varphi$ to stationary points for~\eqref{eq:P}. This was also shown for nonlinear $f$ in~\cite{BM_JBAS}, though under more restrictive conditions on $A_i, b_i$ (e.g., $A_iA_j = 0$ for $i \neq j$). In all cases, these results allow to capture benign non-convexity when $f$ is convex, as then stationary points for~\eqref{eq:P} are global minima.

Using our framework, we can generalize these results to any (twice-differentiable) costs and remove the restrictions on the rank of the 2-critical points.
\begin{proposition}\label{prop:BM_lift}
    The Burer--Monteiro lift~\eqref{eq:BM_lift} satisfies {\localimplieslocal} everywhere. If $\calM$ in~\eqref{eq:BM_lift} is a smooth manifold with (constant-rank) local defining functions $\{h_i\}_{i=1}^m$, then this lift satisfies {\oneimpliesone} at $R$ if and only if $\rank(R)=r$, and {\twoimpliesone} everywhere.
\end{proposition}
We prove this result too in Corollary~\ref{cor:fiber_prod_exs_revisit} using general properties of our lift construction in Section~\ref{sec:fiber_prod_lifts}. 
Our theory also yields explicit expressions for the tangent cones to $\calX$ in~\eqref{eq:smth_sdp_tangent_cone}, which (to our knowledge) have not previously appeared in the literature. 

\subsection{Low-rank matrices}\label{sec:lowrk}
Consider the set $\calX=\Rmnlr$ of matrices in $\calE=\Rmn$ with rank at most $r$.
We study several natural lifts of this real algebraic variety.
The first one we study is based on the rank factorization of a matrix:
\begin{equation}
    \calM = \RR^{m\times r}\times\RR^{n\times r},\qquad \varphi(L,R) = LR^\top.
    \label{eq:LR_lift}
    \tag{LR}
\end{equation}
The authors of~\cite{ha2020criticallowrank} showed (in our terminology) that {\oneimpliesone} does not hold everywhere, but {\twoimpliesone} does.
We further proved in~\cite[Prop.~2.30]{eitan_thesis} that this lift does not satisfy {\localimplieslocal} everywhere either. 
We strengthen these results here using our unified framework, by characterizing precisely where each of these properties hold. The proof of the following proposition is given in Section~\ref{sec:pf_of_LR}.
\begin{proposition}\label{prop:LR_lift}
The lift of $\calX=\Rmnlr$ given by~\eqref{eq:LR_lift} satisfies:
\begin{itemize}
    \item {\localimplieslocal} at $(L,R)$ if and only if $\rank(L)=\rank(R)=\rank(LR^\top)$,
    \item {\oneimpliesone} at $(L,R)$ if and only if $\rank(L)=\rank(R)=r$,
    \item and {\twoimpliesone} everywhere on $\calM$~\cite{ha2020criticallowrank}.
\end{itemize}
\end{proposition}
The second lift we study for $\Rmnlr$ is the desingularization lift introduced in~\cite{desingular}. It is given by
\begin{equation}
    \calM = \{(X,\mc S)\in\RR^{m\times n}\times \mathrm{Gr}(n,n-r): \mc S\subseteq \ker X\},\qquad \varphi(X,\mc S) = X.
    \label{eq:desing_lift}
    \tag{Desing}
\end{equation}
Here $\mathrm{Gr}(n,n-r)$ is the Grassmann manifold of $(n-r)$-dimensional subspaces of $\RR^n$~\cite[\S9]{optimOnMans}.
We proved in~\cite[Prop.~2.37]{eitan_thesis} that this lift too does not satisfy {\localimplieslocal}.
The following proposition parallels the one above and is proved in Section~\ref{sec:pf_of_desing}.
\begin{proposition}\label{prop:desing_lift}
The lift of $\calX=\Rmnlr$ given by~\eqref{eq:desing_lift} satisfies:
\begin{itemize}
    \item {\localimplieslocal} at $(X,\mc S)$ if and only if $\rank(X)=r$; the same is true for {\oneimpliesone}.
    \item {\twoimpliesone} everywhere on $\calM$.
\end{itemize}
\end{proposition}
A potential advantage of the desingularization lift over the matrix factorization lift is that the preimage of a matrix, $\varphi^{-1}(X)$, is compact for the former but not for the latter.

Note that both lifts~\eqref{eq:LR_lift} and~\eqref{eq:desing_lift} satisfy {\oneimpliesone} and {\localimplieslocal} at preimages of rank-$r$ matrices, but the lift~\eqref{eq:LR_lift} further satisfies {\localimplieslocal} at ``balanced'' preimages of lower-rank matrices. 
We also mention that no smooth lift of $\Rmnlr$ can satisfy {\oneimpliesone} at preimages of lower-rank matrices by Theorem~\ref{thm:oneimpliesone_char}, since the tangent cones to such matrices are not linear spaces.

In~\cite{mishra2014fixed}, the authors experiment with various SVD-type lifts for optimization over matrices of rank exactly $r$. The following proposition, proved in \ifthenelse{\boolean{shortver}}{the arxiv version of this paper~\cite[App.~D]{levin2022effectARXIV}}{Appendix~\ref{apdx:svd_lift}}, gives some of the properties of these lifts, extended to parametrize all of $\Rmnlr$. 
\begin{proposition}\label{prop:svd_saxs}
The SVD lift and its modification from~\cite{mishra2014fixed} satisfy the following.
\begin{itemize}
    \item The SVD lift of $\Rmnlr$ given by
    
    \begin{equation}
        \calM = \mathrm{St}(m,r)\times\RR^r\times\mathrm{St}(n,r),\qquad \varphi(U,\sigma,V) = U\mathrm{diag}(\sigma) V^\top,
        \label{eq:SVD_lift}
        \tag{SVD}
    \end{equation}

    satisfies {\localimplieslocal} at $(U,\sigma,V)$ if and only if $|\sigma_1|, \ldots, |\sigma_r|$ are nonzero and distinct; the same holds for {\oneimpliesone}.
    
    \item The modified SVD lift 
    \begin{equation}
        \calM = \mathrm{St}(m,r)\times\mbb S^r\times\mathrm{St}(n,r),\qquad \varphi(U,M,V) = UM V^\top,
        \label{eq:MSVD_lift}
        \tag{MSVD}
    \end{equation}
    satisfies {\localimplieslocal} at $(U,M,V)$ if and only if the eigenvalues of $M$ satisfy $\lambda_i(M)+\lambda_j(M)\neq 0$ for all $i,j$; the same holds for {\oneimpliesone}.
\end{itemize}
\end{proposition}
In~\cite[\S6.3]{mishra2014fixed}, the authors observed that Riemannian gradient descent running on~\eqref{eq:Q} gets stuck in a suboptimal point for a certain matrix completion problem using~\eqref{eq:SVD_lift} but not using~\eqref{eq:MSVD_lift}. We can use Proposition~\ref{prop:svd_saxs} to understand their observation. Their algorithm only generates iterates with strictly positive diagonals $\sigma$ in~\eqref{eq:SVD_lift} and strictly positive-definite middle factors $M$ in~\eqref{eq:MSVD_lift}, and can only converge to such points. Proposition~\ref{prop:svd_saxs} shows that~\eqref{eq:MSVD_lift} satisfies {\localimplieslocal} and {\oneimpliesone} in that region, while~\eqref{eq:SVD_lift} does not.

\subsection{Low-rank tensors}\label{sec:lowrk_tensors}
Tensor factorization formats correspond to lifts mapping factors to low-rank tensors, for various notions of tensor rank.
For example, the canonical polyadic (CP) decomposition of rank at most 1 corresponds to the lift of $\calX = \{X\in\RR^{n_1\times\cdots\times n_d}:\textrm{CP-rank}(X)\leq 1\}$~\cite{tensor_review} via tensor product $\otimes$ as:
\begin{align*}
    \calM = \RR^{n_1}\times\cdots\times\RR^{n_d},\qquad \varphi(v_1,\ldots,v_d) = v_1\otimes\cdots\otimes v_d.
\end{align*}
Other examples include CP decompositions of higher rank, Tucker and Tensor Train (TT) decompositions, and more generally tensor networks~\cite{tensor_review,cichocki2014tensor}. 
Surprisingly, none of these lifts satisfy {\twoimpliesone}: we derive this from more general obstructions to {\twoimpliesone} for multilinear lifts in Section~\ref{sec:multilinear_lifts}.
Here is one take-away: any stationarity guarantees for algorithms targeting second-order critical points over the factors in a tensor decomposition must exploit the structure of the cost function.

\subsection{Neural networks}\label{sec:NNs}
Training neural networks is done via lifts.
Indeed, here $\calM$ is the manifold of weights and biases of a fixed neural network architecture (typically a linear space; sometimes a product of spheres if normalization constraints are present).
The lift $\varphi$ maps a choice of weights and biases to \emph{the function} given by the corresponding neural network.
The image $\calX = \varphi(\calM)$ of this lift is the set of functions that can be represented by the architecture, viewed as a subset of some linear space $\calE$ of functions (e.g., an $L^p$ space\footnote{Even though these spaces are not finite-dimensional, our results still apply, see Appendix~\ref{apdx:liftsopen}.}).

The authors of~\cite{top_nns} show that such $\varphi$ is not open for any choice of (nonconstant) Lipschitz continuous activation functions.
Our Theorem~\ref{thm:localimplocalcharact} then implies that {\localimplieslocal} fails for all neural network lifts used in practice. 
Consequently, training such a neural network by optimizing over its weights and biases might yield a spurious local minimum that does not parametrize a local minimum in function space. In that case, a different parametrization of the same function might not be a local minimum for~\eqref{eq:Q}.

When the neural network architecture is linear with three or more layers, the corresponding lift is multilinear, hence does not satisfy {\twoimpliesone} by the same general obstructions from Section~\ref{sec:multilinear_lifts} we use for tensor decompositions. Similarly to the tensor case, this implies that proofs of guarantees for training algorithms must exploit the structure in specific cost functions (the loss).
Additional study of lifts defined by linear neural networks was done in~\cite{Trager2020Pure, kohn2021geometry}, where the authors characterize (in our terminology) {\oneimpliesone} for lifts defined by linear and linear convolutional architectures.

\subsection{Submersions and higher order stationary points}\label{sec:submersions}
All the sets $\calX$ we consider in this paper contain dense smooth submanifolds. Moreover, even though lifts of such sets $\calX$ do not satisfy {\oneimpliesone} everywhere on the lift, they do so at preimages of points on this submanifold, allowing us to prove much stronger guarantees.
More precisely, we define the following subset of $\calX$.
\begin{definition}\label{def:smth_locus}
    A point $x\in\calX$ is \emph{smooth} if there is an open neighborhood $U\subseteq \calE$ containing $x$ such that $U\cap\calX$ is a smooth embedded submanifold of $\calE$. It is called \emph{nonsmooth} or \emph{singular} otherwise.
    The \emph{smooth locus} of $\calX$, denoted $\calX^{\mathrm{smth}}$, is the set of smooth points of $\calX$.
\end{definition}
For all constraint sets in practical optimization problems we are aware of, $\calX^{\mathrm{smth}}$ is itself a smooth embedded submanifold of $\calE$. In general, it is a union of smooth embedded submanifolds, though possibly of different dimensions.
For example, if $\calX=\Rmnlr$ then $\calX^{\mathrm{smth}} = \RR^{m\times n}_{=r}$ and if $\calX=\Delta^{n-1}$ then $\calX^{\mathrm{smth}}=\Delta^{n-1}_{>0}$ consisting of strictly positive simplex vectors.
All the lifts we consider for these sets in Sections~\ref{sec:sphere_to_simplex} and~\ref{sec:lowrk} indeed satisfy {\oneimpliesone} on the preimages of $\calX^{\mathrm{smth}}$ (though that is not always the case).

If the lift satisfies {\oneimpliesone} at preimages of smooth points, then it is a submersion there and hence preserves not only local minima, but also stationary points of all orders. The following proposition, proved in Section~\ref{sec:1imp1}, formalizes this.
\begin{proposition}\label{prop:1imp1_implies_locimploc_at_smooth}
Let $y\in\varphi^{-1}(\calX^{\mathrm{smth}})\subseteq\calM$. If $\varphi$ satisfies {\oneimpliesone} at $y$, then it also satisfies {\localimplieslocal} and {\kimpliesk} for all $k\geq 1$ at $y$. 
\end{proposition}
Here {\kimpliesk} is defined analogously to Definition~\ref{def:desirable_lifts}, where $k$th-order stationarity (or ``$k$-criticality'' for short) of $x\in\calX^{\mathrm{smth}}$ is defined using curves similarly to 1- and 2-criticality~\cite[\S3.1.1]{cartis2018second}. This property can be used in proofs of benign nonconvexity. 

\begin{remark}[Relations between lift properties]\label{rmk:rels_between_props}
    Aside from Proposition~\ref{prop:1imp1_implies_locimploc_at_smooth}, the only relation between the three properties in Definition~\ref{def:desirable_lifts} is that {\oneimpliesone} at $y$ implies {\twoimpliesone} at $y$ (since 2-critical points are 1-critical).
    None of the other possible implications hold in general:
    The desingularization lift~\eqref{eq:desing_lift} shows that {\twoimpliesone} at $y$ implies neither {\oneimpliesone} nor {\localimplieslocal} at $y$ in general.
    The example $\varphi(x) = x^3$ viewed as a lift from $\calM=\RR$ to $\calX=\RR$ satisfies {\localimplieslocal} at the origin but neither {\twoimpliesone} nor {\oneimpliesone}, hence {\localimplieslocal} does not imply the other two properties.
    Finally, the standard parametrization of the cochleoid curve~\cite{snail_curve} satisfies {\oneimpliesone} but not {\localimplieslocal} at all preimages of the origin, hence {\oneimpliesone} does not imply {\localimplieslocal}.
\end{remark}

Submersions between smooth manifolds, including quotients by group actions~\cite[\S9.2]{optimOnMans} and several lifts arising in practice (Example~\ref{ex:submersions}), satisfy {\localimplieslocal} and {\kimpliesk} for all $k\geq 1$~\cite[Prop.~9.6]{optimOnMans}. Therefore, a lift $\varphi$ composed with a submersion $\psi$ as $\varphi \circ \psi$ 
inherits the properties of $\varphi$. We study such compositions of lifts in Section~\ref{sec:composition_of_lifts}, and apply our results to a composition used in the robotics and computer vision literature~\cite{shohan_rots} in Example~\ref{ex:shohan_rots}.

\section{Characterizations of lifts}\label{sec:lifts}
In this section, we relate the landscapes of~\eqref{eq:P} and~\eqref{eq:Q} and prove the characterizations of our lift properties stated in Section~\ref{sec:main_results}. To this end, we formally define the different types of desirable points we consider.
We first define the  (\emph{contingent} or \emph{Bouligand}) tangent cone\footnote{In this paper, a cone is a set $K$ such that $x\in K\implies \alpha x\in K$ for all $\alpha>0$.}~\cite[\S2.7]{clarke2008nonsmooth}.
\begin{definition} \label{def:tangentcone}
	The \emph{tangent cone} to $\calX$ at $x \in \calX$ is the set
	\begin{align*}
		\T_x\calX & = \left\{ v = \lim_{i \to \infty} \frac{x_i - x}{\tau_i} : x_i \in \calX, \tau_i > 0 \textrm{ for all } i, \tau_i \to 0 \right\}\subseteq \calE.
	\end{align*}
    This is a closed (not necessarily convex) cone~\cite[Lem.~3.12]{nonlin_optim}.
\end{definition}
In particular, if $\gamma$ is a differentiable curve in $\calX$ with $\gamma(0) = x$, then $\gamma'(0)\in\T_x\calX$.
If $x$ is a smooth point of $\calX$ (Definition~\ref{def:smth_locus}), then $\T_x\calX$ is the usual tangent space to $\calX$ at $x$~\cite[Ex.~6.8]{rockafellar2009variational}.

\begin{definition}[Desirable points for~\eqref{eq:P}]\label{def:downstairs_basic}
A point $x\in\calX$ is a
    \begin{enumerate}[(a)]
        \item \emph{global minimum} for~\eqref{eq:P} if $f(x)=\min_{x'\in\calX}f(x')$.
        \item \emph{local minimum} for~\eqref{eq:P} if there is a neighborhood $U\subseteq\calX$ of $x$ such that $f(x)=\min_{x'\in U}f(x')$.
        \item \emph{(first-order) stationary point} for~\eqref{eq:P} if $\D f(x)[v]\geq 0$ for all $v\in \T_x\calX$, or equivalently, if $\nabla f(x)$ is in the dual $(\T_x\calX)^*$ of the tangent cone.
    \end{enumerate}
\end{definition}
In words, $x$ is stationary if the cost function is non-decreasing to first order along all tangent directions at $x$.
Local minima of~\eqref{eq:P} are stationary~\cite[Thm.~3.24]{nonlin_optim}. 
The dual of a cone $K\subseteq\calE$ contained in a Euclidean space $\calE$ with inner product $\langle\cdot,\cdot\rangle$ is defined by
\begin{align*}
    K^* = \{x\in\calE: \langle x,x'\rangle\geq 0 \textrm{ for all } x'\in K\}.
\end{align*}
The equivalence in part (c) then follows since $\D f(x)[v]=\langle\nabla f(x),v\rangle$ by definition of the (Euclidean) gradient $\nabla f(x)$.
We use the following properties of dual cones throughout (see~\cite[Prop.~4.5]{deutsch2012best} for proofs):
\begin{itemize}
	\item The dual cone is always a closed convex cone.
	\item If $K_1\subseteq K_2$, then $K_2^*\subseteq K_1^*$.
	\item The bidual cone $K^{**} = (K^*)^*$ of $K$ is equal to the closure of its convex hull: $K^{**}=\overline{\mathrm{conv}}(K)$. In particular, $K^{**}\supseteq K$.
	\item If $K$ is a linear space, then its dual $K^*$ is equal to its orthogonal complement $K^{\perp}$.
\end{itemize}

Next, we define desirable points for~\eqref{eq:Q}. 
\begin{definition}[Desirable points for~\eqref{eq:Q}]\label{def:criticality_on_lift}
\begin{enumerate}[align=left]
    \item[]
    \item[(a)+(b)] Global and local minima for~\eqref{eq:Q} are defined exactly as for~\eqref{eq:P}. 
    \item[(c)] A point $y\in\calM$ is \emph{first-order stationary} (or ``1-critical'') for~\eqref{eq:Q} if for each smooth curve $c\colon \RR\to\calM$ satisfying $c(0)=y$, we have $(g\circ c)'(0)\geq 0$, or equivalently,\footnote{If $(g\circ c)'(0)>0$, let $\widetilde c(t)=c(-t)$ and note that $(g\circ\widetilde c)'(0)<0$.} $(g\circ c)'(0)=0$.
    
    \item[(d)] A point $y\in\calM$ is \emph{second-order stationary} (or ``2-critical'') for~\eqref{eq:Q} if it is 1-critical and $(g\circ c)''(0)\geq 0$ for all smooth curves $c\colon \RR\to\calM$ satisfying $c(0)=y$.
\end{enumerate}
\end{definition}
If $\calM$ is embedded in a linear space, first-order stationarity in Definition~\ref{def:criticality_on_lift}(c) coincides with Definition~\ref{def:downstairs_basic}(c) by~\cite[Ex.~6.8]{rockafellar2009variational}. 
Definition~\ref{def:criticality_on_lift} can be rephrased in terms of the Riemannian gradient and Hessian of $g$, as follows.
\begin{proposition}[{\cite[\S4.2, \S6.1]{optimOnMans}}]
A point $y\in\calM$ is $1$-critical for~\eqref{eq:Q} if and only if $\nabla g(y)=0$. It is 2-critical if and only if $\nabla g(y)=0$ and $\nabla^2g(y)\succeq0$. 
\end{proposition}

We proceed to study the connections between desirable points for~\eqref{eq:Q} and~\eqref{eq:P}. As mentioned in Section~\ref{sec:main_results}, the connection between global minima of~\eqref{eq:Q} and~\eqref{eq:P} is straightforward.
\begin{proposition}\label{prop:global_min_equiv}
	A point $y \in \calM$ is a global minimum of~\eqref{eq:Q} if and only if $x = \varphi(y)$ is a global minimum of~\eqref{eq:P}.
\end{proposition}
\begin{proof}
    Because $\varphi(\calM)=\calX$, we have $\inf_{y\in\calM}g(y)=\inf_{y\in\calM}f(\varphi(y))=\inf_{x\in\calX}f(x)=:p^*$. Therefore, $y$ is a global minimum for~\eqref{eq:Q} iff $g(y)=f(x)=p^*$ which happens iff $x$ is a global minimum for~\eqref{eq:P}. 
\end{proof}
Since computing global minima is hard, the remainder of this section is devoted to characterizing the properties in Definition~\ref{def:desirable_lifts} that yield connections between the other types of points.
\subsection{Local minima}
In this section, we investigate the relationship between the local minima of~\eqref{eq:P} and those of~\eqref{eq:Q}.
Preimages of local minima on $\calX$ are always local minima on $\calM$ merely because $\varphi$ is continuous.
\begin{proposition}\label{prop:preim_of_loc_is_loc}
	Let $x$ be a local minimum for~\eqref{eq:P}.
	Any $y \in \varphi^{-1}(x)$ is a local minimum for~\eqref{eq:Q}.
\end{proposition}
\begin{proof}
	There exists a neighborhood $U$ of $x$ in $\calX$ such that $f(x) \leq f(x')$ for all $x' \in U$.
	Since $\varphi \colon \calM \to \calX$ is continuous, the set $\calU = \varphi^{-1}(U)$ is a neighborhood of $y$ in $\calM$.
	Pick an arbitrary $y' \in \calU$: it satisfies $\varphi(y') = x'$ for some $x' \in U$.
	Hence, $g(y) = f(x) \leq f(x') = g(y')$, i.e., $y$ is a local minimum of~\eqref{eq:Q}.
\end{proof}
Unfortunately, lifting can introduce \emph{spurious} local minima, that is, local minima for~\eqref{eq:Q} that exist only because of the lift and not because they were present in~\eqref{eq:P} to begin with.
\begin{example}[Nodal cubic]\label{ex:nodal_cubic_bad}
Consider the nodal cubic 
\begin{align}
    \calX = \{x\in\RR^2:x_2^2=x_1^2(x_1+1)\},
    \label{eq:nodal_cubic}
\end{align}
and the following lift,\footnote{The curve $\calM$ is obtained by blowing up $\calX$ at the origin in the sense of algebraic geometry~\cite[Ch.~17]{harris1992algebraic}.} as depicted in Figure~\ref{fig:alpha_curve}:
\begin{align}
    \calM = \{y\in\RR^3:y_1=y_3^2-1,\ y_2=y_1y_3\},\qquad \varphi(y_1,y_2,y_3)=(y_1,y_2).
    \label{eq:node_blowup}
\end{align}
Let $f(x)=-x_1-x_2$. Then the point $y=(0,0,1)$ is a local minimum for $g=f\circ\varphi$ but $\varphi(y)=(0,0)$ is not even stationary for $f$. Indeed, we have $(1,1)\in\T_{(0,0)}\calX$ and $\D f(0,0)[(1,1)]=-2<0$.
\begin{figure}[h]
    \centering
    \includegraphics[width=.5\linewidth]{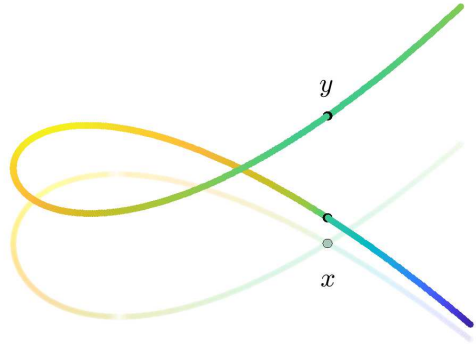}
    \caption{Nodal cubic in $\RR^2$ as the shadow of its lift in $\RR^3$, colored by the value of the function $f(x)=-x_1-x_2$. The highlighted points are $x=(0,0)$ (not stationary for~\eqref{eq:P}), and its two preimages on the lift, including $y=(0,0,1)$ (a spurious local minimum for~\eqref{eq:Q}).}
    \label{fig:alpha_curve}
\end{figure}
\end{example}
To ensure that a lift does not introduce spurious local minima, we need to verify that it satisfies the {\localimplieslocal} property (Definition~\ref{def:desirable_lifts}(a)).
We proceed to prove the easy direction of Theorem~\ref{thm:localimplocalcharact} stating that openness implies {\localimplieslocal}. The converse is more involved and is deferred to Appendix~\ref{apdx:liftsopen}.

\begin{proof}[Proof of Theorem~\ref{thm:localimplocalcharact}]
	Assume $\varphi$ is open at $y$, and that $y$ is a local minimum for~\eqref{eq:Q}.
	Then, there exists a neighborhood $\calU$ of $y$ on $\calM$ such that $g(y) \leq g(y')$ for all $y' \in \calU$.
	The set $U = \varphi(\calU)$ is a neighborhood of $x = \varphi(y)$ in $\calX$ by openness of $\varphi$ at $y$.
	Moreover, each $x' \in U$ is of the form $x' = \varphi(y')$ for some $y' \in \calU$.
	Therefore, $f(x) = g(y) \leq g(y') = f(x')$ for all $x' \in U$, that is, $x$ is a local minimum of~\eqref{eq:P}.
	For the converse, see Theorem~\ref{thm:openmapsonmanifolds}.
\end{proof}
Not all lifts of interest are open. In particular, all lifts of low-rank matrices in Section~\ref{sec:lowrk} as well as the neural network lifts in Section~\ref{sec:NNs} fail to be open. 


\begin{remark}\label{rmk:ESLP}
In Appendix~\ref{apdx:liftsopen}, we introduce an equivalent condition for openness of $\varphi$ at $y$ that we call the Subsequence Lifting Property (SLP), see Definition~\ref{def:preserv_loc_minima}(3); we find that it is sometimes easier to check. 
For example, Burer and Monteiro prove that the lift~\eqref{eq:BM_lift} satisfies {\localimplieslocal} in~\cite[Prop.~2.3]{burer2005local} by (in our terminology) proving SLP holds. 
\end{remark}
We note in passing that all continuous, surjective, open maps are quotient maps, hence if $\varphi$ is a smooth lift of $\calX$ satisfying {\localimplieslocal} then it is a quotient map from $\calM$ to $\calX$.

\subsection{Stationary points}\label{sec:kimpliesone}
In this section, we investigate the relationship between the first- and second-order stationary points for~\eqref{eq:Q} and (first-order) stationary points for~\eqref{eq:P}. To that end, we begin by relating the (Riemannian) gradient and Hessian of $g=f\circ\varphi$ to the (Euclidean) counterparts of $f$. 
This relation depends on the first and second derivatives of the lift $\varphi$.
\begin{definition}\label{def:LQ_maps}
    Let $\varphi\colon\calM\to\calX$ be a smooth lift and fix $y\in\calM$. For each $v\in \T_y\calM$, choose a curve $c_v$ on $\calM$ satisfying $c(0)=y$ and $c'(0)=v$. Define maps $\Lmap_y, \Qmap_y\colon\T_y\calM\to\calE$ by
    \begin{align*}
        \Lmap_y(v) & = (\varphi\circ c_v)'(0), &
        \Qmap_y(v) & = (\varphi\circ c_v)''(0).
    \end{align*}
    We write $\Lmap_y^{\varphi}$ and $\Qmap_y^{\varphi}$ when we wish to emphasize the lift.
\end{definition}
As a point of notation: $\varphi\circ c_v$ is a curve in $\calE$ hence $(\varphi\circ c_v)''$ denotes its Euclidean acceleration.
In contrast, $c_v$ is a curve on $\calM$ hence $c_v''$ denotes its Riemannian acceleration, see~\cite[\S5.8, \S8.12]{optimOnMans}.

Of course, $\Lmap_y$ is simply the differential $\D\varphi(y)$, and is therefore linear and independent of the choice of curves $c_v$. 
The map $\Qmap_y$ will play an important role in characterizing {\twoimpliesone} in Section~\ref{sec:2imp1}, where we also clarify its inconsequential dependence on the choice of curve $c_v$.
We explain how to compute $\Lmap_y$ and $\Qmap_y$ without explicitly choosing curves $c_v$ in Section~\ref{sec:Qmap_comp}. 

The gradients and Hessians of $f$ and $g = f \circ \varphi$ are neatly related as follows in terms of $\Lmap_y$.
\begin{definition}\label{def:phi_w}
    For any $w\in\calE$, define $\varphi_w\colon\calM\to\RR$ by $\varphi_w(y)=\langle w,\varphi(y)\rangle$.    
\end{definition}
\begin{lemma}\label{lem:grad_and_hess_g}
For any twice differentiable cost $f\colon\calE\to\RR$, any $y\in\calM$, and $x=\varphi(y)$, we have
\begin{align*}
    \nabla g(y) & = \Lmap_y^*(\nabla f(x)), &
    \nabla^2g(y) & = \Lmap_y^*\circ\nabla^2f(x)\circ\Lmap_y + \nabla^2\varphi_{\nabla f(x)}(y),
\end{align*}
where $\Lmap_y^*\colon\calE\to\T_y\calM$ is the adjoint of $\Lmap_y$. 
\end{lemma}
\begin{proof}
    For any $v\in\T_y\calM$, let $c_v$ be a smooth curve on $\calM$ satisfying $c_v(0)=y$, $c_v'(0)=v$ and $c_v''(0)=0$ (e.g., let $c_v$ be a geodesic).
    Let $\gamma_v=\varphi\circ c_v$: it satisfies $\gamma_v(0)=x$ and  $\gamma_v'(0)=\Lmap_y(v)$.
    Then
    \begin{align*}
        \langle\nabla g(y),v\rangle = (g\circ c_v)'(0) = (f\circ\gamma_v)'(0) = \langle\nabla f(x),\Lmap_y(v)\rangle = \langle\Lmap_y^*(\nabla f(x)),v\rangle.
    \end{align*} 
    Since this holds for all $v\in\T_y\calM$, we conclude that $\nabla g(y)=\Lmap_y^*(\nabla f(x))$. Next,
    \begin{align*}
        \langle\nabla^2g(y)[v],v\rangle = (g\circ c_v)''(0) = (f\circ\gamma_v)''(0) = \langle\nabla^2f(x)[\Lmap_y(v)],\Lmap_y(v)\rangle + \langle \nabla f(x),\gamma_v''(0)\rangle,
    \end{align*}
    where the first equality uses $c_v''(0)=0$, see~\cite[\S5.9]{optimOnMans}.
    On the other hand, with Definition~\ref{def:phi_w},
    \begin{align*}
        \langle\nabla^2\varphi_{\nabla f(x)}(y)[v],v\rangle = (\varphi_{\nabla f(x)}\circ c_v)''(0) = \left.\frac{\mathrm{d}^2}{\mathrm{d}t^2}\langle \nabla f(x),\gamma_v(t)\rangle\right|_{t=0} = \langle \nabla f(x),\gamma_v''(0)\rangle,
    \end{align*}
    hence
    \begin{align*}
        \langle\nabla^2g(y)[v],v\rangle &= \langle\nabla^2f(x)[\Lmap_y(v)],\Lmap_y(v)\rangle + \langle\nabla^2\varphi_{\nabla f(x)}(y)[v],v\rangle\\ &= \inner{\big(\Lmap_y^*\circ\nabla^2f(x)\circ\Lmap_y+\nabla^2\varphi_{\nabla f(x)}(y)\big)[v]}{v}.
    \end{align*}
    Since this holds for all $v\in\T_y\calM$ and both $\nabla^2g(y)$ and $\Lmap_y^*\circ\nabla^2f(x)\circ\Lmap_y+\nabla^2\varphi_{\nabla f(x)}(y)$ are self-adjoint linear maps on $\T_y\calM$, we conclude that they are equal.
\end{proof}
We turn to proving our characterizations of {\kimpliesone} for $k=1,2$ announced in Section~\ref{sec:main_results}.

\subsubsection{{\oneimpliesone}: Lifts preserving 1-critical points}\label{sec:1imp1}
Preimages of stationary points on $\calX$ are always 1-critical on $\calM$.
We show this after a helpful lemma. 
\begin{lemma}\label{lem:Ay1_props}
    Fix $y\in\calM$ and let $x=\varphi(y)$. Then $\im\Lmap_y\subseteq \T_x\calX$. Moreover, $y$ is 1-critical for~\eqref{eq:Q} if and only if $\nabla f(x)\in(\im\Lmap_y)^{\perp} = (\im\Lmap_y)^*$.
\end{lemma}
\begin{proof}
    The first claim follows from Definition~\ref{def:tangentcone} for the tangent cone $\T_x\calX$ and the fact that $\Lmap_y(v) = (\varphi\circ c_v)'(0)$ for a curve $c_v$ as in Definition~\ref{def:LQ_maps}. 
    For the second claim, $y$ is 1-critical for~\eqref{eq:Q} iff $\nabla g(y)=\Lmap_y^*(\nabla f(x))=0$, or equivalently, $\nabla f(x)\in\ker(\Lmap_y^*)=(\im\Lmap_y)^{\perp}$.
    %
\end{proof}
\begin{proposition}\label{prop:preimage_of_stationary_is_critical}
If $x\in\calX$ is stationary for~\eqref{eq:P}, then any $y\in\varphi^{-1}(x)$ is 1-critical for~\eqref{eq:Q}.
\end{proposition}
\begin{proof}
    If $x\in\calX$ is stationary for~\eqref{eq:P}, then $\nabla f(x)\in(\T_x\calX)^*$. Since $\T_x\calX\supseteq\im\Lmap_y$, taking duals on both sides we get that $\nabla f(x)\in(\T_x\calX)^*\subseteq(\im\Lmap_y)^{\perp}$, hence $y$ is 1-critical for~\eqref{eq:Q} by Lemma~\ref{lem:Ay1_props}.
\end{proof}

The converse to Proposition~\ref{prop:preimage_of_stationary_is_critical} is false in general.
In fact, Example~\ref{ex:nodal_cubic_bad} shows that a lift need not even map local minima to stationary points on $\calX$. 
We therefore proceed to prove Theorem~\ref{thm:oneimpliesone_char} characterizing the {\oneimpliesone} property.


\begin{proof}[Proof of Theorem~\ref{thm:oneimpliesone_char}]
    Suppose $\im\Lmap_y=\T_x\calX$, so $(\im\Lmap_y)^{\perp}=(\im\Lmap_y)^*=(\T_x\calX)^*$. If $y$ is 1-critical for~\eqref{eq:Q}, then $\nabla f(x)\in (\im\Lmap_y)^{\perp}$ by Lemma~\ref{lem:Ay1_props}. Therefore, $\nabla f(x)\in(\T_x\calX)^*$, which is the definition of $x$ being stationary for~\eqref{eq:P}. Thus, {\oneimpliesone} holds.
    
    Now suppose $\im\Lmap_y\neq \T_x\calX$.
    This implies $(\T_x\calX)^* \neq (\im\Lmap_y)^{\perp}$.
    Indeed, otherwise we would have
    \begin{align*}
        \im\Lmap_y = (\im\Lmap_y)^{\perp\perp}=\left((\im\Lmap_y)^{\perp}\right)^*=(\T_x\calX)^{**} \supseteq \T_x\calX\supseteq\im\Lmap_y,    
    \end{align*}
    which would imply $\im\Lmap_y=\T_x\calX$.
    (The right-most inclusion above is by Lemma~\ref{lem:Ay1_props}.)
    Using $\im\Lmap_y\subseteq\T_x\calX$ again, we see that $(\T_x\calX)^*\subseteq (\im\Lmap_y)^{\perp}$.
    Therefore, the above observations imply that $(\T_x\calX)^*\subsetneq (\im\Lmap_y)^{\perp}$.
    Pick $w\in(\im\Lmap_y)^{\perp}\setminus(\T_x\calX)^*$ and define $f(x')=\langle w,x'\rangle$ for $x'\in\calE$. Then $\nabla f(x)=w\in(\im\Lmap_y)^{\perp}$ so $y$ is 1-critical for~\eqref{eq:Q} by Lemma~\ref{lem:Ay1_props}, but $\nabla f(x)\notin(\T_x\calX)^*$, so $x$ is not stationary for~\eqref{eq:P}. Hence {\oneimpliesone} is not satisfied at $y$. 

    This argument also shows that if $\varphi$ does not satisfy {\oneimpliesone} at $y$, then $y$ is 1-critical for~\eqref{eq:Q} but $x$ is not stationary for~\eqref{eq:P} if and only if $\nabla f(x)\in (\im\Lmap_y)^{\perp}\setminus(\T_x\calX)^*$, showing that if {\oneimpliesone} fails at $y$ then this is witnessed by a linear cost $f$.
\end{proof}

As discussed in Section~\ref{sec:main_results}, Theorem~\ref{thm:oneimpliesone_char} implies that {\oneimpliesone} rarely holds on all of $\calM$.
Nevertheless, {\oneimpliesone} does usually hold at preimages of smooth points, that is, points around which $\calX$ is a smooth embedded submanifold of $\calE$ as in Definition~\ref{def:smth_locus}.
We now prove Proposition~\ref{prop:1imp1_implies_locimploc_at_smooth}, stating that if {\oneimpliesone} holds at such points then {\localimplieslocal} and ``$k\! \Rightarrow\! k$'' hold there as well.
\begin{proof}[Proof of Proposition~\ref{prop:1imp1_implies_locimploc_at_smooth}]
    Let $U\subseteq\calE$ be an open neighborhood of $\varphi(y)$ in $\calE$ such that $U\cap\calX$ is a smooth embedded submanifold of $\calE$. Since $\varphi(\calM)=\calX$, we have $\varphi^{-1}(U\cap\calX)=\varphi^{-1}(U)=:V$, which is open in $\calM$ by continuity of $\varphi$. Therefore, $V$ is also a smooth manifold, since it is an open subset of $\calM$, and $\varphi|_V\colon V\to U\cap\calX$ is a smooth map between smooth manifolds.     
    By Theorem~\ref{thm:oneimpliesone_char}, $\varphi$ satisfies {\oneimpliesone} at $y$ iff $\T_x\calX=\im \D\varphi(y) = \im \D(\varphi|_V)(y)$, where $\varphi|_V$ is viewed as a map $V\to \calE$. 
    Since $U\cap\calX$ is an embedded submanifold of $\calE$, the differential of $\varphi|_V$ viewed as a map $V\to\calE$ coincides with its differential viewed as a map $V\to U\cap\calX$, hence the latter is a submersion near $y$~\cite[Prop.~4.1]{lee_smooth}. 
    By~\cite[Prop.~4.28]{lee_smooth}, this implies $\varphi$ is open at $y$, hence it satisfies {\localimplieslocal} at $y$ by Theorem~\ref{thm:localimplocalcharact}.
    To see that $\varphi$ further satisfies {\kimpliesk} for all $k\geq1$, note that any curve passing through $\varphi(y)$ is the image under $\varphi$ of a curve passing through $y$~\cite[Thm.~4.26]{lee_smooth}, and apply Definition~\ref{def:criticality_on_lift} for $k=1,2$ and~\cite[Eq.~(3.11)]{cartis2018second} for $k>2$. 
\end{proof}
The converse of Proposition~\ref{prop:1imp1_implies_locimploc_at_smooth} fails.
For example, $\varphi(y) = y^3$ viewed as a map $\RR\to\RR$ satisfies {\localimplieslocal} at $y=0$ but not {\oneimpliesone} since $\Lmap_y=0$ for this $y$.
That example also shows that {\oneimpliesone} can fail at the preimage of a smooth point. Likewise, {\oneimpliesone} can hold at the preimage of a nonsmooth point, as the standard parametrization of the cochleoid curve~\cite{snail_curve} shows.
%
The only examples of lifts we know of that satisfy {\oneimpliesone} everywhere are smooth maps between smooth manifolds that are submersions.
\begin{example}[Submersions]\label{ex:submersions}
Examples of submersions in optimization, that is, lifts of the form $\varphi\colon\calM\to\calX$ where $\calX$ is an embedded submanifold of $\calE$ and $\im\D\varphi(y)=\T_{\varphi(y)}\calX$ for all $y\in\calM$, include:
\begin{itemize}
    \item Quotient maps by smooth, free, and proper Lie group actions~\cite[\S9]{optimOnMans},~\cite[\S3.4]{absil_book}, used in particular to optimize over Grassmannians by lifting to Stiefel manifolds~\cite{edelman1998geometry}.
    \item The map $\mathrm{SO}(p)\to\mathrm{St}(p,d)$ taking the first $d$ columns of a rotation matrix, which is used in the rotation averaging algorithm of the robotics paper~\cite{shohan_rots}, see Example~\ref{ex:shohan_rots} below.
    \item Deep orthogonal linear networks, mapping $\mathrm{O}(p)^n\to\mathrm{O}(p)$ by  $\varphi(Q_1,\ldots,Q_n)=Q_1\cdots Q_n$, whose properties are studied in~\cite{ablin2020deep}.
\end{itemize}
Theorem~\ref{thm:oneimpliesone_char} and Proposition~\ref{prop:1imp1_implies_locimploc_at_smooth} show that these lifts satisfy {\oneimpliesone} and {\localimplieslocal}, and hence also {\kimpliesk} for all $k\geq1$.
\end{example}

Failure of a lift to satisfy {\oneimpliesone} means that it may introduce spurious critical points. In the next section, we characterize the {\twoimpliesone} property, which allows algorithms to avoid these spurious points by using second-order information.

\subsubsection{{\twoimpliesone}: Lifts mapping 2-critical points to 1-critical points}\label{sec:2imp1}
Since {\oneimpliesone} fails on many sets of interest, we proceed to study {\twoimpliesone}. As Section~\ref{sec:main_results} demonstrates, this property is satisfied for many interesting lifts. 
We begin by stating an equivalent characterization for {\twoimpliesone} involving the following set.
\begin{definition}\label{def:W_y}
    For $y\in\calM$ and $x=\varphi(y)\in\calX$, define
    \begin{align*}
        W_y = \Big\{ w \in \calE : \textrm{there exists a twice differentiable function } f \colon \calE \to \reals \\ \textrm{such that $\nabla f(x) = w$ and $y$ is 2-critical for~\eqref{eq:Q}} \Big\}.
    \end{align*}
    We write $W_y^{\varphi}$ when we wish to emphasize the lift.
\end{definition}

\begin{theorem}\label{thm:W_equiv_cond}
The lift $\varphi\colon\calM\to\calX$ satisfies {\twoimpliesone} at $y$ if and only if $W_y\subseteq(\T_x\calX)^*$ where $x=\varphi(y)$.
\end{theorem}
\begin{proof}
    Say $\varphi$ satisfies {\twoimpliesone} at $y$ and let $w\in W_y$.
    Pick $f$ such that $y$ is 2-critical for~\eqref{eq:Q} and $\nabla f(x) = w$.
    By {\twoimpliesone}, we know $x$ is stationary for $f$, hence $w = \nabla f(x) \in (\T_x\calX)^*$.
    Conversely, say $W_y\subseteq(\T_x\calX)^*$ and let $y$ be 2-critical for~\eqref{eq:Q} with some cost $f$.
    Then $\nabla f(x)\in W_y\subseteq(\T_x\calX)^*$, hence $x$ is stationary for~\eqref{eq:P}. This shows {\twoimpliesone}.
\end{proof}
Since the 2-criticality of $y$ for~\eqref{eq:Q} only depends on the first two derivatives of $f$, we can restrict the functions $f$ in Definition~\ref{def:W_y} to be of class $C^{\infty}$ or even quadratic polynomials whose Hessians are a multiple of the identity, as the following proposition shows.
\begin{proposition}\label{prop:W_mult_of_id}
For $y\in\calM$ and $x=\varphi(y)$, the set $W_y$ in Definition~\ref{def:W_y} satisfies:
\begin{align*}
    W_y = \left\{w\in\calE: \exists\alpha>0 \textrm{ s.t. $y$ is 2-critical for~\eqref{eq:Q} with } f(x') = \langle x',w\rangle + \frac{\alpha}{2}\|x'-x\|^2\right\}.
\end{align*}
In particular, $W_y$ is a convex cone. 
\end{proposition}
\begin{proof}
    The inclusion $\supseteq$ is clear from the definition of $W_y$.
    Conversely, if $w$ is in $W_y$ then $y$ is 2-critical for~\eqref{eq:Q} for some $f$ with $\nabla f(x)=w$.
    Let $g=f\circ\varphi$ and $\alpha=\lambda_{\max}(\nabla^2f(x))$, and define 
    \begin{align*}
        f_{\alpha}(x')=\langle w,x'\rangle + \frac{\alpha}{2}\|x-x'\|^2,\quad g_{\alpha} = f_{\alpha}\circ\varphi.    
    \end{align*} 
    Note that $\nabla f_{\alpha}(x)=w$ and, by Lemma~\ref{lem:grad_and_hess_g}, we have $\nabla g_{\alpha}(y)=\Lmap_y^*(w)=\nabla g(y)=0$ and
    \begin{align*}
        \nabla^2 g_{\alpha}(y) &= \Lmap_y^*\circ \nabla^2 f_{\alpha}(x)\circ\Lmap_y + \nabla^2\varphi_{\nabla f_{\alpha}(x)}(y) = \Lmap_y^*\circ (\alpha I)\circ \Lmap_y + \nabla^2\varphi_w(y)\\ &\succeq \Lmap_y^*\circ \nabla^2f(x)\circ\Lmap_y + \nabla^2\varphi_w(y) = \nabla^2 g(y) \succeq 0.
    \end{align*}
    Thus, $y$ is 2-critical for $g_{\alpha}$. This shows the reverse inclusion. 

    $W_y$ is a convex cone since if $w_1, w_2$ are in $W_y$ as witnessed by functions $f_1, f_2$, then any $w = \lambda_1 w_1 + \lambda_2 w_2$ with $\lambda_1, \lambda_2 \geq 0$ is in $W_y$ as witnessed by $f = \lambda_1 f_1 + \lambda_2 f_2$.
\end{proof}

Proposition~\ref{prop:W_mult_of_id} shows that if {\twoimpliesone} is not satisfied at $y$, then there exists a simple strongly convex quadratic cost $f$ for which $y$ is 2-critical for~\eqref{eq:Q} but $x=\varphi(y)$ is not stationary for~\eqref{eq:P}.
\begin{corollary}\label{cor:2imp1_violating_cost}
Suppose $\varphi$ does not satisfy {\twoimpliesone} at $y\in\calM$ and denote $x=\varphi(y)$. Then $W_y\setminus(\T_x\calX)^*\neq\emptyset$ and for any $w$ in that set there exists $\alpha>0$ such that if $f(x')=\langle w,x'\rangle + \frac{\alpha}{2}\|x'-x\|^2$, then $y$ is 2-critical for~\eqref{eq:Q} but $x$ is not stationary for~\eqref{eq:P}.
\end{corollary}
We conjecture that the reverse inclusion in Theorem~\ref{thm:W_equiv_cond} always holds (it does for all the lifts in Section~\ref{sec:main_results}).
If this is indeed true, then $\varphi$ satisfies the {\twoimpliesone} property at $y$ if and only if $(\T_x\calX)^* = W_y$, neatly echoing the condition for {\oneimpliesone}, namely, $(\T_x\calX)^* = (\im \Lmap_y)^\perp$.
\begin{conjecture} \label{conj:reverseinclusionTxXdual}
	It always holds that $(\T_x \calX)^* \subseteq W_y$.
\end{conjecture}

The description of $W_y$ can be complicated. It is therefore worthwhile to derive sufficient conditions for {\twoimpliesone} that are easier to check. 
We do so by identifying two (admittedly technical) sets whose duals contain $\nabla f(x)$ if $x=\varphi(y)$ and $y$ is 2-critical for~\eqref{eq:Q}. The sufficient conditions then require the duals of these two subsets to be contained in $(\T_x\calX)^*$. 
\begin{definition}\label{def:A_B_sets}
    For $y\in\calM$, define
    \begin{align*}
        &A_y = \{w\in\calE:\exists c\colon \RR\to\calM \textrm{ smooth s.t. } c(0)=y,\ (\varphi\circ c)'(0) = 0,\ (\varphi\circ c)''(0)=w\},\\
        &B_y = \{w\in\calE:\exists c_i\colon \RR\to\calM \textrm{ smooth s.t. } c_i(0)=y,\ \lim_{i\to\infty}(\varphi\circ c_i)'(0) = 0,\ \lim_{i\to\infty}(\varphi\circ c_i)''(0) = w\}.
    \end{align*}
    We write $A_y^{\varphi}, B_y^{\varphi}$ when we wish to emphasize the lift.
\end{definition}
The following are the basic properties these two sets satisfy.
We give further expressions for $A_y, B_y$ and $W_y$ in Proposition~\ref{prop:2implies1_with_Qmap} below.
\begin{proposition}\label{prop:AB_basics}
Fix $y\in\calM$ and denote $x=\varphi(y)$. 
\begin{enumerate}[(a)]
    \item $A_y$ and $B_y$ are cones, and $B_y$ is closed.
    \item $A_y\subseteq\T_x\calX$ and $\im\Lmap_y\subseteq A_y\subseteq B_y$. Moreover, $\im\Lmap_y + A_y = A_y$ and $\im\Lmap_y + B_y = B_y$.
    \item If $y$ is 2-critical for $g=f\circ\varphi$, then $\nabla f(x)\in B_y^*$. 
    \item $W_y\subseteq B_y^*\subseteq A_y^*\subseteq(\im\Lmap_y)^{\perp}$.
\end{enumerate}
\end{proposition}
\begin{proof}
\ifthenelse{\boolean{shortver}}{Part (a) and the second half of part (b) are straightforward, see~\cite[App.~B]{levin2022effectARXIV}.}
{
The proofs of part (a) and the second half of part (b) are given in Appendix~\ref{apdx:AB_sets}.}
\begin{enumerate}[(a)]
    \item[(b)] If $c\colon \RR\to\calM$ satisfies $c(0)=y$ and $(\varphi\circ c)'(0)=0$, then $(\varphi\circ c)(t)\in\calX$ for all $t$ and $(\varphi\circ c)(t)= x + (t^2/2)(\varphi\circ c)''(0) + \mc O(t^3)$, hence by Definition~\ref{def:tangentcone} we have
    \begin{align*}
        (\varphi\circ c)''(0) = \lim_{t\to 0}\frac{(\varphi \circ c)(t) - x}{t^2/2} \in \T_x\calX.
    \end{align*}
    This shows $A_y\subseteq\T_x\calX$.
    
    If $w\in\im\Lmap_y$ so $w=\Lmap_y(v)$ for some $v\in\T_y\calM$, let $c\colon \RR\to\calM$ be a curve satisfying $c(0)=y$ and $c'(0)=v$. Define $\widetilde c\colon \RR\to\calM$ by $\widetilde c(t)=c(t^2/2)$, and note that $\widetilde c(0)=y$, $(\varphi\circ\widetilde c)'(0)=0$, and $(\varphi\circ \widetilde c)''(0) = (\varphi\circ c)'(0) = w$. Hence $w$ is in $A_y$. This shows $\im\Lmap_y\subseteq A_y$.
    
    It is clear that $A_y\subseteq B_y$ from Definition~\ref{def:A_B_sets}.
    
    \item[(c)] Suppose $y$ is $2$-critical for $g=f\circ\varphi$ and $w\in B_y$. Let $c_i\colon \RR\to\calM$ witness $w\in B_y$. Because $y$ is 1-critical, $(g\circ c_i)'(0)=0$ for all $i$. Because $y$ is 2-critical, for all $i$ we have
    \begin{align*}
        (g\circ c_i)''(0) = \langle \nabla f(x),(\varphi\circ c_i)''(0)\rangle + \langle \nabla^2f(x)[(\varphi\circ c_i)'(0)],(\varphi\circ c_i)'(0)\rangle\geq 0. 
    \end{align*}
    Taking $i\to\infty$, we conclude that $\langle\nabla f(x),w\rangle\geq0$ and hence $\nabla f(x)\in B_y^*$ as claimed. 
    
    \item[(d)] If $w\in W_y$ then there exists $f$ such that $\nabla f(x)=w$ and $y$ is 2-critical for~\eqref{eq:Q}, hence $w\in B_y^*$ by part (c). The other inclusions follow by taking duals in part (b). \qedhere
\end{enumerate}
\end{proof}
\ifthenelse{\boolean{shortver}}{
We remark that neither the inclusion $B_y\subseteq \T_x\calX$ nor $\T_x\calX\subseteq B_y$ hold in general, see~\cite[App.~B]{levin2022effectARXIV}.}
{We remark that neither the inclusion $B_y\subseteq \T_x\calX$ nor $\T_x\calX\subseteq B_y$ hold in general, see the end of Appendix~\ref{apdx:AB_sets}.}
Combining part (d) above with Theorem~\ref{thm:W_equiv_cond} yields the following sufficient conditions for {\twoimpliesone}.
\begin{corollary}\label{cor:AB_suff_cond}
Fix $y\in\calM$ and denote $x=\varphi(y)$.
\begin{enumerate}[(a)]
    \item If $A_y^*\subseteq (\T_x\calX)^*$, and in particular if $A_y=\T_x\calX$, then {\twoimpliesone} holds at $y$.
    \item If $B_y^*\subseteq (\T_x\calX)^*$, then {\twoimpliesone} holds at $y$.
\end{enumerate}
\end{corollary}

The conditions in Corollary~\ref{cor:AB_suff_cond} yield simpler proofs of {\twoimpliesone} for many lifts.
For example, the condition in Corollary~\ref{cor:AB_suff_cond}(a) holds for the Burer--Monteiro lift of Section~\ref{sec:smth_sdps}.
While it does not hold for the lifts of low-rank matrices in Section~\ref{sec:lowrk}, they do satisfy the stronger condition in Corollary~\ref{cor:AB_suff_cond}(b).
In fact, the condition in Corollary~\ref{cor:AB_suff_cond}(b) holds in all the examples satisfying {\twoimpliesone} that we consider.
It would be interesting to determine whether it is necessary as well.

We now state and prove the chain of implications we find the most useful for verifying or refuting {\twoimpliesone}, as well as for computing tangent cones (see Section~\ref{sec:low_rk_psd}). 
\begin{theorem}\label{thm:2implies1_chain}
Let $\varphi\colon\calM\to\calX$ be a smooth lift and fix $y\in\calM$. We have the following chain of implications for {\twoimpliesone}:
\begin{align*}
		&\T_x \calX \subseteq A_y\\
		\iff& \T_x \calX = A_y\\
		\implies& B_y^* \subseteq (\T_x \calX)^* \\
		\implies& W_y \subseteq (\T_x \calX)^* \\
		\iff& \varphi \textrm{ satisfies {\twoimpliesone} at } y \\
		\implies& (\im \Lmap_y)^\perp \cap (\Qmap_y(\T_y\calM))^* \subseteq (\T_x \calX)^*.
	\end{align*}
\end{theorem}
\begin{proof}
    The equivalence of the first two conditions follows by Proposition~\ref{prop:AB_basics}(b). The second condition implies the third by Proposition~\ref{prop:AB_basics}(b) as well. The third condition implies the fourth by Proposition~\ref{prop:AB_basics}(d), which itself is equivalent to {\twoimpliesone} at $y$ by Theorem~\ref{thm:W_equiv_cond}. 
    
    The last implication gives a necessary condition for {\twoimpliesone} to hold. Suppose there exists $w\in (\im \Lmap_y)^\perp \cap (\Qmap_y(\T_y\calM))^* \setminus (\T_x \calX)^*$.
	Define $f(x') = \inner{w}{x'}$, whose gradient and Hessian at $x$ are $\nabla f(x) = w$ and $\nabla^2 f(x) = 0$.
	For any curve $c\colon I\to\calM$ satisfying $c(0)=y$, denote $v=c'(0)\in\T_y\calM$.
    Let $g = f \circ \varphi$.
    Note that $(g\circ c)'(0)=\inner{w}{\Lmap_y(v)}=0$ since $w\in(\im\Lmap_y)^{\perp}$ and
    \begin{align*}
        (g\circ c)''(0) = (f\circ\varphi\circ c)''(0) = \langle w,(\varphi\circ c)''(0)\rangle = \langle w,\Qmap_y(v)\rangle \geq0,
    \end{align*}
    where the second equality follows from the chain rule, the third equality follows from Lemma~\ref{lem:diff_of_accel}(a) below, and the inequality follows from $w\in(\Qmap_y(\T_y\calM))^*$. Thus, $y$ is 2-critical for~\eqref{eq:Q}.
	However, $\nabla f(x)=w \notin (\T_x \calX)^*$ so $x$ is not stationary for~\eqref{eq:P}, hence {\twoimpliesone} does not hold at $y$.
\end{proof}

Our goal now is to derive more explicit expressions for the sets $A_y,B_y,W_y$ in terms of the maps $\Lmap_y$ and $\Qmap_y$ from Definition~\ref{def:LQ_maps}. Such expressions allow us to compute these sets in specific examples. 
To do so, we first recall that the value of $\Qmap_y(v)$ depends on the choice of curve $c_v$ in Definition~\ref{def:LQ_maps}. Before proceeding, we characterize the ambiguity in $\Qmap_y(v)$ arising from different such choices, verifying that it causes no issues.
\begin{lemma}\label{lem:diff_of_accel}
For each $y\in\calM$ and $v\in\T_y\calM$, let $c_v\colon I\to\calM$ be a curve satisfying $c_v(0)=y$ and $c_v'(0)=v$, so we can set $\Qmap_y(v)=(\varphi\circ c_v)''(0)$ according to Definition~\ref{def:LQ_maps}. 
\begin{enumerate}[(a)]
    \item For any other curve $c$ satisfying $c(0)=y$ and $c'(0)=v$, we have $(\varphi\circ c)''(0)-(\varphi\circ c_v)''(0)=\Lmap_y(c''(0)-c_v''(0))\in\im\Lmap_y$.
    
    \item For any $u\in\T_y\calM$, there exists a curve $c$ as in part (a) satisfying $c''(0)-c_v''(0)=u$, hence $(\varphi\circ c)''(0)-(\varphi\circ c_v)''(0)=\Lmap_y(u)$.
\end{enumerate}
In particular,
$
    \{(\varphi\circ c)''(0): c(0)=y \textrm{ and } c'(0)=v\} = \Qmap_y(v) + \im\Lmap_y.
$
%
%
%
\end{lemma}
\begin{proof}
\begin{enumerate}[(a)]
    \item For any $w\in\calE$, recall the function $\varphi_w(y)=\langle w,\varphi(y)\rangle$ from Definition~\ref{def:phi_w}. Let $c\colon I\to\calM$ be a curve satisfying $c(0)=y$ and $c'(0)=v$. Then, on the one hand, 
    \begin{align*}
        \left.\frac{\mathrm{d}^2}{\mathrm{d}t^2} \varphi_w(c(t)) \right|_{t=0} = \left.\frac{\mathrm{d}^2}{\mathrm{d}t^2} \inner{w}{(\varphi\circ c)(t)} \right|_{t=0} = \langle w,(\varphi\circ c)''(0)\rangle.
    \end{align*} 
    On the other hand, using the Riemannian structure on $\calM$, 
    \begin{align*}
        \left.\frac{\mathrm{d}^2}{\mathrm{d}t^2} \varphi_w(c(t)) \right|_{t=0} = \langle\nabla^2\varphi_w(y)[c'(0)],c'(0)\rangle + \langle \nabla\varphi_w(y),c''(0)\rangle.
    \end{align*} 
    By Lemma~\ref{lem:grad_and_hess_g}, we have $\nabla \varphi_w(y)=\Lmap_y^*(w)$, so $\langle\nabla\varphi_w(y),c''(0)\rangle = \langle w,\Lmap_y(c''(0))\rangle$. We conclude that
    \begin{align}
        \langle w,(\varphi\circ c)''(0)\rangle = \langle\nabla^2\varphi_w(y)[v],v\rangle + \langle w,\Lmap_y(c''(0))\rangle.
        \label{eq:w_lin_form}
    \end{align}
    The first term on the right-hand side is independent of $c$.
    Thus, for any $w\in\calE$ we have
    \begin{align*}
        \langle w,(\varphi\circ c)''(0) - (\varphi\circ c_v)''(0)\rangle = \langle w,\Lmap_y(c''(0)-c_v''(0))\rangle, 
    \end{align*}
    which proves the claim. 
    
    \item For the first claim, set $c(t)=\exp_y(tv + t^2(c_v''(0)-u)/2)$ where $\exp$ is the exponential map of $\calM$~\cite[Exer.~5.46]{optimOnMans}. The second claim follows from part (a). \qedhere
\end{enumerate}
\end{proof}
Lemma~\ref{lem:diff_of_accel} shows that the possible values of $\Qmap_y(v)$ (depending on the choice of curve $c_v$ in Definition~\ref{def:LQ_maps}) differ by an element of $\im\Lmap_y$, and conversely, every element of $\Qmap_y(v)+\im\Lmap_y$ can be obtained by an appropriate choice of $c_v$. Consequently, if $w\in(\im\Lmap_y)^{\perp}$, then the inner product $\langle w,\Qmap_y(v)\rangle$ is independent of the choice of $c_v$ in Definition~\ref{def:LQ_maps}. In fact, \eqref{eq:w_lin_form} shows that it is a quadratic form in $v\in \T_y\calM$ given by:
\begin{align}
    \langle w,\Qmap_y(v)\rangle = \langle\nabla^2\varphi_w(y)[v],v\rangle\quad \forall v\in\T_y\calM.
    \label{eq:lin_quadr_Qmap_rel}
\end{align}
We stress that this identity requires $w \in (\im \Lmap_y)^\perp$ in general. It allows us to view $\langle w,\Qmap_y(v)\rangle$ interchangeably as either a quadratic form in $v$ on $\T_y\calM$ or a linear form in $w$ on $(\im\Lmap_y)^{\perp}$. 
\begin{remark}[Disambiguation of $\Qmap_y$]\label{rmk:nonunique_Q} 
    \ifthenelse{\boolean{shortver}}{%
    Lemma~\ref{lem:diff_of_accel} implies that it is natural to define $\Qmap_y$ on the quotient $\calE/\im\Lmap_y$ to avoid the above ambiguities, but it is less convenient to use in practice---see~\cite[Rmk.~3.25]{levin2022effectARXIV}. 
    Accordingly, in the remainder of the paper we often refer to $\Qmap_y$ modulo $\im\Lmap_y$.
    }
    {
    Given the above ambiguity in $\Qmap_y$, it would be natural to define the codomain of $\Qmap_y$ to be the quotient vector space $\calE/\im\Lmap_y$. This would make it independent of the choice of $c_v$.
    Subsets of the quotient $\calE/\im\Lmap_y$ are in bijection with subsets of $\calE$ that are closed under addition with $\im\Lmap_y$ (i.e.\ subsets $S\subseteq\calE$ such that $S+\im\Lmap_y = S$), which includes $A_y$ and $B_y$. Subsets of the dual vector space to $\calE/\im\Lmap_y$ are in bijection with subsets of $(\im\Lmap_y)^{\perp}$, which includes $W_y$. Hence we could equivalently phrase our conditions for {\twoimpliesone} in terms of subsets of $\calE/\im\Lmap_y$ and its dual. 
    However, we have several techniques to obtain expressions for $\Qmap_y$ without explicitly choosing curves, see Section~\ref{sec:Qmap_comp} below. Thus, Definition~\ref{def:LQ_maps} mirrors the computations we do more closely. In practice, we view $\Qmap_y$ as taking values in $\calE$, and consider two maps differing by elements of $\im\Lmap_y$ as equivalent for the purpose of verifying {\twoimpliesone} since they yield the same sets $A_y,B_y,W_y$.
    }   
\end{remark}


We now express the sets $A_y,B_y,W_y$ appearing in our conditions for {\twoimpliesone} in terms of the maps $\Lmap_y$ and $\Qmap_y$. We explain how to compute $\Lmap_y$ and $\Qmap_y$ in various settings in Section~\ref{sec:Qmap_comp} below.
\begin{proposition}\label{prop:2implies1_with_Qmap}
For any $y\in\calM$,
\begin{enumerate}[(a)]
    \item $A_y = \Qmap_y(\ker\Lmap_y) + \im\Lmap_y$.
    
    \item $B_y = \underset{\substack{(v_i)_{i\geq1}:\Lmap_y(v_i)\to 0}}{\bigcup}\lim_{i\to\infty}(\Qmap_y(v_i)+ \im\Lmap_y).\footnote{A sequence $(v_i+W)_{i\geq 1}$ of translates of a subspace $W$ of a (topological) vector space $V$ converges (necessarily to another translate of $W$) iff there exist $w_i\in W$ such that $(v_i+w_i)_{i\geq 1}\subseteq V$ converges in $V$.}$
    \item $W_y = \left\{w\in A_y^*: \forall v\in\ker\Lmap_y, \langle \nabla^2\varphi_w(y)[v],v\rangle=0 \implies \nabla^2\varphi_w(y)[v] = 0\right\}$.
\end{enumerate}
\end{proposition}
\ifthenelse{\boolean{shortver}}{}
    {
We always have
\begin{align}
    B_y \supseteq \left\{\lim\nolimits_i\Qmap_y(v_i):\Lmap_y(v_i)\to 0\right\} + \im\Lmap_y,
    \label{eq:By_nice_subset}
\end{align}
but inclusion may be strict depending on the curves $c_v$ in Definition~\ref{def:LQ_maps}. 
In practice, the expressions for $\Qmap_y$ we obtain using our techniques from Section~\ref{sec:Qmap_comp} below are smooth, and the subset of $B_y$ in~\eqref{eq:By_nice_subset} is large enough to prove {\twoimpliesone} in every example where it holds. 
}
\begin{proof}
\begin{enumerate}[(a)]
    \item If $w\in A_y$, then $w=(\varphi\circ c)''(0)$ for some smooth curve $c$ on $\calM$ such that $c(0)=y$ and $0=(\varphi\circ c)'(0) = \Lmap_y(c'(0))$, so $c'(0)\in\ker\Lmap_y$. By Lemma~\ref{lem:diff_of_accel}(a), we have $w \in \Qmap_y(c'(0)) + \im\Lmap_y$, showing $A_y\subseteq \Qmap_y(\ker\Lmap_y)+\im\Lmap_y$. 
    
    Conversely, suppose $w = \Qmap_y(v) + \Lmap_y(u)$ for some $v\in\ker\Lmap_y$ and $u\in\T_y\calM$. By Lemma~\ref{lem:diff_of_accel}(b), there is a smooth curve $c$ on $\calM$ satisfying $c(0)=y$, $c'(0)=v$ and $(\varphi\circ c)''(0) = w$. Since $(\varphi\circ c)'(0)=\Lmap_y(v)=0$, this shows $\Qmap_y(\ker\Lmap_y)+\im\Lmap_y\subseteq A_y$.
    
    \item If $w\in B_y$, then there are smooth curves $c_i$ such that $c_i(0)=y$, $\Lmap_y(c_i'(0))\to0$ and $(\varphi\circ c_i)''(0)\to w$. By Lemma~\ref{lem:diff_of_accel}(a), we have $(\varphi\circ c_i)''(0)\in \Qmap_y(c_i'(0)) + \im\Lmap_y$. Because $\lim_i(\varphi\circ c_i)''(0) = w$ exists, we conclude that $\lim_i(\Qmap_y(c_i'(0))+\im\Lmap_y) = w + \im\Lmap_y$ exists as well, and $w$ is contained in this limit. This shows the inclusion $\subseteq$ in the claim. 
    
    Conversely, suppose $w\in \lim_i(\Qmap_y(v_i)+\im\Lmap_y)$ for some sequence $(v_i)_{i\geq 1}\subseteq \T_y\calM$ such that $\Lmap_y(v_i)\to0$. Then there exist $u_i\in\T_y\calM$ such that $w = \lim_i(\Qmap_y(v_i) + \Lmap_y(u_i))$. By Lemma~\ref{lem:diff_of_accel}(b), there exist curves $c_i$ satisfying $c_i(0)=y$, $c_i'(0)=v_i$ and $(\varphi\circ c_i)''(0) = \Qmap_y(v_i) + \Lmap_y(u_i)$. Then $(\varphi\circ c_i)'(0)=\Lmap_y(v_i)\to 0$ and $(\varphi\circ c_i)''(0)\to w$, so $w\in B_y$ and hence the reverse inclusion in the claim also holds.

    \item Let $x=\varphi(y)$. By Proposition~\ref{prop:W_mult_of_id}, a vector $w\in\calE$ is contained in $W_y$ iff there exists $\alpha>0$ such that $y$ is 2-critical for~\eqref{eq:Q} with cost $g_{\alpha}=f_{\alpha}\circ\varphi$ where $f_{\alpha}(x') = \langle w,x'\rangle + \tfrac{\alpha}{2}\|x'-x\|^2$. By Lemma~\ref{lem:grad_and_hess_g}, this is equivalent to
    \begin{align*}
        \nabla g_{\alpha}(y) = \Lmap_y^*(w) = 0,\qquad \nabla^2g_{\alpha}(y) = \alpha \, \Lmap_y^*\circ\Lmap_y^{} + \nabla^2\varphi_w(y)\succeq0.
    \end{align*}
    In other words, $w\in W_y$ iff $w\in(\im\Lmap_y)^{\perp}$ and $\nabla^2\varphi_w(y) + \alpha \, \Lmap_y^*\circ\Lmap_y^{}\succeq0$ for some $\alpha>0$. To understand when the second condition holds, we decompose $\T_y\calM=\ker\Lmap_y\oplus(\ker\Lmap_y)^{\perp}$ and express the relevant self-adjoint operators on $\T_y\calM$ in block matrix form with respect to a basis compatible with this decomposition. 
    More explicitly, choose a basis as described above and denote $n=\dim\ker\Lmap_y$ and $m=\dim(\ker\Lmap_y)^{\perp}$.
    Assume first that $m>0$.
    We represent $\nabla^2\varphi_w(y)$ and $\alpha \, \Lmap_y^*\circ\Lmap_y^{}$ in that basis as
    \begin{align*}
        \nabla^2\varphi_w(y) &= \begin{bmatrix} \Phi_1 & \Phi_2\\ \Phi_2^\top & \Phi_3\end{bmatrix},\quad &&\textrm{with }\ \Phi_1\in\mbb S^n, \Phi_3\in\mbb S^m.\\
        \alpha \, \Lmap_y^*\circ\Lmap_y^{} &= \begin{bmatrix} 0 & 0\\ 0 & \alpha \Psi\end{bmatrix},\quad &&\textrm{with }\ \Psi\in\mbb S^m_{\succ 0}.
    \end{align*}
	Thus,
	\begin{align*}
		w & \in W_y & \iff && w \in (\im \Lmap_y)^\perp \textrm{ and } \exists \alpha>0 \textrm{ such that } \begin{bmatrix}
			\Phi_1 & \Phi_2 \\ \Phi_2\transpose & \Phi_3 + \alpha\Psi
		\end{bmatrix} & \succeq 0.
	\end{align*}
	By the generalized Schur complement theorem~\cite[Thm.~1.20]{zhang2005schur}, the block-matrix on the right-hand side is positive semidefinite exactly if
	\begin{align*}
		\Phi_1 \succeq 0, && \im \Phi_2 \subseteq \im \Phi_1 && \textrm{ and } && \Phi_3 + \alpha\Psi  \succeq \Phi_2\transpose \Phi_1^\dagger \Phi_2^{},
	\end{align*}
	where $\Phi_1^\dagger$ is the Moore--Penrose pseudo-inverse of $\Phi_1$.
	The last condition holds upon choosing $\alpha \geq \lambda_{\max}(\Phi_2^\top\Phi_1^{\dagger}\Phi_2^{} - \Phi_3) / \lambda_{\min}(\Psi)$.
	Thus, we deduce the following expression for $W_y$:
	\begin{align*}
		W_y = \{ w \in (\im \Lmap_y)^\perp : \Phi_1 \succeq 0 \textrm{ and } \im \Phi_2 \subseteq \im \Phi_1 \},\\ 
	\end{align*}
	with $\Phi_1$ and $\Phi_2$ as defined above.
	We now work out basis-free characterizations of the properties $\Phi_1 \succeq 0$ and $\im \Phi_2 \subseteq \im \Phi_1$.
	
	First, notice that $\Phi_1 \succeq 0$ iff
	\begin{align*}
	    \begin{bmatrix} v_1\transpose & \mathbf{0}_{m}\transpose \, \end{bmatrix} \begin{bmatrix} \Phi_1 & \Phi_2\\ \Phi_2^\top & \Phi_3\end{bmatrix}\begin{bmatrix} v_1 \\ \mathbf{0}_{m}\end{bmatrix}\geq0 \quad \textrm{for all } v_1\in\RR^n,
	\end{align*}
	or in basis-free terms, $\inner{\nabla^2\varphi_w(y)[v]}{v}\geq0$ for all $v\in\ker\Lmap_y$. If $w\in(\im\Lmap_y)^{\perp}$ then~\eqref{eq:lin_quadr_Qmap_rel} shows that this is also equivalent to $\inner{w}{\Qmap_y(v)}\geq0$ for all $v\in\ker\Lmap_y$, which is in turn equivalent to $\inner{w}{\Qmap_y(v) + \Lmap_y(u)}\geq0$ for all $v\in\ker\Lmap_y$ and $u\in\T_y\calM$. This last condition is just $w\in A_y^*$ by part (a).
	
	Second, we must understand for which vectors $w$ it holds that $\im \Phi_2 \subseteq \im \Phi_1$, or equivalently, $\ker \Phi_1^{} \subseteq \ker \Phi_2\transpose$ (recall that $\Phi_1^\top=\Phi_1^{}$). If $\Phi_1\succeq0$, then $v_1\in\ker\Phi_1$ iff $v_1\transpose \Phi_1^{} v_1^{} = 0$. Moreover, if $v_1\in\ker\Phi_1$ then 
	\begin{align*}
	    \begin{bmatrix} \Phi_1 & \Phi_2\\ \Phi_2^\top & \Phi_3\end{bmatrix}\begin{bmatrix} v_1 \\ \mathbf{0}_{m}\end{bmatrix} = \begin{bmatrix} \mathbf{0}_n \\ \Phi_2\transpose v_1^{} \end{bmatrix},
	\end{align*} 
	which vanishes iff $v_1\in\ker\Phi_2\transpose$.
    Thus, assuming $\Phi_1 \succeq 0$, the inclusion $\im\Phi_2\subseteq\im\Phi_1$ is equivalent to the implication
	\begin{align*}
	   \begin{bmatrix} v_1\transpose & \mathbf{0}_{m}\transpose \, \end{bmatrix} \begin{bmatrix} \Phi_1 & \Phi_2\\ \Phi_2^\top & \Phi_3\end{bmatrix}\begin{bmatrix} v_1 \\ \mathbf{0}_{m}\end{bmatrix} = 0 \implies \begin{bmatrix} \Phi_1 & \Phi_2\\ \Phi_2^\top & \Phi_3\end{bmatrix}\begin{bmatrix} v_1 \\ \mathbf{0}_{m}\end{bmatrix} = 0,\qquad \textrm{ for all } v_1\in\RR^n.
	\end{align*}
	In basis-free terms, we have shown that, if $\Phi_1\succeq0$, then $\im\Phi_2\subseteq\im\Phi_1$ is equivalent to the implication $\inner{\nabla^2\varphi_w(y)[v]}{v}=0\implies \nabla^2\varphi_w(y)[v]=0$ holding for all $v\in\ker\Lmap_y$. Putting everything together,
	\begin{align*}
		W_y & = \{ w \in (\im \Lmap_y)^\perp : \Phi_1 \succeq 0 \textrm{ and } \im \Phi_2 \subseteq \im \Phi_1 \}\\ 
		    & = \{w\in (\im\Lmap_y)^{\perp}: w\in A_y^* \textrm{ and } \forall v\in\ker\Lmap_y, \inner{\nabla^2\varphi_w(y)[v]}{v}=0\implies \nabla^2\varphi_w(y)[v]=0\}\\
		    & = \{w\in A_y^*: \forall v\in\ker\Lmap_y, \inner{\nabla^2\varphi_w(y)[v]}{v}=0\implies \nabla^2\varphi_w(y)[v]=0\},
	\end{align*}
	where the last equality holds because $A_y^*\subseteq(\im\Lmap_y)^{\perp}$ by Proposition~\ref{prop:AB_basics}(b). This is the claimed expression for $W_y$.
	%

    If $m=0$, or equivalently, if $\Lmap_y=0$, then $w\in W_y$ iff $w\in(\im\Lmap_y)^{\perp}=\calE$ and $\nabla^2\varphi_w(y)\succeq0$. This in turn is equivalent to $w\in A_y^*=(\Qmap_y(\ker\Lmap_y))^*\cap(\im\Lmap_y)^{\perp}$ by~\eqref{eq:lin_quadr_Qmap_rel}, so $W_y=A_y^*$ in this case. Conversely, if $w\in A_y^*$ and $\Lmap_y=0$ then $\nabla^2\varphi_w(y)\succeq0$ so the condition in the claimed expression for $W_y$ is satisfied automatically: it also evaluates to $A_y^*$.
    This verifies that the claimed expression for $W_y$ holds for $m=0$ as well. \qedhere
\end{enumerate}\end{proof}
\subsection{Composition of lifts}\label{sec:composition_of_lifts}
In this section, we ask: when are lift properties preserved under composition?
We use the following proposition both to compute $\Lmap_y$ and $\Qmap_y$ in various settings, and to study some of the lifts appearing in the literature in Sections~\ref{sec:fiber_prod_lifts} and~\ref{sec:examples}.
\begin{proposition}\label{prop:composition_of_lifts}
Let $\varphi\colon\calM\to\calX$ be a smooth lift, and let $\psi\colon\mc N\to\calM$ be a smooth map between smooth manifolds such that $\varphi\circ\psi\colon\mc N\to\calX$ is surjective. Both $\varphi$ and $\varphi \circ \psi$ are smooth lifts for $\calX$. For $z\in\mc N$ and $y=\psi(z)\in\calM$, the following hold.
\begin{enumerate}[(a)]
    \item If $\varphi\circ\psi$ satisfies {\localimplieslocal} at $z$, then $\varphi$ satisfies {\localimplieslocal} at $y$. If $\psi$ is open (in particular, if $\psi$ is a submersion) at $z$, and if $\varphi$ satisfies {\localimplieslocal} at $y$, then $\varphi\circ\psi$ satisfies {\localimplieslocal} at $z$. 
    
    \item If $\varphi\circ\psi$ satisfies {\oneimpliesone} or {\twoimpliesone} at $z$, then $\varphi$ satisfies the corresponding property at $y$. If $\psi$ is a submersion at $z$ and $\varphi$ satisfies {\oneimpliesone} or {\twoimpliesone} at $y$, then $\varphi\circ\psi$ satisfies the corresponding property at $z$.
    
    \item If $\psi$ is a submersion at $z$, then 
    \begin{align*}
        \Lmap_z^{\varphi\circ\psi} = \Lmap_y^{\varphi} \circ \Lmap_z^{\psi} && \textrm{ and } && \Qmap_z^{\varphi\circ\psi} \equiv \Qmap_y^{\varphi} \circ \Lmap_z^{\psi} \mod \im\Lmap_z^{\varphi\circ\psi}.
    \end{align*}
    Moreover, $\im\Lmap_z^{\varphi\circ\psi} = \im\Lmap_y^{\varphi}$, $A_z^{\varphi\circ\psi} = A_y^{\varphi}$, $B_z^{\varphi\circ\psi} = B_y^{\varphi}$, and $W_z^{\varphi\circ\psi} = W_y^{\varphi}$.
\end{enumerate}
\end{proposition}
\ifthenelse{\boolean{shortver}}
{The proof is straightforward, see~\cite[App.~C.1]{levin2022effectARXIV}.}
{The proof is straightforward but long, and is deferred to Appendix~\ref{apdx:comps}.}
Here we denote $v\equiv w\mod\im\Lmap_y$ to mean $v-w\in\im\Lmap_y$. 
By Lemma~\ref{lem:diff_of_accel}, equality of $\Qmap_z^{\varphi\circ\psi}$ and $\Qmap_y^{\varphi}$ modulo $\im\Lmap_y^{\varphi}=\im\Lmap_z^{\varphi\circ\psi}$ means that either one can be used to verify {\twoimpliesone}.

Proposition~\ref{prop:composition_of_lifts} shows that, given a smooth lift $\varphi\colon\calM\to\calX$, there is no benefit to further lifting $\calM$ to another smooth manifold through $\psi\colon\calN\to\calM$ in terms of our properties.
Indeed, if $\varphi$ does not satisfy one of our properties, then neither does $\varphi\circ\psi$ for any smooth $\psi$ (we cannot `fix' a bad lift by lifting it further).
On the other hand, this proposition also tells us that our properties, as well as the sets involved in their characterization, are preserved under submersions.
This notably means lift properties can be checked through charts of $\calM$.
Moreover, for lifts to a manifold $\calM$ which is a quotient of another manifold $\overline{\calM}$ (these arise naturally when quotienting by group actions, see~\cite[\S9]{optimOnMans}), Proposition~\ref{prop:composition_of_lifts} allows us to verify our properties on the total space $\overline{\calM}$, which is often easier.
\begin{remark}\label{rmk:horiz_space_2imp1}
If $\psi\colon\mc N\to\calM$ is a submersion, for each $z\in\mc N$, let $V_z=\ker\D\psi(z)$ and $H_z=(\ker\D\psi(z))^{\perp}$ be the so-called vertical and horizontal spaces at $z$, which satisfy $\T_z\mc N=V_z\oplus H_z$ and $H_z\cong\T_{\psi(z)}\calM$. Proposition~\ref{prop:composition_of_lifts} implies that $\Lmap_z^{\varphi\circ\psi} = \Lmap_z^{\varphi\circ\psi}\circ\Proj_{H_z}$ and $\Qmap_z^{\varphi\circ\psi} \equiv \Qmap_z^{\varphi\circ\psi}\circ \Proj_{H_z}$ where $\Proj_{H_z}$ denotes orthogonal projection onto $H_z$, so it suffices to consider the restrictions of $\Lmap_z^{\varphi\circ\psi}$ and $\Qmap_z^{\varphi\circ\psi}$ to the horizontal space at $z$. The latter is often simpler than $\T_z\mc N$, see~\cite[\S9.4]{optimOnMans}.
\end{remark}

We end this section with an implicit application of Proposition~\ref{prop:composition_of_lifts} seen in the robotics and computer vision literature.
\begin{example}[Shohan rotation averaging]\label{ex:shohan_rots}
In~\cite{shohan_rots}, Dellaert et al.\ estimate a set of $n$ rotations of $\RR^d$ from noisy measurements of pairs of relative rotations. Their algorithm involves the Burer--Monteiro lift~\eqref{eq:BM_lift} with $\calM=\mathrm{St}(p,d)^n$ for appropriate $p\geq d$, composed with the submersion $\psi\colon \mathrm{SO}(p)^n\to\mathrm{St}(p,d)^n$ extracting the first $d$ columns of each matrix. 
Using our framework, we can analyze this composition and thereby strengthen the guarantees proved in~\cite{shohan_rots}. Indeed, since the Burer--Monteiro lift satisfies {\twoimpliesone} and {\localimplieslocal} by Proposition~\ref{prop:BM_lift} and 
$\psi$ is a submersion, Proposition~\ref{prop:composition_of_lifts} shows that the composed lift satisfies {\localimplieslocal} and {\twoimpliesone} as well. 
Furthermore, if every stationary point for the original low-rank SDP of~\cite{shohan_rots} is globally optimal (e.g., if the conditions of~\cite{BM_det_guarantees} hold), then any 2-critical point for their lifted problem is globally optimal and therefore the lifted problem enjoys benign nonconvexity.

\end{example}

\subsection{Computing $\Lmap_y$ and $\Qmap_y$}\label{sec:Qmap_comp}
Theorem~\ref{thm:2implies1_chain} gives several conditions for {\twoimpliesone} that (together with Proposition~\ref{prop:2implies1_with_Qmap}) can be checked using $\Lmap_y$ and $\Qmap_y$ from Definition~\ref{def:LQ_maps}.
We therefore consider various strategies for computing $\Lmap_y$ and $\Qmap_y$ depending on how we can access $\calM$.
Since $\Qmap_y$ is only defined modulo $\im\Lmap_y$, different methods may yield different expressions, any of which can be used to verify {\twoimpliesone}. 


\paragraph{$\calM$ through charts:} Suppose we are given a chart $\psi\colon U\to\calM$ on $\calM$, which is a diffeomorphism from an open subset $U\subseteq\calE'$ of some linear space $\calE'$ onto its image, and let $y\in\psi(U)$. 
Then we can compose $\varphi$ with $\psi$ to obtain a lift to a linear space $\widetilde\varphi=\varphi\circ\psi$. 
By Proposition~\ref{prop:composition_of_lifts}, the lift $\varphi$ satisfies {\oneimpliesone} or {\twoimpliesone} at $y\in\calM$ if and only if $\widetilde\varphi$ satisfies the corresponding property at $z=\psi^{-1}(y)$. 
Thus, it suffices to compute $\Lmap_z^{\widetilde\varphi}$ and $\Qmap_z^{\widetilde\varphi}$ and use them to check {\twoimpliesone} at $z$. 
Since $U$ is an open subset of a linear space $\calE'$, it is natural to compute $\Lmap_z^{\widetilde\varphi}$ and $\Qmap_z^{\widetilde\varphi}$ directly from Definition~\ref{def:LQ_maps} using curves $\widetilde c_{\widetilde v}(t) = z + t\widetilde v$ which are straight lines through $z$ in direction $\widetilde v\in \calE'=\T_zU$.
This choice yields the expressions
\begin{align}
    \Lmap_z^{\widetilde\varphi}(\widetilde v) = (\widetilde\varphi\circ\widetilde c_{\widetilde v})'(0)=\D\widetilde\varphi(z)[\widetilde v],\quad \textrm{and}\quad \Qmap_z^{\widetilde\varphi}(\widetilde v) = (\widetilde\varphi\circ\widetilde c_{\widetilde v})''(0)=\D^2\widetilde\varphi(z)[\widetilde v,\widetilde v],
    \label{eq:L_and_Q_on_chart}
\end{align} 
where $\D\widetilde\varphi(z)$ and $\D^2\widetilde\varphi(z)$ are the ordinary first- and second-order derivative maps of $\widetilde\varphi$ viewed as a map between linear spaces $\calE'\to\calE$. 
In particular, if $\calM$ is itself a linear space (e.g., for the~\eqref{eq:LR_lift} lift), we may take $U=\calE'=\calM$ and $\psi=\mathrm{id}$ and use~\eqref{eq:L_and_Q_on_chart} with $\widetilde\varphi=\varphi$.

\paragraph{$\calM$ embedded in a linear space:} Suppose now that $\calM$ is an embedded submanifold of another linear space $\calE'$.
By~\cite[Prop.~3.31]{optimOnMans}, the lift $\varphi$ can be extended to a smooth map on a neighborhood $V$ of $\calM$ in $\calE'$, denoted by $\overline{\varphi}\colon V\to\calE$.
This means $\overline{\varphi}$ is a smooth map defined on an open subset $V\subseteq\calE'$ containing $\calM$ and it satisfies $\overline{\varphi}|_{\calM}=\varphi$.
If $c_v$ is a curve on $\calM$ passing through $y\in\calM\subseteq V$ with velocity $v\in\T_y\calM\subseteq \T_yV=\calE'$, then $\varphi\circ c_v = \overline{\varphi}\circ c_v$ because the curve is contained in $\calM$ where $\overline{\varphi}$ agrees with $\varphi$.
Denote by $u_v = \ddot c_v(0)$ the ordinary (extrinsic) acceleration of $c_v$ at $t=0$, viewed as a curve in $\calE'$.
Then Definition~\ref{def:LQ_maps} and the chain rule give
\begin{align*}
    \Lmap_y(v) = (\overline{\varphi}\circ c_v)'(0) = \D\overline{\varphi}(y)[v], && \Qmap_y(v) = (\overline{\varphi}\circ c_v)''(0) = \D^2\overline{\varphi}(y)[v,v] + \D\overline{\varphi}(y)[u_v].
\end{align*}
To better understand $u_v$, let $h\colon\calE'\to\RR^k$ be a local defining function for $\calM$ around $y$, that is, $h$ is smooth, $\calM$ is locally its zero-set, and $\rank\, \D h(y')=k$ for all $y'$ around $y$.
For a curve $c_v$ as above, we have $h(c_v(t))\equiv 0$ around $t = 0$, so in particular $(h\circ c_v)'(0) = (h\circ c_v)''(0)=0$.
By the chain rule, the latter equations can be written as
\begin{align}
    \D h(y)[v] = 0 && \textrm{ and } && \D^2h(y)[v,v] + \D h(y)[u_v] = 0.
    \label{eq:1st_and_2nd_eqn_pass1}
\end{align} 
Conversely, for any $v,u_v\in \calE'$ satisfying~\eqref{eq:1st_and_2nd_eqn_pass1}, there exists a curve on $\calM$ passing through $y$ with velocity $v$ and extrinsic acceleration $u_v$ by~\cite[Prop.~13.13]{rockafellar2009variational}.
Thus, the expressions~\eqref{eq:1st_and_2nd_eqn_pass1} describe all possible velocities and extrinsic accelerations of curve as they pass through $y$.
This set of all possible such accelerations of curves on $\calM$ passing through $y\in\calM$ with velocity $v\in\T_y\calM$ is the \emph{second-order tangent set} to $\calM$ at $y$ for $v$, and is denoted by $\T^2_{y,v}\calM$~\cite[Def.~13.11]{rockafellar2009variational}.
The above discussion shows that
\begin{align}
    \T_y\calM = \ker\D h(y) && \textrm{ and } && \T^2_{y,v}\calM = \left\{u\in\calE':\D h(y)[u] = -\D^2h(y)[v,v]\right\}.
    \label{eq:1st_and_2nd_ord_tngnts}
\end{align}
As a result, for any extension $\overline{\varphi}$ of $\varphi$ and all $v \in \T_y\calM$, we have
\begin{align}
    \Lmap_y(v) = \D\overline{\varphi}(y)[v] && \textrm{ and } && \Qmap_y(v) = \D^2\overline{\varphi}(y)[v,v] + \D\overline{\varphi}(y)[u_v]\ \textrm{ for some } u_v\in\T^2_{y,v}\calM.
    \label{eq:L_and_Q_embedded}
\end{align}
Note that $\T^2_{y,v}\calM$ is an affine subspace of $\calE'$ which is a translate of the subspace $\T_y\calM$, as can be seen from~\eqref{eq:1st_and_2nd_ord_tngnts}.
Therefore, while different choices of $u_v$ lead to different expressions for $\Qmap_y$, they are all equal modulo $\im\Lmap_y$.


\paragraph{$\calM$ as a quotient manifold:} Suppose next that $\calM$ is a quotient manifold of $\overline{\calM}$ with quotient map $\pi\colon\overline{\calM}\to\calM$ \cite[\S9]{optimOnMans}.
Then $\overline{\varphi}=\varphi\circ\pi$ gives a smooth lift of $\calX$ to $\overline{\calM}$. Since $\pi$ is a submersion, Proposition~\ref{prop:composition_of_lifts} and Remark~\ref{rmk:horiz_space_2imp1} imply that to check {\twoimpliesone}, we need only compute $\Lmap_z$ and $\Qmap_z$ for $\overline{\varphi}$ restricted to the horizontal spaces $(\ker\D\pi(z))^{\perp}$ using the preceding two methods.

\paragraph{Computing $\nabla^2\varphi_w$:} To check the equivalent condition in Theorem~\ref{thm:2implies1_chain} for any presentation of $\calM$, we need to compute $W_y$. If we use Proposition~\ref{prop:2implies1_with_Qmap} to do so, we need an expression for the Riemannian Hessian $\nabla^2\varphi_w(y)$ where $\varphi_w(y)=\langle w, \varphi(y)\rangle$ for $w\in(\im\Lmap_y)^{\perp}$. Given $\Qmap_y$, we can obtain $\nabla^2\varphi_w(y)$ as the unique self-adjoint operator on $\T_y\calM$ that defines the quadratic form~\eqref{eq:lin_quadr_Qmap_rel}. Conversely, if we compute $\nabla^2\varphi_w(y)$ for all $w\in(\im\Lmap_y)^{\perp}$, e.g., using the techniques from~\cite[\S5.5]{optimOnMans}, we can set $\Qmap_y(v)$ to be the unique element of $(\im\Lmap_y)^{\perp}$ satisfying~\eqref{eq:lin_quadr_Qmap_rel}, providing another way to compute $\Qmap_y$. 

We now illustrate the above techniques for computing $\Lmap_y$ and $\Qmap_y$.
The following example uses charts.
\begin{example}[Desingularization of $\Rmnlr$]\label{ex:desing_Qmap_comp} 
Consider the desingularization lift~\eqref{eq:desing_lift} of bounded rank matrices $\calX=\Rmnlr$.
We compute $\Lmap$ and $\Qmap$ using charts. For $(X_0,\mc S_0)\in\calM$, let $Y_0\in\RR^{n\times(n-r)}$ be a matrix satisfying $\mathrm{col}(Y_0)=\mc S_0$, so $X_0Y_0=0$. Since $\rank(Y_0)=n-r$, we can find $n-r$ linearly independent rows in $Y_0$. Let $J\in\RR^{(n-r)\times(n-r)}$ be the invertible submatrix of $Y_0$ obtained by extracting these rows, and let $\Pi\in\RR^{n\times n}$ be a permutation matrix sending those $n-r$ rows to the first rows. Then 
\begin{align*}
    \Pi Y_0 J^{-1} = \begin{bmatrix} I_{n-r}\\ W_0\end{bmatrix} && \textrm{ and } && X_0 = \begin{bmatrix}-Z_0W_0, & Z_0\end{bmatrix}\Pi && \textrm{for some } Z_0\in\RR^{m\times r} \textrm{ and } W_0\in\RR^{r\times(n-r)},
\end{align*}
where the second identity is implied by $0 = X_0Y_0J^{-1}=X_0\Pi^\top(\Pi Y_0 J^{-1})$. Accordingly, a chart $\psi\colon \RR^{m\times r}\times \RR^{r\times(n-r)}\to\calM$ containing $(X_0,\mc S_0)$ is given by
\begin{align*}
    \psi(Z, W) = \left(\begin{bmatrix} -ZW, & Z\end{bmatrix}\Pi,\ \mathrm{col}\!\left(\Pi^\top\begin{bmatrix} I_{n-r}\\ W\end{bmatrix}\right)\right).
\end{align*}
Composing with $\varphi$, we obtain the lift $\widetilde\varphi(Z,W) = \varphi(\psi(Z,W))=\begin{bmatrix} -ZW, & Z\end{bmatrix}\Pi$, and by~\eqref{eq:L_and_Q_on_chart},
\begin{align}\label{eq:desing_LQ_maps_chart}
\begin{split} 
    &\Lmap_{(Z,W)}^{\widetilde\varphi}(\dot Z,\dot W) = \D\widetilde\varphi(Z,W)[\dot Z,\dot W] = \begin{bmatrix} -\dot ZW - Z\dot W, & \dot Z\end{bmatrix}\Pi,\\
    &\Qmap_{(Z,W)}^{\widetilde\varphi}(\dot Z,\dot W) = \D^2\widetilde\varphi(Z,W)[(\dot Z,\dot W),(\dot Z,\dot W)] = \begin{bmatrix} -2\dot Z\dot W, & 0\end{bmatrix}\Pi.
\end{split}
\end{align}
For $V=\begin{bmatrix}V_1, & V_2\end{bmatrix}\Pi\in\RR^{m\times n}$ where $V_2$ is $m\times r$, the Hessian $\nabla^2\widetilde\varphi_V(Z,W)$ is the ordinary Euclidean Hessian of $\widetilde\varphi_V(Z,W)=\langle V,\widetilde \varphi(Z,W)\rangle$, given by
\begin{align}
    \nabla^2\widetilde\varphi_V(Z,W)[\dot Z,\dot W] = \begin{bmatrix} -V_1\dot W^\top, & \dot Z^\top V_1\end{bmatrix}.
    \label{eq:desing_hess_chart}
\end{align}
We use the above expressions to show that this lift satisfies {\twoimpliesone} everywhere on $\calM$ in Section~\ref{sec:pf_of_desing}. 

\end{example} 

Next, we illustrate the embedded submanifold case on the following low-dimensional example, which shows that the necessary condition for {\twoimpliesone} in the last implication of Theorem~\ref{thm:2implies1_chain} is not sufficient. 
\begin{example}\label{ex:nec_cond_isnt_suff}
Consider the lift of the unit disk $\calX$ in $\RR^2$ to 
\begin{align*}
    \calM = \{y\in\RR^3:y_1^2+y_2^2+y_3^4=1\},\qquad \varphi(y)=(y_1,y_2).
\end{align*} 
Note that $\calM$ is an embedded submanifold of $\RR^3$ with defining function $h(y)=y_1^2+y_2^2+y_3^4-1$, since $\nabla h(y)\neq0$ for all $y\in\calM$. The first two derivatives of $h$ are
\begin{align*}
    \D h(y)[\dot y] = 2y_1\dot y_1 + 2y_2\dot y_2 + 4y_3^3\dot y_3,\qquad \D^2 h(y)[\dot y,\dot y] = 2\dot y_1^2 + 2\dot y_2^2 + 12y_3^2\dot y_3^2.
\end{align*}
Let $y=(1,0,0)$ and $x=\varphi(y)=(1,0)$. We get from~\eqref{eq:1st_and_2nd_ord_tngnts} that
\begin{align*}
    \T_y\calM = \{\dot y\in\RR^3:\dot y_1 = 0\},\qquad \T^2_{y,\dot y}\calM = (-\dot y_2^2,0,0)+\T_y\calM.
\end{align*}
Because $\varphi$ extends to a linear map $\overline{\varphi}(y)=(y_1,y_2)$ on all of $\RR^3$, whose first two derivatives are $\D\overline{\varphi}(y)[\dot y]=(\dot y_1,\dot y_2)$ and $\D^2\overline{\varphi}(y)[\dot y,\dot y]=0$, we have from~\eqref{eq:L_and_Q_embedded} that
\begin{align*}
    \Lmap_y(\dot y)=(0,\dot y_2) && \textrm{ and } && \Qmap_y(\dot y) = (-\dot y_2^2,0).
\end{align*}
On the other hand, 
\begin{align*}
    \T_x\mathcal{X} = \{\dot x\in\RR^2:\dot x_1\leq 0\} \supsetneq \im\Lmap_y = \{\dot x\in\RR^2:\dot x_1 = 0\}. 
\end{align*}
Since $\im\Lmap_y\neq\T_x\calX$, {\oneimpliesone} does not hold at $y$. 
\ifthenelse{\boolean{shortver}}{All the sufficient conditions for {\twoimpliesone} in Theorem~\ref{thm:2implies1_chain} fail as well, see~\cite[Ex.~3.31]{levin2022effectARXIV}.}{
Checking our sufficient conditions, Proposition~\ref{prop:2implies1_with_Qmap} gives
\begin{align*}
    A_y = \Qmap_y(\mathrm{ker}(\Lmap_y)) + \im\Lmap_y = \im\Lmap_y,\qquad B_y = \bigcup_{\substack{\Lmap_y(\dot y_i)\to 0\\ (\dot y_i)_1=0}}\lim_i(\Qmap_y(\dot y_i) + \im\Lmap_y) = \im\Lmap_y,
\end{align*} 
where the last equality follows since $\Lmap_y(\dot y_i) = (0,(\dot y_i)_2)\to 0$ implies $\Qmap_y(\dot y_i) = (-(\dot y_i)_2^2,0)\to0$.
Thus, all sufficient conditions in Theorem~\ref{thm:2implies1_chain} fail. 
}

For the equivalent condition in Theorem~\ref{thm:2implies1_chain}, if $w\in A_y^*$, or equivalently, $w=(w_1,0)$, then $\nabla^2\varphi_w(y)$ is the unique self-adjoint operator on $\T_y\calM$ satisfying
\begin{align*}
    \langle\nabla^2\varphi_w(y)[\dot y],\dot y\rangle = \langle w,\mathbf{Q}_y(\dot y)\rangle = -w_1\dot y_2^2, && \textrm{ hence } && \nabla^2\varphi_w(y)[\dot y] = (0, -w_1\dot y_2,0).
\end{align*}
For $u\in \ker\Lmap_y$, or equivalently, $u = (0,0,u_3)$, we get $\langle \nabla^2\varphi_w(y)[u],u\rangle = 0$ and $\nabla^2\varphi_w(y)[u] = 0$. Proposition~\ref{prop:2implies1_with_Qmap} then shows that $W_y = A_y^*\not\subseteq (\T_x\mathcal{X})^*$, hence {\twoimpliesone} does not hold.

Nevertheless, the necessary condition in Theorem~\ref{thm:2implies1_chain} does hold. Indeed,
\begin{align*}
    \Qmap_y(\T_y\calM) + \im\Lmap_y = \{(-\dot y_2^2,0):\dot y_2\in\RR\} + \{(0,\dot y_2):\dot y_2\in\RR\} = \T_x\calX,
\end{align*}
so taking duals on both sides yields the desired condition.
\end{example}

\subsection{Low-rank PSD matrices, and computing tangent cones}\label{sec:low_rk_psd}
As an application of the theory we developed so far, we consider the set of bounded-rank PSD matrices 
\begin{align*}
    \calX = \{X\in\RR^{n\times n}: X^\top = X,\ X\succeq0,\ \rank(X)\leq r\} = \mbb S_{\succeq0}^n\cap\RR^{n\times n}_{\leq r},
\end{align*}
with $r<n$ together with its lift to $\calM=\RR^{n\times r}$ via the factorization map $\varphi(R)=RR^\top$. 
This is a special case of the Burer--Monteiro lift for SDPs~\eqref{eq:BM_lift} without constraints ($m=0$).
Constructions in the next section enable us to deduce the properties of the general lift from this special case.
\begin{proposition}\label{prop:BM_lift_noA}
The lift $\varphi(R)=RR^\top$ from $\calM=\RR^{n\times r}$ to $\calX = \RR^{n\times n}_{\leq r}\cap\mbb S_{\succeq0}^n$ satisfies
\begin{itemize}
    \item {\localimplieslocal} and {\twoimpliesone} everywhere, and
    \item {\oneimpliesone} at $R\in\calM$ if and only if $\rank(R)=r$.
\end{itemize} 
Moreover, with $X = RR\transpose$, the sufficient condition $A_R = \T_{X}\calX$ for {\twoimpliesone} holds everywhere, and
\begin{align*}
    \T_X\calX &= \T_X\RR^{n\times n}_{\leq r}\cap\T_X\mbb S_{\succeq0}^n\\ &= \{V\in\mbb S^n: V_{\perp}\succeq0,\ \rank(V_{\perp})\leq r-\rank(X) \textrm{ where } V_{\perp} = \Proj_{\mathrm{col}(X)^{\perp}}V\Proj_{\mathrm{col}(X)^{\perp}}\}\\
    &= \left\{U\begin{pmatrix} V_1 & V_2\\ V_2^\top & V_3\end{pmatrix}U^\top : V_1\in\mbb S^{\rank(X)},\ V_3\succeq0,\ \rank(V_3)\leq r-\rank(X)\right\},
\end{align*}
where in the last line $U\in O(n)$ is an eigenmatrix for $X$ satisfying $X=U\Sigma U^\top$ with $\Sigma\in\RR^{n\times n}$ diagonal containing the eigenvalues of $X$ in descending order.
\end{proposition}
\begin{proof}
\ifthenelse{\boolean{shortver}}{The {\localimplieslocal} property for this lift was proved in~\cite[Prop.~2.3]{burer2005local} by (in our terminology) establishing SLP (the condition from Remark~\ref{rmk:ESLP}).}
{For {\localimplieslocal}, we sketch the proof from~\cite{burer2005local} of (in our terminology) SLP, the condition from Remark~\ref{rmk:ESLP} that is equivalent to {\localimplieslocal} by Theorem~\ref{thm:openmapsonmanifolds}. Fix $R\in\calM$ and $X=RR^\top$. For any sequence $(X_i)_{i\geq 1}\subseteq\calX$ converging to $X$, let $X_i=U_i\Sigma_i U_i^\top$ be a size-$r$ eigendecomposition for each $X_i$ (so $\Sigma_i\in\RR^{r\times r}$ is a diagonal matrix of eigenvalues of $X_i$, possibly including zeros). Let $R_i = U_i\Sigma_i^{1/2}$ and note that $\|R_i\|=\|X_i\|^{1/2}$ which is bounded, hence after passing to a subsequence we may assume that $\lim_iR_i=\widetilde R$ exists. By continuity of $\varphi$, we have $X=\lim_iX_i=\lim_i\varphi(R_i) = \varphi(\lim_iR_i)=\widetilde R\widetilde R^\top$, hence $\widetilde R\widetilde R^\top = RR^\top$. By~\cite[Lemma~2.5]{burer2005local}, there is an orthogonal $Q\in O(r)$ satisfying $R=\widetilde RQ$, so $(R_iQ)_{i\geq 1}\subseteq\calM$ is a sequence converging to $R$ satisfying $\varphi(R_iQ)=X_i$. This proves SLP holds.}

    For $V\in\RR^{n\times n}$ and $X=RR^\top\in\calX$, define 
    \begin{align*}
        V_{\perp} \coloneqq \Proj_{\ker(X)}V\Proj_{\ker(X)} = \Proj_{\mathrm{col}(R)^{\perp}}V\Proj_{\mathrm{col}(R)^{\perp}},
    \end{align*}
    where we used the fact that $\mathrm{col}(R)=\mathrm{col}(X) = \mathrm{ker}(X)^{\perp}$.
    The tangent cone at $X$ to $\mbb S_{\succeq0}^n$ is given by \cite[Eq.~(9)]{hiriart2012fresh}
    \begin{align}
        \T_X\mbb S_{\succeq0}^n &= \{V\in\mbb S^n: \langle Vu,u\rangle\geq 0,\ \textrm{for all } u\in\mathrm{ker}(X)\} = \{V\in\mbb S^n:V_{\perp}\succeq0\} \nonumber\\
        &= \left\{U\begin{pmatrix} V_1 & V_2\\ V_2^\top & V_3\end{pmatrix}U^\top:V_1\in\RR^{\rank(X)\times\rank(X)},\ V_3\succeq0\right\},
        \label{eq:tangent_cone_PSD}
    \end{align}
    and the tangent cone at $X$ to $\RR^{n\times n}_{\leq r}$ is given by~\cite[Thm.~3.2]{schneider2015convergence}
    \begin{align*}
        \T_X\RR^{n\times n}_{\leq r} &= \{V\in\RR^{n\times n}: \rank(V_{\perp})\leq r-\rank(X)\}\\
        &= \left\{U\begin{pmatrix} V_1 & V_2\\ \widetilde V_2 & V_3\end{pmatrix} U^\top: V_1\in\RR^{\rank(X)\times\rank(X)},\ \rank(V_3)\leq r-\rank(X)\right\}. 
    \end{align*}
    Hence the intersection $\T_X\RR^{n\times n}_{\leq r}\cap\T_X\mbb S_{\succeq0}^n$ is given by the claimed expression. Furthermore, the tangent cone to an intersection is always included in the intersection of the tangent cones, which follows easily from Definition~\ref{def:tangentcone}. Hence $\T_X\calX\subseteq \T_X\RR^{n\times n}_{\leq r}\cap\T_X\mbb S_{\succeq0}^n$ and it suffices to show the reverse inclusion. We do so simultaneously with proving {\twoimpliesone}.
    
    Since $\calM$ is a linear space, the expressions~\eqref{eq:L_and_Q_on_chart} with the identity chart give
    \begin{align*}
        &\Lmap_R(\dot R) = \D\varphi(R)[\dot R] = R\dot R^\top + \dot R R^\top,\\
        &\Qmap_R(\dot R) = \D^2\varphi(R)[\dot R,\dot R] = 2\dot R\dot R^\top. 
    \end{align*}
    Therefore,
    \begin{align}
        \im\Lmap_R = \{V\in\mbb S^n: V_{\perp}  = 0\}.
        \label{eq:BM_lift_imLmap}
    \end{align}
    Indeed, if $V\in\im\Lmap_R$ then $V=R\dot R^\top + \dot R R^\top$ for some $\dot R\in\RR^{n\times r}$ and hence $V_{\perp}=0$ since $\Proj_{\mathrm{col}(R)^{\perp}}R=0$, while if $V\in\mbb S^n$ satisfies $V_{\perp} = 0$ then 
    \begin{align*}
        V = \Proj_{\mathrm{col}(R)}V + V\Proj_{\mathrm{col}(R)} - \Proj_{\mathrm{col}(R)}V\Proj_{\mathrm{col}(R)} = \Lmap_R\left(VR^{\dagger\top} - \frac{1}{2}RR^{\dagger}VR^{\dagger\top}\right)\in\im\Lmap_R,
    \end{align*}
    where $RR^{\dagger} = \Proj_{\mathrm{col}(R)}$.
    Furthermore, we have 
    \begin{align*}
        \Qmap_R(\ker\Lmap_R) \supseteq \{V\in\mbb S^n_{\succeq0}: \rank(V)\leq r-\rank(X)\}.
    \end{align*}
    Indeed, if $V\in\mbb S^n_{\succeq0}$ with $\rank(V)\leq r-\rank(X)$, let $V=2\dot R_0^{} \dot R_0\transpose$ be a (rescaled) Cholesky decomposition with $\dot R_0\in\RR^{n\times r}$, so $\rank(\dot R_0)=\rank(V)$. Since 
    \begin{align*}
        \dim\mathrm{col}(R^\top)=\rank(R^\top)=\rank(R)=\rank(X)\leq r-\rank(\dot R_0)=\dim\mathrm{ker}(\dot R_0),   
    \end{align*}
    there is an orthogonal matrix $Q\in O(r)$ satisfying $Q\mathrm{col}(R^\top)\subseteq\mathrm{ker}(\dot R_0)$. Let $\dot R=\dot R_0Q$ and note that $V=2\dot R\dot R^\top$ so $\Qmap_R(\dot R)=V$, and $\dot RR^\top = 0$ so $\dot R\in\ker\Lmap_R$.
    Thus, with Proposition~\ref{prop:2implies1_with_Qmap}(a),
    \begin{align*}
        A_R = \Qmap_R(\ker\Lmap_R) + \im\Lmap_R \supseteq \T_X\RR^{n\times n}_{\leq r}\cap\T_X\mbb S_{\succeq0}^n.
    \end{align*}
    On the other hand, by Proposition~\ref{prop:AB_basics}(b), we have $A_R\subseteq \T_X\calX$. Thus, we have the chain of inclusions
    \begin{align*}
        \T_X\RR^{n\times n}_{\leq r}\cap\T_X\mbb S_{\succeq0}^n\subseteq \Qmap_R(\ker\Lmap_R) + \im\Lmap_R \subseteq \T_X\calX \subseteq \T_X\RR^{n\times n}_{\leq r}\cap\T_X\mbb S_{\succeq0}^n,
    \end{align*}
    so all the above inclusions are equalities. In particular, we obtain the claimed expression for $\T_X\calX$ and {\twoimpliesone} everywhere on $\calM$ by Theorem~\ref{thm:2implies1_chain}. Our claims about {\oneimpliesone} follow from~\eqref{eq:BM_lift_imLmap} and Theorem~\ref{thm:oneimpliesone_char}.
\end{proof}
%
%
Finding an explicit expression for tangent cones can be difficult in general. 
In Proposition~\ref{prop:BM_lift_noA}, the set $\calX$ was an intersection of two sets whose tangent cones are known, namely $\RR^{n\times n}_{\leq r}$ and $\mbb S_{\succeq0}^n$, which gave us an inclusion $\T_X\calX\subseteq \T_X\RR^{n\times n}_{\leq r}\cap\T_X\mbb S_{\succeq0}^n$. 
However, the tangent cone to an intersection can be strictly contained in the intersection of the tangent cones.\footnote{For example, consider intersecting the circle in the plane with one of its tangent lines.}
The proof of {\twoimpliesone} in Proposition~\ref{prop:BM_lift_noA} proceeds by showing $A_R = \T_X\RR^{n\times n}_{\leq r}\cap\T_X\mbb S_{\succeq0}^n$, which gives $\T_X\calX = \T_X\RR^{n\times n}_{\leq r}\cap\T_X\mbb S_{\succeq0}^n = A_R$ because $A_R\subseteq\T_X\calX$ by Proposition~\ref{prop:AB_basics}(b). 
This simultaneously gives us {\twoimpliesone} and an expression for the tangent cone.

This illustrates a more general and, as far as we know, novel technique of getting expressions for the tangent cones using lifts. 
Generalizing the above discussion, if we have an inclusion $\T_x\calX\subseteq S$ for some set $S$ and we are able to prove $A_y\supseteq S$ for some $y\in\varphi^{-1}(x)$, then we must have $\T_x\calX = S$ by Proposition~\ref{prop:AB_basics}(b). 
In this case, we also conclude that {\twoimpliesone} holds at $y$ by Theorem~\ref{thm:2implies1_chain}. 
In Section~\ref{sec:fiber_prod_lifts}, we shall see another setting in which we naturally have a superset for $\T_x\calX$ (see Lemma~\ref{lem:tangent_cone_fiber_prod_inclusion}), and which allows us to derive expressions for $\T_x\calX$ from lifts satisfying {\oneimpliesone} and {\twoimpliesone}. 
\ifthenelse{\boolean{shortver}}{}{
If $\calX$ is defined by polynomial equalities and inequalities, 
a particular superset of the tangent cone, called the algebraic tangent cone, can be computed using Gr\"obner bases~\cite[\S9.7]{cox2013ideals}.}

A general condition implying that the tangent cone to an intersection is the intersection of the tangent cones is given in~\cite[Thm.~6.42]{rockafellar2009variational}. That condition does not apply to $\calX=\RR^{n\times n}_{\leq r}\cap\mbb S_{\succeq0}^n$ because $\RR^{n\times n}_{\leq r}$ is not Clarke-regular in the sense of~\cite[Def.~6.4]{rockafellar2009variational}. Our approach circumvents Clarke regularity, exploiting the existence of an appropriate lift instead. 


\section{Constructing lifts via fiber products}
\label{sec:fiber_prod_lifts}
In this section, we give a systematic construction of lifts for a large class of sets $\calX$. If the resulting lifted space is a smooth manifold, we also give conditions under which the lift satisfies our desirable properties. Moreover, under these conditions we can obtain expressions for the tangent cones to $\calX$. We shall see that several natural lifts, including the Hadamard and Burer--Monteiro lifts from Section~\ref{sec:main_results}, are special cases of this construction. 

Suppose the set $\calX$ is presented in the form
\begin{align*}
    \calX = \{x\in\calE: F(x)\in \mc Z\} = F^{-1}(\mc Z),
\end{align*}
where $\mc Z\subseteq\calE'$ is some subset of a linear space and $F\colon\calE\to\calE'$ is smooth. This form is general---any set $\calX$ can be written in this form by letting $F$ be the identity and $\mc Z=\calX$. However, we shall see that our framework is most useful when $\mc Z$ is a product of simple sets for which we have smooth lifts satisfying desirable properties. For example, any set defined by $k$ smooth equalities $g_i(x)=0$ and $\ell$ smooth inequalities $h_j(x)\geq 0$ can be written in this form by letting $F(x)=(g_1(x),\ldots,g_k(x),h_1(x),\ldots,h_{\ell}(x))$ and $\mc Z=\{0\}^k\times\RR_{\geq 0}^{\ell}$. We can also incorporate semidefiniteness and rank constraints of smooth functions of $x$ by taking Cartesian products of $\mc Z$ with $\Rmnlr$ or $\mbb S^n_{\succeq0}$. 

Suppose now that we have a smooth lift $\psi\colon\mc N\to\mc Z$. We can use this lift of $\mc Z$ to construct a lift of $\calX$ by taking the \emph{fiber product} of $F$ and $\psi$.
\begin{definition}\label{def:fib_prod_lift}
Let $\calX$ be a subset of $\calE$ defined by a smooth map $F \colon \calE \to \calE'$ and a set $\mc Z \subseteq \calE'$ as $\calX = F^{-1}(\mc Z)$. Suppose $\psi\colon\mc N\to\mc Z$ is a smooth lift of $\mc Z$ to the smooth manifold $\mc N$. Then the \emph{fiber product lift} of $\calX$ with respect to $F$ and $\psi$ is $\varphi\colon\mc M_{F,\psi}\to\calX$ where
\begin{align*}
    \calM_{F,\psi} = \{(x,y)\in\calE\times\mc N:F(x) = \psi(y)\} && \textrm{ and } && \varphi(x, y) = x.
\end{align*}
Here $\calM_{F,\psi}$ is the (set-theoretic) \emph{fiber product} of the maps $F\colon\cal E\to\calE'$ and $\psi\colon\mc N\to\calE'$.
\end{definition}
The following commutative diagram illustrates Definition~\ref{def:fib_prod_lift}.
Its top horizontal arrow is the coordinate projection $\pi(x,y) = y$.
\begin{center}
\begin{tikzcd}
            &[-2.7em] {\mathcal{M}_{F,\psi}} \arrow[d, "\varphi"', two heads] \arrow[r, "\pi"] & \mathcal{N} \arrow[d, "\psi", two heads] &[-2em]              \\
\mathcal{E} & \mathcal{X} \arrow[r, "F"'] \arrow[l, phantom, "\supseteq"]                       & \mathcal{Z} \arrow[r, phantom, "\subseteq"]              & \mathcal{E}'
\end{tikzcd}
\end{center}

The fiber product $\calM_{F,\psi}$ need not be a smooth manifold even when both $F$ and $\psi$ are smooth maps between smooth manifolds. 
Accordingly, we make the following assumption:
\begin{assumption}\label{assu:fib_prod_def_eqn} 
The differential of $(x,y)\mapsto F(x)-\psi(y)$ has constant rank in a neighborhood of $\calM_{F,\psi}$ in $\calE\times\mc N$.
\end{assumption}
Assumption~\ref{assu:fib_prod_def_eqn} not only implies $\calM_{F,\psi}$ is a smooth embedded submanifold of $\calE\times\mc N$, but also that $F(x) - \psi(y) = 0$ is a (constant-rank) defining function for it (in the sense of~\cite[Thm.~5.12]{lee_smooth}). 
Under this assumption, the tangent space to $\calM_{F,\psi}$ is given by
\begin{align}
    \T_{(x,y)}\calM_{F,\psi} = \{(\dot x,\dot y)\in\calE\times\T_y\mc N:\D F(x)[\dot x] = \D\psi(y)[\dot y]\}. 
    \label{eq:fib_prod_tangent_cone}
\end{align}

We proceed to give some examples of the above construction. We then study fiber product lifts in general and instantiate our results on these examples.
\begin{example}[Sphere to ball]\label{ex:sphere2ball}
Let $\calX = \{x\in\RR^n:\|x\|_2\leq 1\}$ be the unit Euclidean ball. Let $\mc Z = \RR_{\geq 0}$ and $F(x)=1-x\transpose x$ so $\calX=F^{-1}(\mc Z)$. Let $\psi\colon\RR\to\RR_{\geq 0}$ be the smooth lift $\psi(y)=y^2$. Then $\calM_{F,\psi}=\{(x,y)\in\RR^n\times\RR:1-x\transpose x = y^2\}$, which is just the unit sphere in $\RR^{n+1}$, and $\varphi(x,y)=x$ projects onto the first $n$ coordinates. This lift is used in~\cite[\S2.7]{Phan2019} to apply a solver for quadratic programming over the sphere~\eqref{eq:Q} to quadratic programs over the ball~\eqref{eq:P}.
\end{example}
\begin{example}[Sphere to simplex]\label{ex:sphere2simplex}
The Hadamard lift~\eqref{eq:sphere_to_simplex} from Section~\ref{sec:sphere_to_simplex} can be obtained as a special case of Definition~\ref{def:fib_prod_lift}. 
Indeed, let $\calX = \Delta^{n-1}=\{x\in\RR^n:x\geq 0,\ \sum_{i=1}^nx_i=1\}$ be the standard simplex, let $\mc Z=\RR_{\geq 0}^n\times \{0\}$ and $F(x)=(x,\sum_{i=1}^nx_i-1)$ so $\calX=F^{-1}(\mc Z)$. Let $\psi\colon \RR^n\to \mc Z$ be $\psi(y)=(y^{\odot 2},0)$ where superscript $\odot 2$ denotes entrywise squaring. Then 
\begin{align*}
    \calM_{F,\psi} = \left\{(x,y)\in\RR^n\times\RR^n:x=y^{\odot 2},\ \sum_{i=1}^nx_i=1\right\} = \left\{(x,y)\in\RR^n\times \RR^n:x=y^{\odot 2},\ \|y\|_2=1\right\},
\end{align*}
and $\varphi(x,y)=x$. It is easy to check that the coordinate projection $\pi$ defines a diffeomorphism of $\calM_{F,\psi}$ with $\mathrm{S}^{n-1}$, and that the composition $\varphi\circ\pi^{-1}\colon\mathrm{S}^{n-1}\to \calX$ yields the Hadamard lift~\eqref{eq:sphere_to_simplex} from Section~\ref{sec:sphere_to_simplex}.
By Proposition~\ref{prop:composition_of_lifts}, the fiber product lift of the simplex is equivalent (for the purposes of checking our desirable properties) to the lift~\eqref{eq:sphere_to_simplex}.
\end{example}

\begin{example}[Torus to annulus]\label{ex:torus2annulus}
Let $\calX = \{x\in \RR^n:r_1 \leq \|x\|_2\leq r_2\}$, where we assume $0<r_1<r_2$.
Let $\mc Z=\RR_{\geq0}^2$ and $F(x)=(x\transpose x-r_1^2,r_2^2-x\transpose x)$.
Let $\psi\colon \RR^2\to\mc Z$ be $\psi(y)=y^{\odot 2}$ so
\begin{align*}
    \calM_{F,\psi} &= \{(x,y)\in\RR^n\times\RR^2:x\transpose x-r_1^2=y_1^2,\ r_2^2-x\transpose x=y_2^2\},\\
    &= \left\{(x,y)\in\RR^n\times \RR^2: \|x\|_2 = \sqrt{r_1^2+y_1^2},\ \|y\|_2 = \sqrt{r_2^2-r_1^2}\right\}.
\end{align*}
This is an $n$-dimensional manifold diffeomorphic to $\mathrm{S}^{n-1}\times\mathrm{S}^1$, with diffeomorphism
\begin{align*}
    \Phi(x,y) = \begin{bmatrix} (r_1^2+y_1^2)^{-1/2}x, & (r_2^2-r_1^2)^{-1/2}y\end{bmatrix}.
\end{align*}
Viewed differently, the equivalent (by Proposition~\ref{prop:composition_of_lifts}) lift $\varphi\circ\Phi^{-1}$ is the composition 
\begin{align*}
    \mathrm{S}^{n-1}\times \mathrm{S}^1\to\mathrm{S}^{n-1}\times\Delta^1\to \calX,
\end{align*}
where the first map is the Hadamard lift from the sphere to the simplex from the preceding example, and the second map is $(y,\theta)\mapsto \sqrt{\theta_1 r_1^2 + \theta_2r_2^2}y$.
If $n=2$, then $\calX$ is an annulus and $\calM$ is a torus. 
\end{example}
\begin{example}[Smooth SDPs]\label{ex:smth_sdps}
The Burer--Monteiro lift~\eqref{eq:BM_lift} from Section~\ref{sec:smth_sdps} is also a special case of Definition~\ref{def:fib_prod_lift}.
To see this, in the notation of Section~\ref{sec:smth_sdps}, let
\begin{align*}
    \mc Z = (\mbb S_{\succeq0}^n\cap\RR^{n\times n}_{\leq r})\times\{0\}^m, &&
    F(X)=(X,\langle A_1,X\rangle-b_1,\ldots,\langle A_m,X\rangle-b_m),
\end{align*}
and $\psi(R) = (RR\transpose,0,\ldots,0)$ defined on $\mc N = \RR^{n\times r}$. In this case,
\begin{align*}
    \calM_{F,\psi} = \{(X,R)\in \mbb S^n\times \RR^{n\times r}: X=RR^\top, \langle A_iR,R\rangle = b_i\},
\end{align*}
which the projection $\pi$ maps diffeomorphically onto the set $\calM$ in~\eqref{eq:BM_lift}. 
Furthermore, Assumption~\ref{assu:fib_prod_def_eqn} is equivalent in this case to the assumption in Proposition~\ref{prop:BM_lift} that $h_i(R)=\langle A_iR,R\rangle-b_i$ are local defining functions. 
Thus, proving Proposition~\ref{prop:BM_lift} is equivalent to proving the corresponding properties for the fiber product lift above under Assumption~\ref{assu:fib_prod_def_eqn}. 
\end{example}

Now that we have seen several examples of fiber product lifts, we ask: when do desirable properties of the lift $\psi\colon\mc N\to\mc Z$ imply the corresponding properties for the fiber product lift $\varphi\colon\calM_{F,\psi}\to\calX$?
This is answered by the next few propositions.
\begin{proposition}\label{prop:fib_prod_loc}
Under Assumption~\ref{assu:fib_prod_def_eqn}, if $\psi\colon\mc N\to\mc Z$ satisfies {\localimplieslocal} at $y\in\mc N$, then $\varphi\colon\calM_{F,\psi}\to \calX$ satisfies {\localimplieslocal} at $(x,y)\in\calM_{F,\psi}$ for any $x\in \pi^{-1}(y)$.
\end{proposition}
\begin{proof}
By Theorem~\ref{thm:localimplocalcharact}, it is equivalent to show that openness of $\psi$ at $y$ implies openness of $\varphi$ at $(x,y)$. Assumption~\ref{assu:fib_prod_def_eqn} implies that $\calM_{F,\psi}$ is an embedded submanifold of $\calE\times\mc N$, hence its manifold topology coincides with the subspace topology induced from $\calE\times\mc N$. Thus, to show $\varphi$ is open at $(x,y)$, it suffices to show that $\varphi((U\times V)\cap\calM_{F,\psi})$ is open for any open $U\subseteq\calE$ containing $x$ and open $V\subseteq\mc N$ containing $y$, since such sets form a basis for the subspace topology on $\calM_{F,\psi}$. Since $\psi$ is open at $y$, we get that $\psi(V)\subseteq\mc Z$ is open. Since $F$ is continuous, $F^{-1}(\psi(V))\subseteq\calX$ is open. Since $\varphi(x,y)=x$, we have
\begin{align*}
    \varphi((U\times V)\cap\calM_{F,\psi}) &= \{x\in U\cap\calX: \exists\ y\in V \textrm{ s.t. } F(x)=\psi(y)\}\\ &= \{x\in U\cap\calX: F(x)\in \psi(V)\} = (U\cap\calX)\cap F^{-1}(\psi(V)),
\end{align*}
which is open in $\calX$ as the intersection of two open sets. Thus, $\varphi$ is open at $(x,y)$.
\end{proof}
Note that the above proof, and hence the conclusion of Proposition~\ref{prop:fib_prod_loc}, apply more generally whenever $\calM_{F,\psi}$ is endowed with the subspace topology induced from $\calE\times\mc N$ (but is not necessarily a smooth manifold) and when all maps involved are continuous (but not necessarily smooth).

We now turn to studying {\oneimpliesone} and {\twoimpliesone}. Along the way, we give another instance of the technique for finding tangent cones via lifts outlined in Section~\ref{sec:low_rk_psd}. To do so, we begin by giving a superset of the tangent cone, obtained from the fact that $\calX$ is an inverse image~\cite[Thm.~6.31]{rockafellar2009variational}.
As usual, $\D F(x)^{-1}$ denotes the preimage under the differential (which may not be invertible). 
\begin{lemma}\label{lem:tangent_cone_fiber_prod_inclusion}
The following inclusion holds for all $x \in \calX$:
\begin{align*}
    \T_x\calX \subseteq \{\dot x\in\calE:\D F(x)[\dot x]\in \T_{F(x)}\mc Z\} = \D F(x)^{-1}(\T_{F(x)}\mc Z).
\end{align*}
\end{lemma}
\begin{proof}
If $v\in\T_x\calX$ then by Definition~\ref{def:tangentcone} there exist sequences $(x_i)_{i\geq 1}\subseteq\calX$ converging to $x$ and $(\tau_i)_{i\geq 1}\subseteq\RR_{>0}$ converging to zero satisfying $v = \lim_{i\to\infty}\frac{x_i-x}{\tau_i}$. Because $F$ is differentiable at $x$, we have $F(x_i) = F(x) + \D F(x)[x_i-x] + o(\|x_i-x\|),$
so
\begin{align*}
    \D F(x)[v] = \lim_{i\to\infty}\frac{F(x_i)-F(x)}{\tau_i}.
\end{align*}
Since $F(x_i)\in\mc Z$ for all $i$, we conclude that $\D F(x)[v]\in \T_{F(x)}\mc Z$ by Definition~\ref{def:tangentcone}.
\end{proof}

\begin{proposition}\label{prop:1imp1_fiber_prod}
    Let $(x,y)\in\calM_{F,\psi}$.
    Under Assumption~\ref{assu:fib_prod_def_eqn}, if $\psi$ satisfies {\oneimpliesone} at $y\in\mc N$, then $\varphi$ satisfies {\oneimpliesone} at $(x,y)$, and equality holds in Lemma~\ref{lem:tangent_cone_fiber_prod_inclusion}.
\end{proposition}
\begin{proof}
Since $\psi$ satisfies {\oneimpliesone} at $y$, Theorem~\ref{thm:oneimpliesone_char} yields $\im\Lmap_y^{\psi} = \T_{\psi(y)}\mc Z = \T_{F(x)}\mc Z$. 
Assumption~\ref{assu:fib_prod_def_eqn} implies that $\calM_{F,\psi}$ is an embedded submanifold of $\calE\times\mc N$. Since $\varphi$ extends to $\overline{\varphi}(x,y)=x$ defined on all of $\calE\times\mc N$, we get from~\eqref{eq:L_and_Q_embedded} that $\Lmap_{(x,y)}^{\varphi}(\dot x,\dot y)=\dot x$ for all $(\dot x,\dot y)\in\T_{(x,y)}\calM_{F,\psi}$. By~\eqref{eq:fib_prod_tangent_cone},
\begin{align*}
    \im\Lmap_{(x,y)}^{\varphi} = \D F(x)^{-1}(\im\Lmap_y^{\psi}) = \D F(x)^{-1}(\T_{F(x)}\mc Z).
\end{align*}
Using Lemma~\ref{lem:tangent_cone_fiber_prod_inclusion} and Proposition~\ref{prop:AB_basics}(b), we get the chain of inclusions 
\begin{align*}
    \T_x\calX\subseteq \D F(x)^{-1}(\T_{F(x)}\mc Z) = \im\Lmap_{(x,y)}^{\varphi} \subseteq\T_x\calX.   
\end{align*}
We conclude that all these sets are equal and hence that {\oneimpliesone} holds for $\varphi$ at $(x,y)$.
\end{proof}

\begin{proposition}\label{prop:fib_prod_2imp1}
Under Assumption~\ref{assu:fib_prod_def_eqn}, if $\psi$ satisfies the sufficient condition $A_y^{\psi} = \T_{\psi(y)}\mc Z$ for {\twoimpliesone} at $y\in\mc N$ (recall Theorem~\ref{thm:2implies1_chain}), then $\varphi$ satisfies the sufficient condition $A_{(x,y)}^{\varphi}=\T_x\calX$ for {\twoimpliesone} at $(x,y)\in\calM_{F,\psi}$, and equality holds in Lemma~\ref{lem:tangent_cone_fiber_prod_inclusion}.
\end{proposition}
\begin{proof}
By Lemma~\ref{lem:tangent_cone_fiber_prod_inclusion}, we always have $\T_x\calX \subseteq \D F(x)^{-1}(\T_{F(x)}\mc Z)$. For the reverse inclusion and the desired sufficient condition for {\twoimpliesone}, it suffices to prove that $\D F(x)^{-1}(\T_{F(x)}\mc Z)\subseteq A_{(x,y)}^{\varphi}$ since $A_{(x,y)}^{\varphi}\subseteq\T_x\calX$ by Proposition~\ref{prop:AB_basics}(b).

Suppose $\D F(x)[\dot x]\in\T_{F(x)}\mc Z$.
By hypothesis, $\T_{F(x)}\mc Z = \T_{\psi(y)}\mc Z = A_y^{\psi}$,
so by Proposition~\ref{prop:2implies1_with_Qmap}(a):
\begin{align}
    \D F(x)[\dot x] = \Qmap_y^{\psi}(v) + \Lmap_y^{\psi}(u),\quad \textrm{for some } v\in\ker\Lmap_y^{\psi} \textrm{ and } u\in\T_y\mc N.
    \label{eq:fib_prod_2imp1_expr1}
\end{align}
Because $v\in\ker\Lmap_y^{\psi}$, we have $(0,v)\in\T_{(x,y)}\calM_{F,\psi}$ by~\eqref{eq:fib_prod_tangent_cone}.
Let $t \mapsto c(t) = (c_x(t),c_y(t))$ be a smooth curve on $\calM_{F,\psi}$ passing through $(x,y)$ with velocity $(0,v)$ at $t = 0$.
Because $F(c_x(t))=\psi(c_y(t))$ for all $t$ near 0, differentiating this expression twice we get 
\begin{align*}
    \D^2 F(x)[c_x'(0), c_x'(0)] + \D F(x)[c_x''(0)] = (\psi\circ c_y)''(0).
\end{align*}
The first term vanishes since $c_x'(0) = 0$.
Using Definition~\ref{def:LQ_maps} for $\Qmap_y^{\psi}$ with Lemma~\ref{lem:diff_of_accel}, we obtain
\begin{align}
    \D F(x)[c_x''(0)]=\Qmap_y^{\psi}(v) + \Lmap_y^{\psi}(u'),\quad \textrm{for some } u'\in\T_y\mc N.
    \label{eq:fib_prod_2imp1_expr2}
\end{align}
Subtracting~\eqref{eq:fib_prod_2imp1_expr2} from~\eqref{eq:fib_prod_2imp1_expr1} yields (using Definition~\ref{def:LQ_maps} for $\Lmap_y^{\psi}$ and~\eqref{eq:fib_prod_tangent_cone} for $\T_{(x,y)}\calM_{F,\psi}$)
\begin{align*}
    \D F(x)[\dot x-c_x''(0)] = \Lmap_y^{\psi}(u-u') = \D\psi(y)[u-u'], && \textrm{ hence } && (\dot x-c_x''(0),u-u')\in\T_{(x,y)}\calM_{F,\psi}.
\end{align*}
Since $\D\varphi(x, y)[\dot x, \dot y] = \dot x$, it follows that $\dot x-c_x''(0)\in\im\Lmap_{(x,y)}^{\varphi}$.
By Definition~\ref{def:LQ_maps} and Lemma~\ref{lem:diff_of_accel} again, there exists $w\in\T_{(x,y)}\calM_{F,\psi}$ satisfying
\begin{align*}
    c_x''(0) + \Lmap_{(x,y)}^{\varphi}(w) = \Qmap_{(x,y)}^{\varphi}(0,v) \in \Qmap_{(x,y)}^{\varphi}(\ker\Lmap_{(x,y)}^{\varphi}),
\end{align*} 
from which we see $\dot x\in \Qmap_{(x,y)}^{\varphi}(\ker\Lmap_{(x,y)}^{\varphi})+\im\Lmap_{(x,y)}^{\varphi} = A_{(x,y)}^{\varphi}$ (again with Proposition~\ref{prop:2implies1_with_Qmap}(a)). 
\end{proof}
We remark that other sufficient conditions for equality in Lemma~\ref{lem:tangent_cone_fiber_prod_inclusion} to be achieved are given in~\cite[Exer.~6.7, Thm.~6.31]{rockafellar2009variational}. However, they do not apply to Example~\ref{ex:smth_sdps} ($\mc Z$ is not Clarke-regular and $\D F(X)$ may not be surjective). In contrast, our approach via lifts does apply to this example, and gives {\twoimpliesone} and an expression for the tangent cones simultaneously, see Corollary~\ref{cor:fiber_prod_exs_revisit} below.

As the examples in the beginning of this section illustrate, $\mc Z$ is often a product of sets. It is therefore useful to note that a product of lifts satisfying desirable properties also satisfies those properties:
\begin{proposition}\label{prop:prod_lifts}
Suppose $\mc Z_i\subseteq\calE_i$ for $i=1,\ldots,k$ are subsets admitting smooth lifts $\psi_i\colon\mc N_i\to\mc Z_i$. Let $\mc Z=\mc Z_1\times\cdots\times\mc Z_k$ and $\psi=\psi_1\times\cdots\times\psi_k\colon\mc N_1\times\cdots\times\mc N_k\to\mc Z$, which is a smooth lift of $\mc Z$. Then the following hold.
\begin{enumerate}[(a)]
    \item $\T_{z}\mc Z \subseteq \T_{z_1}\mc Z_1\times\cdots\times\T_{z_k}\mc Z_k$ (the inclusion may be strict, see~\cite[Prop.~6.41]{rockafellar2009variational}).
    \item $\psi$ satisfies {\localimplieslocal} at $y$ if and only if $\psi_i$ satisfies {\localimplieslocal} at $y_i$ for all $i$.
    \item We have $\im\Lmap_{y}^{\psi}=\im\Lmap_{y_1}^{\psi_1}\times\cdots\times\im\Lmap_{y_k}^{\psi_k}$. In particular, $\psi$ satisfies {\oneimpliesone} at $y$ if $\psi_i$ satisfies {\oneimpliesone} at $y_i$ for all $i$, in which case equality in (a) holds. 
    \item We have $\Qmap_{y}^{\psi}\equiv\Qmap_{y_1}^{\psi_1}\times\cdots\times\Qmap_{y_k}^{\psi_k}\mod\im\Lmap_{y}^{\psi}$. Moreover, $A_{y}^{\psi}=A_{y_1}^{\psi_1}\times\cdots\times A_{y_k}^{\psi_k}$ and likewise for $B_{y}^{\psi}$ and $W_{y}^{\psi}$. In particular, $\psi$ satisfies {\twoimpliesone} at $y$ if $\psi_i$ satisfies {\twoimpliesone} at $y_i$ for all $i$. 
    \item $\psi$ satisfies the sufficient condition $A_{y}^{\psi}=\T_{\psi(y)}\mc Z$ for {\twoimpliesone} if $\psi_i$ satisfies the corresponding conditions $A_{y_i}^{\psi_i}=\T_{\psi_i(y_i)}\mc Z_i$ for all $i$.
\end{enumerate}
\end{proposition}
\ifthenelse{\boolean{shortver}}{The proof is straightforward, see~\cite[App.~C.2]{levin2022effectARXIV}.}
{The proof is given in Appendix~\ref{apdx:prods}. 
Equality in Proposition~\ref{prop:prod_lifts}(a) is achieved when each $\mc Z_i$ is Clarke-regular at $z_i$~\cite[Prop.~6.41]{rockafellar2009variational}.}
By Remark~\ref{rmk:nonunique_Q}, equality of $\Qmap_{y}^{\psi}$ and $\Qmap_{y_1}^{\psi_1}\times\cdots\times\Qmap_{y_k}^{\psi_k}$ modulo $\im\Lmap_{y}^{\psi}$ means that either one can be used to verify {\twoimpliesone}.

We can now revisit the examples from the beginning of this section.
\begin{corollary}\label{cor:fiber_prod_exs_revisit}
The lifts in Examples~\ref{ex:sphere2ball}, \ref{ex:sphere2simplex}, \ref{ex:torus2annulus} and \ref{ex:smth_sdps} satisfy the following.
\begin{itemize}
    \item The sphere to ball lift in Example~\ref{ex:sphere2ball} satisfies {\localimplieslocal} everywhere, {\oneimpliesone} at $y$ if and only if $y\neq 0$ (i.e., at preimages of the interior of the ball), and {\twoimpliesone} everywhere.
    
    \item The sphere to simplex lift~\eqref{eq:sphere_to_simplex} satisfies {\localimplieslocal} everywhere, {\oneimpliesone} at $y$ if and only if $y_i\neq 0$ for all $i$ (i.e., at preimages of the relative interior of the simplex), and {\twoimpliesone} everywhere.
    
    \item The lift of the annulus in Example~\ref{ex:torus2annulus} satisfies {\localimplieslocal} everywhere, {\oneimpliesone} at $y$ if and only if $y_1,y_2\neq0$ (i.e., at preimages of the interior of the annulus), and {\twoimpliesone} everywhere. 
    
    \item The Burer--Monteiro lift~\eqref{eq:BM_lift} under the smoothness assumption satisfies {\localimplieslocal} everywhere, {\oneimpliesone} at $R$ if and only if $\rank(R)=r$ (i.e., at preimages of points of rank $r$), and {\twoimpliesone} everywhere. Moreover, we get an expression for the tangent cones to $\calX$:
    \begin{align}
        \T_X\calX = \{V\in\mbb S^n: V\in \T_X\mbb S_{\succeq0}^n\cap\T_X\RR_{\leq r}^{n\times n} \textrm{ and } \inner{A_i}{V} = 0 \textrm{ for all } i \}.
        \label{eq:smth_sdp_tangent_cone}
    \end{align}
\end{itemize}
\end{corollary}
In particular, this proves Propositions~\ref{prop:sphere_to_simplex} and~\ref{prop:BM_lift}.
Note that an expression for $\T_X\mbb S_{\succeq0}^n\cap\T_X\RR_{\leq r}^{n\times n}$ is derived in Proposition~\ref{prop:BM_lift} (incidentally, also as a consequence of the sufficient condition for {\twoimpliesone} used in Proposition~\ref{prop:fib_prod_2imp1}). 

\begin{proof}
For the first three bullet points, consider the lift $\psi(y)=y^2$ from $\mc N=\RR$ to $\mc Z=\RR_{\geq 0}$. Observe that it satisfies {\localimplieslocal} everywhere, {\oneimpliesone} at $y\neq 0$, and the sufficient condition $A_y=\T_{\psi(y)}\mc Z$ for {\twoimpliesone} at $y=0$. 
Indeed, at $y\neq 0$ we have $\Lmap_y(\dot y)=2y\dot y$ which is an isomorphism of $\T_y\mc N=\RR$ and $\T_{y^2}\mc Z=\RR$; and at $y=0$ we have $\Lmap_y=0$ and $\Qmap_y(\dot y)=2\dot y^2$ by~\eqref{eq:L_and_Q_on_chart} so $A_y=\Qmap_y(\ker\Lmap_y)+\im\Lmap_y=\RR_{\geq 0} = \T_0\mc Z$. Propositions~\ref{prop:fib_prod_loc} and~\ref{prop:1imp1_fiber_prod}--\ref{prop:prod_lifts} imply that the first three lifts satisfy {\localimplieslocal} and {\twoimpliesone} everywhere, and give the claimed ``if'' directions for {\oneimpliesone}. The ``only if'' directions follow from Theorem~\ref{thm:oneimpliesone_char}.

For the Burer--Monteiro lift, consider the lift $\psi(R)=RR^\top$ from $\mc N=\RR^{n\times r}$ to $\mc Z=\mbb S_{\succeq 0}^n\cap\RR^{n\times n}_{\leq r}$. Proposition~\ref{prop:BM_lift_noA} shows that $\psi$ satisfies {\localimplieslocal} and the sufficient condition $A_R = \T_{RR^\top}\mc Z$ for {\twoimpliesone} everywhere, as well as {\oneimpliesone} at points $R$ of rank $r$. Therefore, Propositions~\ref{prop:fib_prod_loc} and~\ref{prop:1imp1_fiber_prod}--\ref{prop:prod_lifts} imply that the Burer--Monteiro lift satisfies {\localimplieslocal} and {\twoimpliesone} everywhere and {\oneimpliesone} at $R$ if $\rank(R)=r$. The {\oneimpliesone} property does not hold at other points by Theorem~\ref{thm:oneimpliesone_char}. Proposition~\ref{prop:fib_prod_2imp1} gives the claimed expression for the tangent cones to $\calX$.
\end{proof}

\begin{example}\label{ex:eigenval_example}
We can now revisit the example from Section~\ref{sec:intro} about computing the smallest eigenvalue of a symmetric matrix $A = U\diag(\lambda)U\transpose$ with $U$ orthogonal.
There, 
\begin{align*}
    \calX = \Delta^{d-1}, &&
    \calM=\mathrm{S}^{d-1} && \textrm{ and } && \varphi(y)=\mathrm{diag}(U^\top yy^\top U).
\end{align*}
Observe that $\varphi(y)=(U^\top y)^{\odot 2}$, which is the composition of the diffeomorphism $y\mapsto U^\top y$ from the sphere to itself and the Hadamard lift from Example~\ref{ex:sphere2simplex}. We conclude that this lift satisfies {\twoimpliesone} everywhere on $\calM$ by Proposition~\ref{prop:composition_of_lifts} and Corollary~\ref{cor:fiber_prod_exs_revisit}. Therefore, any 2-critical point for~\eqref{eq:Q} maps to a stationary point for~\eqref{eq:P}, for any cost $f$. If $f$ is convex, then since $\calX$ is also convex any stationary point for~\eqref{eq:P} is globally optimal. Thus, in this case any 2-critical point for~\eqref{eq:Q} is globally optimal and its nonconvexity is benign. This is well-known for the eigenvalue problem, which corresponds to the case of linear $f$.
\end{example}
\begin{example}\label{ex:stochastic_mats}
    Proposition~\ref{prop:prod_lifts} together with Corollary~\ref{cor:fiber_prod_exs_revisit} implies the properties stated in Proposition~\ref{prop:stochastic_mats} for the lift~\eqref{eq:HadProd} of stochastic matrices. Indeed, the set of stochastic matrices $\calX = \{X\in\RR^{n\times m}_{\geq0}: X^\top \mathbbm{1}_n=\mathbbm{1}_n\}$ is just the product of simplices $\calX=\left(\Delta^{n-1}\right)^m$, and the lift~\eqref{eq:HadProd} is precisely the $m$-fold power lift of~\eqref{eq:sphere_to_simplex}. Thus, Proposition~\ref{prop:prod_lifts}(b)-(d) yields {\localimplieslocal} and {\twoimpliesone} everywhere and {\oneimpliesone} at tuples $(y_i)$ with no zero entries. Furthermore, since simplices are closed convex sets and hence Clarke-regular~\cite[Thm.~6.9]{rockafellar2009variational}, the tangent cone to their product is equal to the product of tangent cones (i.e., equality in Proposition~\ref{prop:prod_lifts}(a) holds) by~\cite[Prop.~6.41]{rockafellar2009variational}, and hence {\oneimpliesone} does not hold elsewhere.
\end{example}


\section{Analysis of low rank lifts}\label{sec:examples}
In this section, we use our theory to prove the remaining results from Section~\ref{sec:main_results} concerning lifts of low rank matrices and tensors.
\ifthenelse{\boolean{shortver}}{Some of the straightforward but technical arguments are omitted here and are given in the arxiv version of this paper~\cite{levin2022effectARXIV}.}{}

\subsection{Proof of Proposition~\ref{prop:LR_lift} ($LR\transpose$ lift)}\label{sec:pf_of_LR}
\ifthenelse{\boolean{shortver}}{The {\localimplieslocal} property at ``balanced'' factorizations $(L,R)$ satisfying $\rank(L)=\rank(R)=\rank(LR^\top)$ was proved in~\cite[Prop.~2.34]{eitan_thesis} by showing that (in our terminology) SLP holds there. We show {\localimplieslocal} does not hold at other pairs $(L,R)$ anywhere else by disproving SLP there via an explicit construction, see the arxiv version for details~\cite[Sec.~5.1]{levin2022effectARXIV}.}
{For {\localimplieslocal}, we follow~\cite[Prop.~2.34]{eitan_thesis} and show SLP holds at every ``balanced'' $(L,R)$ factorizations, i.e., one satisfying $\rank(L)=\rank(R)=\rank(LR^\top)$.\footnote{Compare this argument to the proof of Proposition~\ref{prop:BM_lift}. In both cases, the proof relies on a characterization of the fibers of $\varphi$ or a subset of them.} 
Let $X=LR^\top$ and suppose $(X_i)_{i\in\NN}\subseteq\calX$ converges to $X$. Let $X_i=U_i\Sigma_iV_i^\top$ be a size-$r$ SVD of $X_i$ where $\Sigma_i\in\RR^{r\times r}$ is diagonal with the first $r$ singular values of $X_i$ (possibly including zeros) on the diagonal. Let $L_i=U_i\Sigma_i^{1/2}$ and $R_i=V_i\Sigma_i^{1/2}$. These satisfy $\varphi(L_i,R_i)=L_iR_i^\top=X_i$ and $\|L_i\|=\|R_i\|=\|X_i\|^{1/2}$. 
Since $\|X_i\|$ are bounded, after passing to a subsequence we may assume that the limit $(L_{\infty},R_{\infty})=\lim_i(L_i,R_i)$ exists. 
By continuity of $\varphi$, we must have $L_{\infty}R_{\infty}^\top = X$. 
Since $L_i^\top L_i=R_i^\top R_i=\Sigma_i$ for all $i$, we also have $L_{\infty}^\top L_{\infty} = R_{\infty}^\top R_{\infty}$. This implies $\rank(L_{\infty}) = \rank(R_{\infty})=\rank(X)$ by considering the polar decompositions of $L_{\infty},R_{\infty}$, see~\cite[Lem.~2.33]{eitan_thesis}. 
By~\cite[Lem.~2.32]{eitan_thesis}, there exists $J\in\mathrm{GL}(r)$ satisfying $L = L_{\infty}J$ and $R = R_{\infty}J^{-\top}$. 
Therefore, $(L_iJ,R_iJ^{-\top})$ converges to $(L,R)$ and is a lift of $X_i$, showing that SLP holds. We conclude that $\varphi$ satisfies {\localimplieslocal} at $(L,R)$ such that $\rank(L)=\rank(R)=\rank(LR^\top)$ by Theorem~\ref{thm:openmapsonmanifolds}.

Conversely, we show that if $\rank(L), \rank(R)$ and $\rank(LR^\top)$ are not all equal then {\localimplieslocal} does not hold at $(L,R)$ by constructing a sequence $(X_i)$ converging to $X=LR^\top$ no subsequence of which can be lifted to a sequence converging to $(L,R)$. It always holds that $\rank(LR^\top)\leq\min\{\rank(L),\rank(R)\}$. Assume $\rank(X)<\rank(L)$ (in particular, $\rank(X)<r$). The case $\rank(X)<\rank(R)$ is similar. Define 
\begin{align*}
    L_i=\Proj_{\mathrm{col}(X)}L+i^{-1}L_{\perp},\quad \textrm{and}\quad R_i=\Proj_{\mathrm{row}(X)}R+i^{-1}R_{\perp},
\end{align*}
for $L_{\perp}\in\RR^{m\times r}$ satisfying 
\begin{enumerate}[(a)]
    \item $\mathrm{col}(L_{\perp})\perp\mathrm{col}(X)$;
    \item $\rank(L_{\perp})=r-\rank(X)$.
    \item $\mathrm{col}(L_{\perp})\neq \mathrm{col}(L)\cap\mathrm{col}(X)^{\perp}$;
\end{enumerate} 
The matrix $R_{\perp}\in\RR^{n\times r}$ satisfies the analogous conditions. Such $L_{\perp}$ can be obtained by letting its first column be an arbitrary nonzero vector $v_1\in \mathrm{col}(L)^{\perp}$ (which exists since $\rank(L)\leq r<\min\{m,n\}$), letting the next $r-\rank(X)-1$ columns be linearly independent vectors in $\mathrm{col}(X)^{\perp}\cap\mathrm{span}\{v_1\}^{\perp}$ (which exist because $r-\rank(X)-1 \leq \min\{m,n\} - \mathrm{rank}(X)-1 \leq \dim\mathrm{col}(X)^{\perp}\cap\mathrm{span}\{v_1\}^{\perp}$), and the remaining $\rank(X)$ columns to be zero.

Note that $\rank(L_i)=\rank(R_i)=r$ for all $i$ by conditions (a) and (b). Define $X_i=L_iR_i^\top$ which converges to $X$ as $i\to\infty$, and suppose $(\tilde L_{i_j},\tilde R_{i_j})$ is a lift of a subsequence of $(X_i)$ converging to $(L,R)$. Because $\rank(L_i)=\rank(R_i)=r$, we also have $\rank(X_i)=r$, and there exist $J_{i_j}\in\mathrm{GL}(r)$ such that $(\tilde L_{i_j},\tilde R_{i_j})=(L_{i_j}J_{i_j},R_{i_j}J_{i_j}^{-\top})$, see~\cite[Lem.~2.32]{eitan_thesis}. 
However, we have $\mathrm{col}(\widetilde L_{i_j}) = \mathrm{col}(L_{i_j}) = \mathrm{col}(X)\oplus \mathrm{col}(L_{\perp})$ is constant and not equal to $\mathrm{col}(L)$ by condition (c), contradicting $\widetilde L_{i_j}\to L$.
Thus, no subsequence of $(X_i)$ can be lifted, showing that $\varphi$ does not satisfy {\localimplieslocal} at $(L,R)$ by Theorem~\ref{thm:openmapsonmanifolds}. For an explicit example of a cost $f$ and point $(L,R)$ which is a local minimum for~\eqref{eq:Q} but such that $LR^\top$ is not a local minimum for~\eqref{eq:P}, see~\cite[Prop.~2.30]{eitan_thesis}.}

For {\oneimpliesone} and {\twoimpliesone}, note that $\calM$ is a linear space, hence~\eqref{eq:L_and_Q_on_chart} gives
\begin{align*}
    \Lmap_{(L,R)}(\dot L,\dot R) = \dot LR^\top + L\dot R^\top, && \Qmap_{(L,R)}(\dot L,\dot R) = 2\dot L\dot R^\top. 
\end{align*}
If $\rank(LR^\top)=r$, then~\cite[Prop.~2.15]{eitan_thesis} (which is a slight generalization of the proof of {\oneimpliesone} in Proposition~\ref{prop:BM_lift}) shows $\im\Lmap_{(L,R)}=\T_{LR^\top}\calX$ hence {\oneimpliesone} holds. If $\rank(LR^\top)<r$ then $\T_{LR^\top}\calX$ is not a linear space~\cite[Thm.~2.2]{hosseini2019coneslowrank}, hence {\oneimpliesone} does not hold at $(L,R)$ by Theorem~\ref{thm:oneimpliesone_char}.

To show {\twoimpliesone} holds everywhere on $\calM$, it suffices to show $B_{(L,R)}^*\subseteq (\T_X\calX)^*$ whenever $\rank(LR^\top)<r$ by Theorem~\ref{thm:2implies1_chain}. Since $\rank(LR^\top)<r$ we must have either $\rank(L)<r$ or $\rank(R)<r$, assume the former (the case $\rank(R)<r$ is similar). Then there exists $w\in\RR^r$ such that $Lw=0$ and $\|w\|^2=1$. For any $u\in\RR^m$, $v\in\RR^n$, and $i\in\NN$, let $\dot L_i = i^{-1} uw^\top$ and $\dot R_i = (i/2) vw^\top$. Then
\begin{align*}
    \Lmap_{(L,R)}(\dot L_i,\dot R_i) = i^{-1}uw^\top R^\top \xrightarrow{i\to\infty}0, &&
    \Qmap_{(L,R)}(\dot L_i,\dot R_i) = uv^\top,
\end{align*}
showing that $uv^\top\in B_{(L,R)}$. Thus, $B_{(L,R)}$ contains all rank-1 matrices, showing that $B_y^*=\{0\} = (\T_X\calX)^*$. 

\subsection{Proof of Proposition~\ref{prop:desing_lift} (desingularization lift)}\label{sec:pf_of_desing}
\ifthenelse{\boolean{shortver}}
{We show that {\localimplieslocal} does not hold at $(X,\mc S)\in\calM$ if $\rank(X)<r$ by disproving SLP via an explicit construction, see the arxiv version~\cite[App.~5.2]{levin2022effectARXIV}.}
{We begin by showing {\localimplieslocal} does not hold at $(X,\mc S)\in\calM$ if $\rank(X)<r$. Since $(X,\mc S)\in\calM$, the subspace $\mc S\subseteq\RR^n$ has dimension $n-r$ and $\mc S\subseteq\ker(X)$. We construct a sequence converging to $X$ such that no subsequence of it can be lifted to a sequence converging to $(X,\mc S)$, demonstrating that SLP does not hold at $(X,\mc S)$. Let $X=U\Sigma V^\top$ be a thin SVD for $X$ where $U\in\mathrm{St}(m,\rank(X)),\ V\in\mathrm{St}(n,\rank(X))$. Since $\mc S\subseteq\ker(X)$, we have $\mathrm{col}(V)\subseteq\mc S^{\perp}$. Let $U_{\perp}\in\mathrm{St}(m,r-\rank(X))$ have columns orthogonal to those of $U$, and $V_{\perp}\in\mathrm{St}(n,r-\rank(X))$ have columns orthogonal to those of $V$, such that some column of $V_{\perp}$ is contained in $\mc S$. Define 
\begin{align*}
    X_i = X + i^{-1}U_{\perp}V_{\perp}^\top\in\RR^{m\times n}.
\end{align*}
Note that $X_i\to X$, that $\rank(X_i)=r$ for all $i$, and that $\ker(X_i) = \ker(X)\cap\mathrm{col}(V_{\perp})^{\perp} =:\mc S'\neq\mc S$ for all $i$. Suppose $(X_{i_j},\mc S_{i_j})\in\calM$ is a lift of a subsequence converging to $(X,\mc S)$. Since $\rank(X_i)=r$, we must have $\mc S_{i_j}=\ker(X_{i_j})=\mc S'$ for all $j$. Thus, $\lim_j\mc S_{i_j}=\mc S'\neq\mc S$, a contradiction. Therefore, no subsequence of $(X_i)$ can be lifted, so we conclude that $\varphi$ does not satisfy {\localimplieslocal} at $(X,\mc S)$ by Theorem~\ref{thm:openmapsonmanifolds}.
For an explicit example of a cost $f$ and point $(X,\mc S)\in\calM$ that is a local minimum for~\eqref{eq:Q}, but such that $X$ is not a local minimum for~\eqref{eq:P}, see~\cite[Prop.~2.37]{eitan_thesis}. 

}
We show {\localimplieslocal} does hold at $(X,\mc S)\in\calM$ if $\rank(X)=r$ by showing that {\oneimpliesone} holds there, which suffices by Proposition~\ref{prop:1imp1_implies_locimploc_at_smooth}. 

For {\oneimpliesone} and {\twoimpliesone}, we use the results of Example~\ref{ex:desing_Qmap_comp}. 
\ifthenelse{\boolean{shortver}}{Using the notation of that example, recall that every $(X,\mc S)\in\calM$ is in the image of a chart $\psi(Z,W)$, and that $\Lmap_{(Z,W)}$ and $\Qmap_{(Z,W)}$ in this chart are given by~\eqref{eq:desing_LQ_maps_chart}.}
{Recall from that example that every $(X,\mc S)\in\calM$ is in the image of a chart of the form
\begin{align*}
    \psi(Z,W) = \left(\begin{bmatrix} -ZW, & Z\end{bmatrix}\Pi, \mathrm{col}\!\left(\Pi^\top\begin{bmatrix} I_{n-r}\\ W\end{bmatrix}\right)\right),\qquad (Z,W)\in\RR^{m\times r}\times\RR^{r\times (n-r)}
\end{align*}
for some permutation matrix $\Pi\in\RR^{n\times n}$. Proposition~\ref{prop:composition_of_lifts} implies that our desirable properties hold at $(Z,W)$ iff they hold at $(X,\mc S)$, so it suffices to consider the composed lift 
\begin{align*}
    \widetilde{\varphi}(Z,W)=\begin{bmatrix} -ZW, & Z\end{bmatrix}\Pi =: X,   
\end{align*}
defined on the linear space $\RR^{m\times r}\times\RR^{r\times(n-r)}$. For this lift, we computed $\Lmap_{(Z,W)}$ and $\Qmap_{(Z,W)}$ in~\eqref{eq:desing_LQ_maps_chart}:
\begin{align*}
    \Lmap_{(Z,W)}(\dot Z,\dot W) = \begin{bmatrix} -\dot ZW - Z\dot W, & \dot Z\end{bmatrix}\Pi,\qquad
    \Qmap_{(Z,W)}(\dot Z,\dot W) = \begin{bmatrix} -2\dot Z\dot W, & 0\end{bmatrix}\Pi.
\end{align*}}

Suppose $\rank(X)=r$ and $(X, \mc S)=\psi(Z,W)$. Note that $\mathrm{col}(X)=\mathrm{col}(Z)$, so $\rank(Z)=r$. If $\Lmap_{(Z,W)}(\dot Z,\dot W) = 0$, then $\dot Z=0$ and $Z\dot W=0$. This implies $\dot W=0$ since $Z$ has full column rank. Thus, $\Lmap_{(Z,W)}$ is injective, but since its domain has dimension $(m+n-r)r=\dim\RR^{m\times n}_{=r}$, we conclude that it is an isomorphism. Thus, {\oneimpliesone} holds at $(Z,W)$ by Theorem~\ref{thm:oneimpliesone_char}. If $\rank(X)<r$ then $\T_X\calX$ is not a linear space~\cite[Thm.~2.2]{hosseini2019coneslowrank}, hence {\oneimpliesone} cannot hold for any lift by Theorem~\ref{thm:oneimpliesone_char}.

Suppose $\rank(X)<r$. We show {\twoimpliesone} holds at $(Z,W)$ by showing that $B_{(Z,W)}^*\subseteq(\T_X\calX)^*$. To that end, note that $\rank(Z)=\rank(X)<r$, so there is a unit vector $w\in\RR^r$ satisfying $Zw=0$. Let $u\in\RR^m$ and $v\in\RR^{n-r}$ be arbitrary. For any $i\in\NN$, let $\dot Z_i=i^{-1} uw^\top$ and $\dot W_i=i wv^\top$. Then 
\begin{align*}
    &\Lmap_{(Z,W)}(\dot Z_i,\dot W_i) = \begin{bmatrix} -i^{-1}uw^\top W - i(Zw)v^\top, & i^{-1} uw^\top\end{bmatrix}\Pi\xrightarrow{i\to\infty} 0,\\
    &\Qmap_{(Z,W)}(\dot Z_i,\dot W_i) \equiv \begin{bmatrix} -2uv^\top, & 0\end{bmatrix}\Pi. 
\end{align*}
We conclude that 
\begin{align*}
    B_{(Z,W)}\supseteq (\RR^{m\times (n-r)}_{\leq 1}\times\{0\})\Pi + \im\Lmap_{(Z,W)} \implies B_{(Z,W)}^*\subseteq (\{0\}\times \RR^{m\times r})\Pi\cap(\im\Lmap_{(Z,W)})^{\perp}.
\end{align*}
To characterize $(\im\Lmap_{(Z,W)})^{\perp}$, observe that $V=\begin{bmatrix}V_1, & V_2\end{bmatrix}\Pi\in\RR^{m\times n}$ with $V_1\in\RR^{m\times(n-r)}$ satisfies $V\in(\im\Lmap_{(Z,W)})^{\perp}$ iff the following holds for all $(\dot Z,\dot W)\in\RR^{m\times r}\times\RR^{r\times(n-r)}$:
\begin{align*}
    \langle V,\Lmap_{(Z,W)}(\dot Z,\dot W)\rangle = \langle V_1, -\dot ZW - Z\dot W\rangle + \langle V_2,\dot Z\rangle = \langle \dot Z, V_2-V_1W^\top\rangle - \langle\dot W,Z^\top V_1\rangle = 0.
\end{align*}
This is equivalent to $V_2=V_1W^\top$ and $Z^\top V_1=0$. Thus, if $V_1=0$ then $V=0$, hence $B_{(Z,W)}^*=\{0\}=(\T_X\calX)^*$. This shows {\twoimpliesone} holds at $(Z,W)$, and hence also at $(X,\mc S)$ by Proposition~\ref{prop:composition_of_lifts}. 


\subsection{Multilinear lifts, and tensors}\label{sec:multilinear_lifts}
In this section, we prove two obstructions to {\twoimpliesone} for multilinear lifts, which apply in particular to lifts defined by tensor factorizations and linear neural networks as discussed in Sections~\ref{sec:lowrk_tensors}-\ref{sec:NNs}.
\begin{proposition}\label{prop:multilin_lifts_nonnorm}
Suppose $\varphi\colon\calM\to\calX\subseteq\calE$ is a smooth lift where $\calM\subseteq \calE' = \calE_1\times\cdots\times\calE_d$ is a smooth embedded submanifold of a product of Euclidean spaces $\calE_i$, and $\varphi$ is defined on all of $\calE'$ and is multilinear in its $d$ arguments. If $\calM$ contains a point $(y_1,\ldots,y_d)$ such that $y_i=0$ for three indices $i$, and $0=\varphi(y_1,\ldots,y_d)$ is not an isolated point of $\calX$, then $\varphi$ does not satisfy {\twoimpliesone} at $(y_1,\ldots,y_d)$.
\end{proposition}
\begin{proof}
Let $(y_1,\ldots,y_d)\in\calE'$. Note that if $y_i=0$ for some $i$, then $\varphi(y_1,\ldots,y_d)=0$ by the multilinearity of $\varphi$. 
\ifthenelse{\boolean{shortver}}{Similarly, multilinearity gives $\D\varphi(y_1,\ldots,y_d) = \D^2\varphi(y_1,\ldots,y_d) = 0$ if $y_i=0$ for at least three indices $i$.}{
Also, since $\varphi$ is multilinear and defined on all of $\calE'$, its differential in the ambient Euclidean space is 
\begin{align*}
    \D\varphi(y_1,\ldots,y_d)[\dot y_1,\ldots,\dot y_d] &= \left.\frac{d}{dt}\right|_{t=0}\varphi(y_1+t\dot y_1,\ldots, y_d+t\dot y_d)\\ &= \sum_{i=1}^d\varphi(y_1,\ldots,y_{i-1},\dot y_i,y_{i+1},\ldots,y_d).
\end{align*}
Similarly,
\begin{align*}
    &\D^2\varphi(y_1,\ldots,y_d)[(\dot y_1,\ldots,\dot y_d),(\dot y_1,\ldots,\dot y_d)] = \left.\frac{d}{dt}\right|_{t=0}\D\varphi(y_1+t\dot y_1,\ldots,y_d+t\dot y_d)[\dot y_1,\ldots,\dot y_d]\\ &= 2\sum_{1\leq i<j\leq d}\varphi(y_1,\ldots,y_{i-1},\dot y_i,y_{i+1},\ldots,y_{j-1},\dot y_j,y_{j+1},\ldots,y_d).
\end{align*}
In particular, if $y_i=0$ for at least three indices $i$, then $\D\varphi(y_1,\ldots,y_d) = \D^2\varphi(y_1,\ldots,y_d) = 0$.}
Hence~\eqref{eq:L_and_Q_embedded} gives $\Lmap_{(y_1,\ldots,y_d)}=0$ and $\Qmap_{(y_1,\ldots,y_d)}=0$.
This implies 
\begin{align*}
    (\im\Lmap_{(y_1,\ldots,y_d)})^{\perp}\cap(\Qmap_{(y_1,\ldots,y_d)}(\T_{(y_1,\ldots,y_d)}\calM))^* = \calE.
\end{align*}
The necessary condition for {\twoimpliesone} given by the last implication in Theorem~\ref{thm:2implies1_chain} is satisfied iff $(\T_0\calX)^*=\calE$, or equivalently $\T_0\calX=\{0\}$. This holds if and only if $0$ is an isolated point of $\calX$, since if $(x_i)\subseteq\calX\setminus\{0\}$ is a sequence converging to 0, then after passing to a subsequence $(x_i/\|x_i\|)$ converges and gives a nonzero element of $\T_0\calX$. 
\end{proof}
Proposition~\ref{prop:multilin_lifts_nonnorm} implies that the lifts corresponding to linear neural networks, as well as standard tensor decompositions such as CPD, Tensor Train (TT), and Tucket, all do not satisfy {\twoimpliesone} as points with at least three zero factors.

Proposition~\ref{prop:multilin_lifts_nonnorm} might suggest that failure of {\twoimpliesone} can be avoided by normalizing the arguments of the lift to have unit norm. Specifically, by multilinearity of $\varphi$ we have 
\begin{align*}
    \varphi(y_1,\ldots,y_d) = \left(\prod_{i=1}^d\|y_i\|\right)\varphi\!\left(\frac{y_1}{\|y_1\|},\ldots,\frac{y_d}{\|y_d\|}\right),\quad \textrm{whenever } y_i\neq 0 \textrm{ for all } i.
\end{align*}
Using this observation, one could replace a lift $\varphi\colon\RR^{n_1}\times\cdots\times\RR^{n_d}\to\calX$ to a product of Euclidean spaces by a lift $\psi\colon\RR\times\mathrm{S}^{n_1-1}\times\cdots\times\mathrm{S}^{n_d-1}$ to a product of $\RR$ and several spheres, satisfying $\psi(\lambda,x_1,\ldots,x_d)=\lambda\varphi(x_1,\ldots,x_d)$. Only one factor can be zero in this new lift, so Proposition~\ref{prop:multilin_lifts_nonnorm} does not apply and we might hope that {\twoimpliesone} holds. Unfortunately, this may not resolve the problem as there is another obstruction to {\twoimpliesone} for the following specific form of a lift. 
\begin{proposition}\label{prop:multilinear_lifts_norm}
Suppose $\varphi\colon\calM\to\calX$ is a smooth lift of the form
\begin{align*}
    \varphi(\lambda,Y_1,\ldots,Y_d) = \sum_{i=1}^r\lambda_i\cdot (Y_1)_{:,i}\otimes\cdots\otimes (Y_d)_{:,i},
\end{align*}
where $\calM\subseteq \RR^r\times\RR^{n_1\times r}\times\cdots\times\RR^{n_d\times r}$. Denote $X=\varphi(\lambda,Y_1,\ldots,Y_d)$. If $d\geq 3$ and \begin{align}
    \mathrm{col}(Y_1)^{\perp}\otimes\cdots\otimes\mathrm{col}(Y_d)^{\perp}\not\subseteq(\T_X\calX)^*,
    \label{eq:multilin_col_cond}
\end{align}
then $\varphi$ does not satisfy {\twoimpliesone} at $(\lambda,Y_1,\ldots,Y_d)$ for any $\lambda\in\RR^r$. If $d=2$ and~\eqref{eq:multilin_col_cond} holds, then $\varphi$ does not satisfy {\twoimpliesone} at $(0,Y_1,Y_2)$.
\end{proposition}
\begin{proof}[Proof of Proposition~\ref{prop:multilinear_lifts_norm}]
For any $W\in \mathrm{col}(Y_1)^{\perp}\otimes\ldots\otimes\mathrm{col}(Y_d)^{\perp}$, we have 
\begin{align*}
    \langle W,\D\varphi(\lambda,Y_1,\ldots,Y_d)[\dot\lambda,\dot Y_1,\ldots,\dot Y_d]\rangle = \langle W,\D^2\varphi(\lambda,Y_1,\ldots,Y_d)[(\dot\lambda,\dot Y_1,\ldots,\dot Y_d),(\dot\lambda,\dot Y_1,\ldots,\dot Y_d)]\rangle = 0,
\end{align*}
for all $(\dot Y_1,\ldots,\dot Y_d)$ if $d\geq 3$ or $d=2$ and $\lambda=0$, by multilinearity.
Since $\Lmap_{(\lambda,Y_1,\ldots, Y_d)}$ is the restriction of $\D\varphi(\lambda,Y_1,\ldots,Y_d)$ to $\T_{(\lambda,Y_1,\ldots,Y_d)}\calM$ and $\Qmap_{(\lambda,Y_1,\ldots,Y_d)}$ is given by~\eqref{eq:L_and_Q_embedded}, we get
\begin{align*}
    \mathrm{col}(Y_1)^{\perp}\otimes\ldots\otimes\mathrm{col}(Y_d)^{\perp}\subseteq (\im\Lmap_{(\lambda,Y_1,\ldots,Y_d)})^{\perp}\cap(\Qmap_{(\lambda,Y_1,\ldots,Y_d)}(\T_{(\lambda,Y_1,\ldots,Y_d)}\calM))^*,
\end{align*}
if either $d\geq 3$ or $d=2$ and $\lambda=0$.
Thus, if $\mathrm{col}(Y_1)^{\perp}\otimes\ldots\otimes\mathrm{col}(Y_d)^{\perp}\not\subseteq(\T_X\calX)^*$ then the necessary condition for {\twoimpliesone} from Theorem~\ref{thm:2implies1_chain} does not hold.
\end{proof}
Proposition~\ref{prop:multilinear_lifts_norm} applies in particular to lifts corresponding to symmetric and normalized CP decompositions and ODECO tensors~\cite{odeco}, as well as the SVD lift~\eqref{eq:SVD_lift}.
As discussed in Section~\ref{sec:main_results}, these obstructions to {\twoimpliesone} imply that guarantees for second-order optimization algorithms running on~\eqref{eq:Q} must use the structure in the particular cost function involved. This is particularly significant since our obstructions apply to a broad class of lifts arising naturally in several applications.

\section{Conclusions and future work}\label{sec:concs}
For the pair of problems~\eqref{eq:Q} and~\eqref{eq:P}, we characterized the properties the lift $\varphi\colon\calM\to\calX$ needs to satisfy in order to map desirable points of~\eqref{eq:Q} to desirable points of~\eqref{eq:P}. 
We noted that global minima for~\eqref{eq:Q} always map to global minima for~\eqref{eq:P} (Proposition~\ref{prop:global_min_equiv}), and showed that local minima for~\eqref{eq:Q} map to local minima for~\eqref{eq:P} if and only if $\varphi$ is open (Theorem~\ref{thm:localimplocalcharact}).
We also showed that 1-critical points for~\eqref{eq:Q} map to stationary points for~\eqref{eq:P} if and only if the differential of $\varphi$, viewed as a map from tangent spaces of $\calM$ to tangent cones of $\calX$, is surjective (Theorem~\ref{thm:oneimpliesone_char}).
This requires the tangent cones of $\calX$ to be linear spaces.
We then characterized when 2-critical points for~\eqref{eq:Q} map to stationary points for~\eqref{eq:P}, and gave two sufficient conditions and a necessary condition that may be easier to check for some examples (Theorem~\ref{thm:2implies1_chain}). 
We explained several techniques to compute all quantities involved in these conditions in Section~\ref{sec:Qmap_comp}. 

Using our theory, we studied the above properties for a variety of lifts, including several lifts of low-rank matrices and tensors (Section~\ref{sec:examples}) and the Burer--Monteiro lift for smooth SDPs (Corollary~\ref{cor:fiber_prod_exs_revisit}). 
We also proposed a systematic construction of lifts using fiber products that applies when $\calX$ is the preimage of a smooth function (Section~\ref{sec:fiber_prod_lifts}).
We gave conditions under which it satisfies our desirable properties. 
In some cases, we can also obtain an expression for the tangent cones of $\calX$ simultaneously with {\twoimpliesone}, as explained in Section~\ref{sec:low_rk_psd}. 

We end by listing several future directions suggested by this work. 
\begin{enumerate}[(a)]
    \item \textbf{{\kimpliesone} for general $k$:} Several lifts of interest, notably tensor factorizations with more than two factors, do not satisfy {\twoimpliesone}. It would therefore be interesting to characterize {\kimpliesone} for general $k$, i.e., when do $k$-critical points for~\eqref{eq:Q} map to stationary points for~\eqref{eq:P} for any $k$ times differentiable cost $f$? 
    %
    Do lifts that are multilinear in $k$ arguments, such as order-$k$ tensor lifts, satisfy {\kimpliesone}?
    
    What can be said about ``$k\! \Rightarrow\! \ell$'' for $\ell>1$? Already for $\ell=2$, the second-order optimality conditions on $\calX$ can be involved~\cite[Thm.~3.45]{nonlin_optim}. On the positive side, if {\oneimpliesone} holds at a preimage of a smooth point, then ``$k\! \Rightarrow\! k$'' holds there for all $k\geq 1$ by Proposition~\ref{prop:1imp1_implies_locimploc_at_smooth}.
    
    
    \item \textbf{Robust {\kimpliesone}:} Algorithms run for finitely many iterations in practice, hence can only find approximate $k$-critical points for~\eqref{eq:Q}. It is therefore important to characterize ``robust'' versions of {\kimpliesone}, guaranteeing that approximate $k$-critical points for~\eqref{eq:Q} map to approximate stationary points for~\eqref{eq:P}. 
    Note that if $\calX$ lacks regularity, care is needed when defining approximate stationarity for~\eqref{eq:P}, see~\cite{levin2021finding}.
    
    
    \item \textbf{Obstructions to {\localimplieslocal} and {\kimpliesone}:} 
    For some sets $\calX$, we are not aware of any lifts which satisfy, say, {\localimplieslocal} or {\twoimpliesone}.
    Are there fundamental obstructions which preclude existence of such lifts for those sets and others?
    For example, is there a lift for low-rank tensors satisfying {\twoimpliesone}? Is there a lift for $\Rmnlr$ satisfying {\localimplieslocal}?
    

    \item \textbf{Regularization on the lift:} It is common to modify~\eqref{eq:Q} by adding a regularizer to $g = f\circ\varphi$, see~\cite{kolb2023smoothing,NEURIPS2022_8f41d580,NIPS2004_e0688d13}.
    For example, with the lift $(L, R) \mapsto LR\transpose$, we may regularize~\eqref{eq:Q} by adding $\frac{1}{2}\left(\|L\|_\mathrm{F}^2 + \|R\|_\mathrm{F}^2\right)$, motivated by the fact that its minimum over a fiber $\{ (L, R) : LR\transpose = X \}$ is the nuclear norm $\|X\|_*$~\cite{NIPS2004_e0688d13}.
    Our framework does not directly apply in this case (because the regularizer may not be constant over fibers, hence may not factor through $\varphi$).
    Can it be extended to relate the landscape of the regularized~\eqref{eq:Q} to that of~\eqref{eq:P}?
    
    \item \textbf{Bypassing tangent cones via lifts:} 
    To verify {\oneimpliesone} and {\twoimpliesone} on concrete examples of $\calX$ using the theory in this paper, we need to understand the tangent cones to $\calX$, which is often challenging. Many sets $\calX$ encountered in applications are only defined implicitly via a lift $\varphi\colon\calM\to\calE$. Examples include the set of tensors admitting a certain type of factorization, the set of functions parametrized by a given neural network architecture, and the set of positions and orientations attainable by a robotic arm with a given joint configuration. Are there sufficient conditions for {\kimpliesone} that can be checked using $\varphi$ and $\calM$ alone, without an explicit expression for the tangent cones to $\calX$? 
    
    \item \textbf{Dynamical systems on $\calM$ and their image on $\calX$:} This paper is focused on comparing properties of points on $\calM$ and their images on $\calX$. In contrast, several applications are concerned with properties of entire trajectories of dynamical systems on $\calM$, and it may be interesting to compare these properties with their counterparts for the images of the trajectories on $\calX$. Examples of such comparisons include relating gradient flow on the weights of a neural network to gradient flow in function or measure spaces~\cite{gradflow_NNs1,bach2021gradient,neural_tangent_kernels}, and the ``algorithmic equivalence'' technique used in~\cite{NEURIPS2020_604b37ea,ghai2022non,li2022implicit} to study mirror descent by showing that its continuous-time analogue is equivalent to gradient flow on a reparametrized problem. 
\end{enumerate}

\section*{Acknowledgements}
We thank Christopher Criscitiello and Quentin Rebjock for numerous conversations and comments on drafts of this paper. 
%
%
NB was supported by the Swiss State Secretariat for Education, Research and Innovation (SERI) under contract number MB22.00027.
JK was supported in part by NSF DMS 2309782 and NSF CISE-IIS 2312746.

\bibliographystyle{abbrv}
\bibliography{bibl_lifts}

\appendix

\section{Lifts preserving local minima}
\label{apdx:liftsopen}

We characterize the lifts that map local minima of~\eqref{eq:Q} to local minima of~\eqref{eq:P}. To this end, we introduce a number of properties related to preservation of local minima and then prove that they are all equivalent.
Recall that $\overline{S}$ is our notation for the closure of a set $S$.

\begin{definition}\label{def:preserv_loc_minima}
Let $\varphi \colon \calM \to \calX$ be a continuous, surjective 
map from a topological space $\calM$
to a metric space $\calX$ with distance $\dist$,
and let $x = \varphi(y)$.
\begin{enumerate}
    \item $\varphi$ is \emph{open} at $y$ if $\varphi(U)$ is a neighborhood of $x$ in $\calX$ for all neighborhoods $U$ of $y$ in $\calM$.
    
    \item $\varphi$ is \emph{approximately open} at $y$ if $\overline{\varphi(U)}$ is a neighborhood of $x$ in $\calX$ for all neighborhoods $U$ of $y$ in $\calM$.
    
    \item $\varphi$ satisfies the \emph{Subsequence Lifting Property (SLP)} at $y$ if for every sequence $(x_i)_{i \geq 1} \subseteq \calX$ converging to $x$ there exists a subsequence indexed by $(i_j)_{j\geq 1}$ and a sequence $(y_{i_j})_{j \geq 1} \subseteq \calM$ converging to $y$ such that $\varphi(y_{i_j}) = x_{i_j}$ for all $j \geq 1$.

    \item $\varphi$ satisfies the \emph{Approximate Subsequence Lifting Property (ASLP)} at $y$ if for every sequence $(x_i)_{i > 1} \subseteq \calX$ converging to $x$ and every sequence $(\epsilon_i)_{i \geq 1} \subseteq \RR_{> 0}$ converging to 0 there exists a subsequence indexed by $(i_j)_{j\geq 1}$ and a sequence $(y_{i_j})_{j \geq 1} \subseteq \calM$ converging to $y$ such that $\dist(\varphi(y_{i_j}), x_{i_j}) \leq \epsilon_{i_j}$ for all $j \geq 1$.
\end{enumerate}
\end{definition}

\begin{theorem} \label{thm:openmapsonmanifolds}
    If $\calM$ is Hausdorff, second-countable and locally compact (all of which hold if $\calM$ is a topological manifold), then the four properties of $\varphi$ at $y \in \calM$ in Definition~\ref{def:preserv_loc_minima} are equivalent to each other and to the {\localimplieslocal} property at $y$ (Definition~\ref{def:desirable_lifts}(a)).
\end{theorem}
\begin{proof}
We show that ASLP $\implies$ approximate openness $\implies$ openness $\implies$ SLP $\implies$ \localimplieslocal $\implies$ ASLP.

Suppose $\varphi$ satisfies ASLP at $y$.
Suppose there exists a neighborhood $U$ of $y$ such that $\overline{\varphi(U)}$ is not a neighborhood of $x=\varphi(y)$.
Then we can find a sequence $(x_i)_{i\geq 1}\subseteq\calX$ such that $x_i\to x$ but $x_i \notin \overline{\varphi(U)}$ for all $i$.
Set $\epsilon_i = \frac{1}{2}\dist(x_i, \overline{\varphi(U)}) > 0$ and apply ASLP to find a sequence $(y_i)_{i\geq 1} \subseteq \calM$ such that $y_i\to y$ and $\dist(\varphi(y_i), x_i) \leq \epsilon_i$.
Because $\dist(\varphi(y_i), x_i) < \dist(x_i, \overline{\varphi(U)})$, we have $\varphi(y_i) \notin \overline{\varphi(U)}$ for all $i$.
However, because $U$ is a neighborhood of $y$ and $y_i\to y$, we must have $y_i\in U$ for all large $i$, a contradiction.
Thus, $\overline{\varphi(U)}$ is a neighborhood of $x$, so $\varphi$ is approximately open at $y$.

Suppose $\varphi$ is approximately open at $y$, and let $U$ be a neighborhood of $y$ in $\calM$.
Because $\calM$ is locally compact, we can find a compact neighborhood $V \subseteq U$ of $x$. 
Since $\varphi$ is continuous and $V$ is compact, we have that $\varphi(V)$ is compact; since $\calX$ is Hausdorff (it is a metric space), it follows that $\varphi(V)$ is closed.
Combining with the fact that $\varphi$ is approximately open at $y$, we deduce that $\varphi(V)$ is a neighborhood of $x$.
Since $\varphi(U) \supseteq \varphi(V)$, we conclude that $\varphi(U)$ is a neighborhood of $x$ as well.
Thus, $\varphi$ is open at $y$.

Suppose $\varphi$ is open at $y$, and $(x_j)_{j\geq 1}\subseteq\calX$ converges to $x = \varphi(y)$.
Owing to the topological properties of $\calM$, there is a sequence of open neighborhoods $U_i$ of $y$ with compact closures such that $U_i \supseteq \overline{U_{i+1}}$ and $\bigcap_{i=1}^{\infty} U_i = \{y\}$, see Lemma~\ref{lem:nice_basis} following this proof.
Because $\varphi$ is open, each $\varphi(U_i)$ is an open neighborhood of $x$ such that $\varphi(U_i)\supseteq\varphi(U_{i+1})$ and $x\in\bigcap_{i=1}^{\infty}\varphi(U_i)$.
Moreover, because $\varphi(U_i)$ is a neighborhood of $x$ and $x_j\to x$, there exists index $J(i)$ such that $x_j\in\varphi(U_i)$ for all $j\geq J(i)$.
After passing to a subsequence of $(x_j)$, we may assume $x_j\in\varphi(U_j)$ and pick $y_j\in U_j$ satisfying $x_j=\varphi(y_j)$.
Because $(y_j)$ is an infinite sequence contained in the compact set $\overline{U_1}$, after passing to a subsequence again we may assume that $\lim_jy_j$ exists.
With $i$ arbitrary, we have for all $j > i$ that $y_j\in U_j\subseteq U_{i+1}$, hence that $\lim_j y_j \in \overline{U_{i+1}} \subseteq U_i$.
This holds for all $i$, hence $\lim_jy_j\in\bigcap_iU_i=\{y\}$.
Thus, $y=\lim_iy_i$ and $\varphi(y_i)=x_i$, so $\varphi$ satisfies SLP.

Suppose $\varphi$ satisfies SLP at $y$.
Let $f \colon \calX \to \RR$ be a cost function on $\calX$ and $g = f \circ \varphi$.
Suppose $x = \varphi(y)$ is not a local minimum for $f$ on $\calX$, that is, there exists a sequence $(x_i)_{i\geq 1} \subseteq \calX$ converging to $x$ such that $f(x_i) < f(x)$ for all $i$.
Applying SLP, after passing to a subsequence we can find a sequence $(y_i)_{i\geq1} \subseteq \calM$ converging to $y$ such that $\varphi(y_i) = x_i$.
Since $g(y_i) = f(x_i) < f(x) = g(y)$ and $y_i \to y$, we conclude that $y$ is not a local minimum for $g$.
By contrapositive, this shows that $\varphi$ satisfies the {\localimplieslocal} property at $y$. 

For the last implication, we proceed by contrapositive once again.
Suppose $\varphi$ does not satisfy ASLP at $y$.
Then, we can find sequences $(x_i)_{i\geq 1}\subseteq\calX$ converging to $x$ and $(\epsilon_i)_{i \geq 1}\subseteq\RR_{> 0}$ converging to 0 such that no subsequence of $(x_i)$ can be approximately lifted to $\calM$ in the sense of ASLP.
Let $\bar B(x, \epsilon) = \{x'\in\calX:\dist(x, x') \leq \epsilon\}$.
Notice that $x=\varphi(y) \notin \bar B(x_i,\epsilon_i)$ for all but finitely many indices $i$, as otherwise the constant sequence $y_i \equiv y$ would give an approximate lift of a subsequence.
Since $x_i \to x$ and $\epsilon_i \to 0$, after passing to a subsequence we may assume that the closed balls $\bar B(x_i, \epsilon_i)$ are pairwise disjoint and none contain $x$.
Define the following sum of smooth bump functions centered at the $x_i$
\begin{align*}
    f(x') & = \begin{cases} -\exp\left(1-\frac{1}{1-(\dist(x_i, x')/\epsilon_i)^2}\right) & \textrm{ if } x'\in \bar B(x_i,\epsilon_i) \text{ for some $i$,} \\ 0 & \textrm{ otherwise.} \end{cases}
\end{align*}
This is well defined because the balls $\bar B(x_i,\epsilon_i)$ are disjoint.
(As a side note, we remark that if $\calX$ is a metric subspace of a Euclidean space $\calE$ as in our general treatment, then $f$ extends to a smooth function on $\calE$.)
Note that $x$ is not a local minimum for $f$ since $x_i\to x$ and $f(x_i)=-1<0=f(x)$.
However, $y$ is a local minimum for $g = f\circ\varphi$.
Indeed, if there was a sequence $(y_i)$ converging to $y$ such that $g(y_i) < g(y) = 0$, then we would have $\varphi(y_i) \in \bar B(x_{n_i}, \epsilon_{n_i})$ for an infinite subsequence $(n_i)$ with $n_i \to \infty$ (since we must have $\varphi(y_i)\to x$ by continuity of $\varphi$), showing that $(y_i)$ is an approximate lift of the subsequence $(x_{n_i})$: a contradiction to our assumptions about $(x_i), (\epsilon_i)$.
Thus, $\varphi$ does not satisfy the {\localimplieslocal} property at $y$.
\end{proof}

\begin{lemma}\label{lem:nice_basis}
Suppose $\calM$ is Hausdorff, second-countable, and locally compact. Then for any $y\in\calM$ there is a sequence of open neighborhoods $U_i$ of $y$ with compact closures such that $U_i \supseteq \overline{U_{i+1}}$ and $\bigcap_{i=1}^{\infty} U_i = \{y\}$. 
\end{lemma}
\begin{proof}
Because $\calM$ is second-countable and locally compact, we can find a countable basis of open neighborhoods with compact closures $\{V_j\}_{j\geq 1}$ for $y$. Since $\calM$ is Hausdorff and $\{V_j\}$ is a basis for $y$, we have $\bigcap_{j=1}^{\infty}V_j=\{y\}$. Indeed, if $y'\neq y$, then there exists a neighborhood of $y$ not containing $y'$, and this neighborhood contains $V_i$ for some $i$ by definition of a local basis. By replacing $V_i$ by $\bigcap_{j=1}^iV_j$ (which preserves their intersection), we may assume $V_i\supseteq V_{i+1}$. We construct $\{U_i\}_{i\geq 1}$ inductively. Set $U_1=V_1$, which is an open neighborhood of $y$ with compact closure by assumption. Having constructed $U_1,\ldots,U_i$, use local compactness to find a compact neighborhood $K_{i+1}\subseteq V_{i+1}\cap U_i$ of $y$ and let $U_{i+1}$ be the interior of $K_{i+1}$. Then $U_{i+1}$ is an open neighborhood of $y$ by construction, and $\overline{U_{i+1}}\subseteq K_{i+1}\subseteq U_i$ which also shows $\overline{U_{i+1}}$ is compact as a closed subset of the compact set $K_{i+1}$. Finally, we have $\{y\}\subseteq\bigcap_{i=1}^{\infty}U_i\subseteq \bigcap_{i=1}^{\infty}V_i=\{y\}$ hence $\bigcap_{i=1}^{\infty}U_i=\{y\}$. 
\end{proof}

\ifthenelse{\boolean{shortver}}{}
    {

\section{Basic properties of $A_y$ and $B_y$}\label{apdx:AB_sets}
\begin{proof}[Proof of Proposition~\ref{prop:AB_basics}]
    \begin{enumerate}[(a)]
        \item For $w\in A_y$, let $c\colon \RR\to\calM$ satisfy $c(0)=y$, $(\varphi\circ c)'(0)=0$ and $(\varphi\circ c)''(0)=w$. For any $\alpha\geq 0$ define $\widetilde c(t)=c(\sqrt{\alpha} t)$. Note that $\widetilde c(0)=y$, $(\varphi\circ \widetilde c)'(0)=\sqrt{\alpha}(\varphi\circ c)'(0)=0$ and $(\varphi\circ\widetilde c)''(0)=\alpha(\varphi\circ c)''(0)\in A_y$. Thus, $\alpha w\in A_y$ so $A_y$ is a cone. The proof that $B_y$ is a cone is analogous. 
        
        
        Suppose $w_j\in B_y$ and $w_j\to w$. Let $c_{j,i}\colon \RR\to\calM$ satisfy $(\varphi\circ c_{j,i})'(0)\xrightarrow{i\to\infty}0$ and $(\varphi\circ c_{j,i})''(0)\xrightarrow{i\to\infty} w_j$. For each $j$, pick $i_j$ large enough so that $\|(\varphi\circ c_{j,i_j})'(0)\|\leq 1/j$ and $\|(\varphi\circ c_{j,i_j})''(0)- w_j\|\leq 1/j$. Set $\tilde c_j = c_{j,i_j}$ and observe that $(\varphi\circ\tilde c_j)'(0)\to 0$ and $(\varphi\circ\tilde c_j)''(0)\to w$, showing $w\in B_y$ so $B_y$ is closed.
        
        \item We prove the last claim in part (b).
        Since $0\in \im\Lmap_y$, we trivially have $A_y + \im\Lmap_y\supseteq A_y$ and similarly for $B_y$. Conversely, if $w\in A_y$, let $c\colon \RR\to\calM$ satisfy $c(0)=y$, $(\varphi\circ c)'(0)=0$ and $(\varphi\circ c)''(0)=w$. By Lemma~\ref{lem:diff_of_accel}(b), for any $v\in\T_y\calM$ there exists a curve $\widetilde c\colon \RR\to\calM$ satisfying $\widetilde c(0)=y$, $\widetilde c'(0)=c'(0)$, and $\widetilde c''(0) = c''(0) + v$, the sum being taken inside $\T_y\calM$. Then 
        \begin{align*}
            (\varphi\circ \widetilde c)'(0)=\D\varphi(y)[\widetilde c'(0)] = \D\varphi(y)[c'(0)] = (\varphi\circ c)'(0) = 0,
        \end{align*}
        and by Lemma~\ref{lem:diff_of_accel}(a), $(\varphi\circ\widetilde c)''(0) = (\varphi\circ c)''(0) + \Lmap_y(v) = w + \Lmap_y(v)$. This shows $A_y+\im\Lmap_y\subseteq A_y$. The proof that $B_y + \im\Lmap_y\subseteq B_y$ is analogous.
        
        
        
        
    \end{enumerate}
The proofs of the first half of part (b) and parts (c)-(d) are given in the main body.
\end{proof}

To see that $B_y\not\subseteq \T_x\calX$ in general, consider $B_y$ for the lift $\varphi\colon\calM=\RR\times(\RR^2)^2\to\calX=\RR^{2\times 2}_{\leq 1}$ given by $\varphi(\lambda,u,v) = \lambda uv^\top$ at $y=(0,e_1,e_1)$ where $e_1=\begin{bmatrix}1, & 0\end{bmatrix}^\top$.
We remark that {\twoimpliesone} does not hold for this lift by Proposition~\ref{prop:multilinear_lifts_norm}. To see that $\T_x\calX\not\subseteq B_y$ in general, note that if $\T_x\calX\subseteq B_y$ then $B_y^*\subseteq (\T_x\calX)^*$ and hence {\twoimpliesone} holds at $y$ by Corollary~\ref{cor:AB_suff_cond}(b). In particular, we have $\T_x\calX\not\subseteq B_y$ for the above example as well.

\section{Compositions and products of lifts}\label{apdx:comps_and_prods}
\subsection{Compositions}\label{apdx:comps}
\begin{proof}[Proof of Proposition~\ref{prop:composition_of_lifts}]
Recall that {\localimplieslocal} is equivalent to openness by Theorem~\ref{thm:localimplocalcharact}.
    \begin{enumerate}[(a)]
        \item For the first statement, suppose $\varphi\circ\psi$ is open at $z$ and let $V\subseteq\calM$ be a neighborhood of $y$. Because $\psi$ is continuous, $\psi^{-1}(V)$ is a neighborhood of $z$, so $(\varphi\circ\psi)(\psi^{-1}(V))$ is a neighborhood of $\varphi(y)$, hence $\varphi(V)\supseteq (\varphi\circ\psi)(\psi^{-1}(V))$ is a neighborhood of $\varphi(y)$ as well. This shows $\varphi$ is open at $y$, and hence satisfies {\localimplieslocal} there by Theorem~\ref{thm:localimplocalcharact}.
        
        For the second statement, if $U\subseteq\mc N$ is a neighborhood of $z$, then $\psi(U)\subseteq\calM$ is a neighborhood of $y$ since $\psi$ is open at $z$. If $\varphi$ satisfies {\localimplieslocal}, or equivalently, is open at $y$, then $\varphi(\psi(U))\subseteq\calX$ is a neighborhood of $\varphi(y)$, hence $\varphi\circ\psi$ is open at $z$ and satisfies {\localimplieslocal} there by Theorem~\ref{thm:localimplocalcharact}.
        
        \item For {\oneimpliesone}, the chain rule gives $\D(\varphi\circ\psi)(z) = \D\varphi(y)\circ\D\psi(z)$ for any $z\in\mc N$. If $\varphi\circ\psi$ satisfies {\oneimpliesone}, then Theorem~\ref{thm:oneimpliesone_char} shows that $\im \D(\varphi\circ\psi)(z) = \T_x\calX$, hence $\im\D\varphi(y)\supseteq \T_x\calX$. Since $\im\D\varphi(y)\subseteq \T_x\calX$ by Lemma~\ref{lem:Ay1_props}, we conclude that $\im\D\varphi(y)=\T_x\calX$ and hence $\varphi$ satisfies {\oneimpliesone} at $y$ by Theorem~\ref{thm:oneimpliesone_char}. Conversely, suppose $\psi$ is a submersion at $z$, so $\D\psi(z)$ is surjective. Then $\im\D\varphi(y) = \im\D(\varphi\circ\psi)(z)$. By Theorem~\ref{thm:oneimpliesone_char}, we conclude that $\varphi$ satisfies {\oneimpliesone} at $y$ if and only if $\varphi\circ\psi$ does so at $z$. 
        
        For {\twoimpliesone}, suppose first that $\varphi\circ\psi$ satisfies {\twoimpliesone} at $z$ and $y$ is 2-critical for $f\circ\varphi$ where $f$ is some cost function on $\calX$. For any curve $\bar c\colon \RR\to\mc N$ such that $\bar c(0)=z$, the curve $c = \psi\circ \bar c\colon \RR\to\calM$ satisfies $c(0)=y$. Because $y$ is 2-critical, we have 
        \begin{align}
            (f\circ\varphi\circ c)'(0) = (f\circ\varphi\circ\psi\circ \bar c)'(0)=0,\quad \textrm{and}\quad (f\circ\varphi\circ c)''(0) = (f\circ\varphi\circ\psi\circ \bar c)''(0)\geq 0.
            \label{eq:2imp1_comps}
        \end{align}
        This shows $z$ is 2-critical for $f\circ\varphi\circ\psi$ on $\mc N$. Because $\varphi\circ\psi$ satisfies {\twoimpliesone}, we conclude that $\varphi(\psi(z))=\varphi(y)$ is stationary for $f$ on $\calX$, hence $\varphi$ satisfies {\twoimpliesone} at $y$.
        
        Conversely, suppose $\psi$ is a submersion at $z$ and $\varphi$ satisfies {\twoimpliesone} at $y$. Suppose $z$ is 2-critical for $f\circ\varphi\circ\psi$. Because $\psi$ is a submersion at $z$, for any curve $c\colon \RR\to\calM$ satisfying $c(0)=y$, there exists a potentially smaller interval $I\subseteq \RR$ containing $t=0$ in its interior and a curve $\bar c\colon I\to\mc N$ such that $\psi\circ\bar c = c$. For example, by~\cite[Thm.~4.26]{lee_smooth} there exists a smooth local section $\sigma$ of $\psi$ defined near $y$ satisfying $\sigma(y)=z$, in which case we can set $\bar c = \sigma\circ c$. Because $z$ is 2-critical, it follows that $y$ is 2-critical for $f\circ\varphi$ by~\eqref{eq:2imp1_comps}. Since $\varphi$ satisfies {\twoimpliesone} at $y$, we conclude that $\varphi(y)=\varphi(\psi(z))$ is stationary for $f$, which shows $\varphi\circ\psi$ satisfies {\twoimpliesone} at $z$.
        
        \item The first claimed equality is just the chain rule $\D(\varphi\circ\psi)(z) = \D\varphi(y)\circ\D\psi(z)$. 
        For the second claimed equality, let $v = \Lmap_z^{\psi}(u)\in\T_y\calM$. Recall that $\Qmap_z^{\varphi\circ\psi}(u)=(\varphi\circ\psi\circ \bar c_u)''(0)$ where $\bar c_u\colon \RR\to\mc N$ is some curve satisfying $\bar c_u(0)=z$ and $\bar c_u'(0)=u$, and $\Qmap_y^{\varphi}(v)=(\varphi\circ c_{v})''(0)$ for some curve $c_{v}$ on $\calM$ satisfying $c_{v}(0)=y$ and $c_{v}'(0)=v$. 
        Since $\psi\circ\bar c_u$ is another curve on $\calM$ satisfying $(\psi\circ\bar c_u)(0)=y$ and $(\psi\circ\bar c_u)'(0)=\Lmap_z^{\psi}(u)=v$, Lemma~\ref{lem:diff_of_accel} shows that $\Qmap_y^{\varphi}(\Lmap_z^{\psi}(u))-\Qmap_z^{\varphi\circ\psi}(u) \in \im\Lmap_y^{\varphi}=\im\Lmap_z^{\varphi\circ\psi}$ for all $u\in\T_z\mc N$ as claimed, where we used the fact that $\Lmap_z^{\varphi\circ\psi}=\Lmap_y^{\varphi}\circ\Lmap_z^{\psi}$ and the surjectivity of $\Lmap_z^{\psi}$.
        
        Since $\Lmap_z^{\varphi\circ\psi} = \Lmap_y^{\varphi}\circ\Lmap_z^{\psi}$, we have $\ker\Lmap_z^{\varphi\circ\psi}=(\Lmap_z^{\psi})^{-1}(\ker\Lmap_y^{\varphi})$. Because $\Lmap_z^{\psi}$ is surjective, we have $\Lmap_z^{\psi}\Big((\Lmap_z^{\psi})^{-1}(\ker\Lmap_y^{\varphi})\Big)=\ker\Lmap_y^{\varphi}$. Using the results of the preceding paragraph,
        \begin{align*}
            A_z^{\varphi\circ\psi} &= \Qmap_z^{\varphi\circ\psi}(\ker\Lmap_z^{\varphi\circ\psi}) + \im\Lmap_z^{\varphi\circ\psi} = \Qmap_y^{\varphi}\circ\Lmap_z^{\psi}\Big((\Lmap_z^{\psi})^{-1}(\ker\Lmap_y^{\varphi})\Big) + \im\Lmap_y^{\varphi}\\ &= \Qmap_y^{\varphi}(\ker\Lmap_y^{\varphi}) + \im\Lmap_y^{\varphi} = A_y^{\varphi}.
        \end{align*}
        If $w\in B_y^{\varphi}$ then there exist $v_i\in\T_y\calM$ such that $\Lmap_y^{\varphi}(v_i)\to0$ and $w\in\lim_i(\Qmap_y^{\varphi}(v_i)+\im\Lmap^{\varphi}_y)$. Since $\Lmap_z^{\psi}$ is surjective, there exist $u_i\in \T_z\mc N$ satisfying $\Lmap_z^{\psi}(u_i)=v_i$. We then have $\Lmap_z^{\varphi\circ\psi}(u_i) = \Lmap_y^{\varphi}(v_i)\to 0$, and $\Qmap_z^{\varphi\circ\psi}(u_i)+\im\Lmap_z^{\varphi\circ\psi}=\Qmap_y^{\varphi}(v_i)+\im\Lmap_y^{\varphi}$. Taking $i\to\infty$, we conclude that $w\in B_z^{\varphi\circ\psi}$ and $B_y^{\varphi}\subseteq B_z^{\varphi\circ\psi}$. 
        
        Conversely, if $w\in B_z^{\varphi\circ\psi}$ then there exist $u_i\in\T_z\mc N$ such that $\Lmap_z^{\varphi\circ\psi}(u_i)\to0$ and $w\in\lim_i(\Qmap_z^{\varphi\circ\psi}(u_i)+\im\Lmap_z^{\varphi\circ\psi})$. Let $v_i=\Lmap_z^{\psi}(u_i)$ to conclude that $w\in B_y^{\varphi}$, and hence $B_y^{\varphi} = B_z^{\varphi\circ\psi}$.
        
        The argument for part (b) shows that $y$ is 2-critical for $f\circ\varphi$ if and only if $z$ is 2-critical for $f\circ\varphi\circ\psi$, which implies $W_y^{\varphi} = W_z^{\varphi\circ\psi}$ by Definition~\ref{def:W_y}. \qedhere
    \end{enumerate}
\end{proof}

\subsection{Products}\label{apdx:prods}
\begin{proof}[Proof of Proposition~\ref{prop:prod_lifts}]
\begin{enumerate}[(a)]
    \item If sequences $((y_1^{(j)},\ldots,y_k^{(j)}))_{j\geq 1}\subseteq\mc Z$ and $\tau_j\to 0$ such that $\tau_j>0$ satisfy 
    \begin{align*}
        (v_1,\ldots,v_k) = \lim_{j\to\infty}\frac{(y_1^{(j)},\ldots,y_k^{(j)})-(y_1,\ldots,y_k)}{\tau_j},
    \end{align*}
    then $v_i = \lim_j\frac{y_i^{(j)}-y_i}{\tau_j}$ for all $i=1,\ldots,k$, which shows that if $(v_1,\ldots,v_k)\in \T_{(z_1,\ldots,z_k)}\mc Z$ then $v_i\in \T_{z_i}\mc Z_i$ for all $i$. For an example where the inclusion is strict, see~\cite[Prop.~6.41]{rockafellar2009variational}.
    
    \item Since products of open sets form a basis for the (product) topology on $\mc Z$, we conclude that $\psi$ is open at $y=(y_1,\ldots,y_k)$ iff $\psi(U_1\times\cdots\times U_k)=\psi_1(U_1)\times\cdots\times\psi_k(U_k)$ is a neighborhood of $y$ whenever $U_i\subseteq\calM_i$ is a neighborhood of $y_i$ for all $i$. Because the interior of a product of sets in the product topology is the product of their interiors, we conclude that $\psi_1(U_1)\times\cdots\times\psi_k(U_k)$ is a neighborhood of $y$ if and only if $\psi_i(U_i)$ is a neighborhood of $y_i$ for all $i$. This proves part (b) by Theorem~\ref{thm:localimplocalcharact}.
    
    \item We have $\Lmap_{(y_1,\ldots,y_k)}^{\psi} = \Lmap_{y_1}^{\psi_1}\times\cdots\times\Lmap_{y_k}^{\psi_k}$, which is defined on $\T_{(y_1,\ldots,y_k)}\mc N = \T_{y_1}\mc N_1\times\cdots\times \T_{y_k}\mc N_k$. Therefore, $\im\Lmap_{(y_1,\ldots,y_k)}^{\psi}=\im\Lmap_{y_1}^{\psi_1}\times\cdots\times\im\Lmap_{y_k}^{\psi_k}$. If $\psi_i$ satisfies {\oneimpliesone} at $y_i$ for all $i$, then $\im\Lmap_{y_i}^{\psi_i}=\T_{z_i}\mc Z_i$ for all $i$ by Theorem~\ref{thm:oneimpliesone_char}, so $\im\Lmap_{(y_1,\ldots,y_k)}^{\psi}=\T_{z_1}\mc Z_1\times\cdots\times\T_{z_k}\mc Z_k$ where $z_i=\psi_i(y_i)$. Since $\im\Lmap_{(y_1,\ldots,y_k)}^{\psi}$ is always contained in $\T_{(z_1,\ldots,z_k)}\mc Z$ by Proposition~\ref{prop:AB_basics}(b), we get the chain of inclusions
    \begin{align*}
        \T_{z_1}\mc Z_1\times\ldots\times\T_{z_k}\mc Z_k = \im\Lmap_{(y_1,\ldots,y_k)}^{\psi} \subseteq \T_{(z_1,\ldots,z_k)}\mc Z\subseteq \T_{z_1}\mc Z_1\times\ldots\times\T_{z_k}\mc Z_k,
    \end{align*}
    where the last inclusion is part (a). We conclude that equality in part (a) holds and $\im\Lmap_{(y_1,\ldots,y_k)}^{\psi}=\T_{(z_1,\ldots,z_k)}\mc Z$ so {\oneimpliesone} holds.
    
    \item For any $v=(v_1,\ldots,v_k)\in\T_{(y_1,\ldots,y_k)}\calM$, let $c_{v_i}$ be a smooth curve on $\calM_i$ satisfying $\Qmap_{y_i}^{\psi_i}(v_i)=(\psi_i\circ c_{v_i})''(0)$ as in Definition~\ref{def:LQ_maps}. Then $c(t)=(c_{v_1}(t),\ldots,c_{v_k}(t))$ is a smooth curve on $\calM$ passing through $(y_1,\ldots,y_k)$ with velocity $v$ and 
    \begin{align*}
        (\psi\circ c)''(0) = ((\psi_1\circ c_{v_1})''(0),\ldots,(\psi_1\circ c_{v_k})''(0)) = (\Qmap_{y_1}^{\psi_1}\times\cdots\times\Qmap_{y_k}^{\psi_k})(v),
    \end{align*}
    showing that $\Qmap_{(y_1,\ldots,y_k)}^{\psi}\equiv\Qmap_{y_1}^{\psi_1}\times\cdots\times\Qmap_{y_k}^{\psi_k}\mod\im\Lmap_{(y_1,\ldots,y_k)}^{\psi}$ by Lemma~\ref{lem:diff_of_accel}. By part (c) above, we also have 
    \begin{align}
        \ker\Lmap_{(y_1,\ldots,y_k)}^{\psi} = \ker(\Lmap_{y_1}^{\psi_1}\times\cdots\times\Lmap_{y_k}^{\psi_k}) = \ker\Lmap_{y_1}^{\psi_1}\times\cdots\times\ker\Lmap_{y_k}^{\psi_k},
        \label{eq:ker_of_prod}
    \end{align}
    so Proposition~\ref{prop:2implies1_with_Qmap}(a) gives
    \begin{align*}
        A_{(y_1,\ldots,y_k)}^{\psi} &= \Qmap_{(y_1,\ldots,y_k)}^{\psi}(\ker\Lmap_{(y_1,\ldots,y_k)}^{\psi}) + \im\Lmap_{(y_1,\ldots,y_k)}^{\psi}\\ &= \Qmap_{y_1}^{\psi_1}(\ker\Lmap_{y_1}^{\psi_1})\times\cdots\times\Qmap_{y_k}^{\psi_k}(\ker\Lmap_{y_k}^{\psi_k}) + \im\Lmap_{y_1}^{\psi_1}\times\cdots\times\im\Lmap_{y_k}^{\psi_k}\\
        &= \Big(\Qmap_{y_1}(\ker\Lmap_{y_1}^{\psi_1}) + \im\Lmap_{y_1}^{\psi_1}\Big)\times\cdots\times\Big(\Qmap_{y_k}^{\psi_k}(\ker\Lmap_{y_k}^{\psi_k}) + \im\Lmap_{y_k}^{\psi_k}\Big)\\
        &= A_{y_1}^{\psi_1}\times\cdots\times A_{y_k}^{\psi_k}.
    \end{align*}
    The proof for $B_{(y_1,\ldots,y_k)}$ similarly follows from Proposition~\ref{prop:2implies1_with_Qmap}(b).  
    Now recall Definition~\ref{def:phi_w}. For any $(w_1,\ldots,w_k)\in\calE_1\times\cdots\times\calE_k$ we have
    \begin{align*}
        \psi_{(w_1,\ldots,w_k)}(y_1,\ldots,y_k)=\langle (w_1,\ldots,w_k),\psi(y_1,\ldots,y_k)\rangle = \sum_{i=1}^k\langle w_i,\psi_i(y_i)\rangle = \sum_{i=1}^k(\psi_i)_{w_i}(y_i),
    \end{align*}
    where the first and last equalities are the definitions of $\psi_{(w_1,\ldots,w_k)}(y_1,\ldots,y_k),(\psi_i)_{w_i}(y_i)$ from Definition~\ref{def:phi_w}. Therefore,
    \begin{align*}
        \nabla^2\psi_{(w_1,\ldots,w_k)}(y_1,\ldots,y_k)[\dot y_1,\ldots,\dot y_k]=\Big(\nabla^2(\psi_1)_{w_1}(y_1)[\dot y_1],\ldots,\nabla^2(\psi_k)_{w_k}(y_k)[\dot y_k]\Big).
    \end{align*} 
    The claimed expression for $W_{(y_1,\ldots,y_k)}^{\psi}$ follows from Proposition~\ref{prop:2implies1_with_Qmap}(c) together with~\eqref{eq:ker_of_prod}.
    
    Taking duals in part (a), we get $(\T_{(z_1,\ldots,z_k)}\mc Z)^*\supseteq (\T_{z_1}\mc Z_1)^*\times\cdots\times(\T_{z_k}\mc Z_k)^*$ where we used the fact that the dual of a product of cones is the product of their duals. If $\psi_i$ satisfies {\twoimpliesone} for all $i$, then $W_{y_i}^{\psi_i}\subseteq(\T_{z_i}\mc Z_i)^*$ by Theorem~\ref{thm:W_equiv_cond}, in which case 
    \begin{align*}
        W_{(y_1,\ldots,y_k)}^{\psi}=W_{y_1}^{\psi_1}\times\cdots\times W_{y_k}^{\psi_k}\subseteq (\T_{z_1}\mc Z_1)^*\times\cdots\times(\T_{z_k}\mc Z_k)^*\subseteq (\T_{(z_1,\ldots,z_k)}\mc Z)^*.
    \end{align*}
    Thus, $\psi$ satisfies {\twoimpliesone} at $(y_1,\ldots,y_k)\in\mc N$.
    
    \item This follows from (a) and (d). \qedhere
\end{enumerate}
\end{proof}


\section{Proof of Proposition~\ref{prop:svd_saxs} (SVD lifts)}\label{apdx:svd_lift} 
We first analyze the SVD lift, and then its modification.

\vspace{-1.365em}~\paragraph{SVD lift:} Considering the lift~\eqref{eq:SVD_lift}.
%
Fix $(U,\sigma,V)\in\calM$ such that the entries of $|\sigma|$ are all nonzero but not distinct. Choose $k<\ell$ such that $|\sigma_k|=|\sigma_{\ell}|$ and let $X = \varphi(U,\sigma,V)=U\diag(\sigma)V^\top$. We show {\localimplieslocal} does not hold at $(U,\sigma,V)$ by constructing a sequence converging to $X$ such that no subsequence of it can be lifted to a sequence converging to $(U,\sigma,V)$.
Choose $\alpha_i\in\RR^r$ converging to zero such that $|\sigma+\alpha_i|$ has distinct entries which are all nonzero. 
Let $Q^{(k,\ell)}\in O(r)$ be the Givens rotation matrix rotating by $\pi/4$ in $(k,\ell)$ plane, given explicitly by
\begin{align*}
    Q^{(k,\ell)}_{r,s} = \begin{cases} 1/\sqrt{2} & \textrm{if } r=s=k \textrm{ or } r=s=\ell,\\ 1/\sqrt{2} & \textrm{if } r=\ell,s=k,\\ -1/\sqrt{2} & \textrm{if } r=k,s=\ell,\\ 1 & \textrm{if } r=s\notin\{k,\ell\},\\ 0 & \textrm{otherwise},\end{cases}
\end{align*}
and let $U^{(k,\ell)} = U\mathrm{sign}(\sigma)Q^{(k,\ell)}\mathrm{sign}(\sigma)\in\mathrm{St}(m,r)$ and $V^{(k,\ell)}=VQ^{(k,\ell)}\in\mathrm{St}(n,r)$.
Define
\begin{align*}
    X_i= \varphi(U^{(k,\ell)},\sigma + \alpha_i, V^{(k,\ell)})\in\calM,
\end{align*}
whose limit is 
\begin{align*}
    \lim_{i\to\infty}X_i = U^{(k,\ell)}\mathrm{diag}(\sigma)(V^{(k,\ell)})^\top = U\mathrm{sign}(\sigma)\Big[Q^{(k,\ell)}\diag(|\sigma|)(Q^{(k,\ell)})^\top\Big] V^\top = U\mathrm{diag}(\sigma)V^\top = X.
\end{align*}
The third equality above follows since $|\sigma_k|=|\sigma_{\ell}|$ so the corresponding submatrix of $\diag(\sigma)$ is a multiple of the identity, and $Q^{(k,\ell)}$ is orthogonal and acts by the identity outside of that submatrix. 
Suppose $(U_{i_j},\sigma_{i_j},V_{i_j})$ is a lift of some subsequence of $(X_i)$ converging to $(U,\sigma,V)$. The singular values of $X_{i_j}$ are the entries of $|\sigma_{i_j}|$, which are also the entries of $|\sigma + \alpha_{i_j}|$. Since all these singular values are distinct, $U_{i_j}$ and $V_{i_j}$ must contain the $k$th and $\ell$th columns of $U^{(k,\ell)}$ and $V^{(k,\ell)}$ up to sign, since the singular vectors are unique up to sign~\cite[Thm.~4.1]{trefethen1997numerical}. But then it cannot happen that $(U_{i_j},V_{i_j})\to(U,V)$ by construction of $Q^{(k,\ell)}$, a contradiction. 
Thus, no subsequence of $(X_i)$ can be lifted to a sequence on $\calM$ converging to $(U,\sigma,V)$, showing that {\localimplieslocal} does not hold there. Proposition~\ref{prop:1imp1_implies_locimploc_at_smooth} together with the fact that $\calX^{\mathrm{smth}}=\RR^{m\times n}_{=r}$ implies that {\oneimpliesone} does not hold at such $(U,\sigma,V)$ either.

Next, fix $(U,\sigma,V)\in\calM$ such that $\sigma_k=0$ for some $k$, and let $X=\varphi(U,\sigma,V)$. Since $\rank(X)<r<\min(m,n)$, there exist unit vectors $u_k'\in\RR^m$ and $v_k'\in\RR^n$  such that $U^\top u_k'=0$ and $V^\top v_k'=0$.
Let $U^{(k)},V^{(k)}$ be obtained from $U,V$ by replacing their $k$th columns with $u_k',v_k'$, respectively, and let $\alpha_i\in\RR^r$ be a sequence converging to zero such that $|\sigma+\alpha_i|$ has distinct entries that are all nonzero. Define $X_i = \varphi(U^{(k)},\sigma+\alpha_i,V^{(k)})$ which converge to $X$ as $i\to\infty$. 
If $(U_{i_j},\sigma_{i_j},V_{i_j})$ is a lift of a subsequence of $(X_i)$ converging to $(U,\sigma,V)$, then one of the columns of $U_{i_j},V_{i_j}$ must be $u_k',v_k'$ up to sign because $|\sigma_{i_j}|$ has distinct entries which are the singular values of $X_{i_j}$. This contradicts $U_{i_j}\to U$ and $V_{i_j}\to V$.
Thus, {\localimplieslocal} does not hold at such $(U,\sigma,V)$. Theorem~\ref{thm:oneimpliesone_char} shows that {\oneimpliesone} does not hold there either since $\T_X\calX$ is not a linear space.


Finally, fix $(U,\sigma,V)\in\calM$ such that all the entries of $|\sigma|$ are nonzero and distinct. We verify that {\oneimpliesone} holds there using Theorem~\ref{thm:oneimpliesone_char} by showing $\im\Lmap_{(U,\sigma,V)}=\T_X\calX$. Since $\calM$ is a product of embedded submanifolds of linear spaces, we have from~\eqref{eq:L_and_Q_embedded} that
\begin{align*}
    \T_{(U,\sigma,V)}\calM = \T_U\mathrm{St}(m,r)\times \T_{\sigma}\RR^r\times \T_V\mathrm{St}(n,r) = \{(\dot U,\dot\sigma,\dot V):U^\top\dot U+\dot U^\top U = V^\top \dot V + \dot V^\top V = 0\},
\end{align*}
and
\begin{align*}
    \Lmap_{(U,\sigma,V)}(\dot U,\dot \sigma,\dot V) = \D\varphi(U,\sigma,V)[\dot U,\dot \sigma, \dot V] = \dot U\diag(\sigma) V^\top + U\diag(\dot\sigma)V^\top + U\diag(\sigma)\dot V^\top,
\end{align*}
where $(\dot U,\dot\sigma,\dot V)\in\T_{(U,\sigma,V)}\calM$. 
Let $U_{\perp}\in\mathrm{St}(m,m-r)$ and $V_{\perp}\in\mathrm{St}(n,n-r)$ satisfy $U^\top U_{\perp}=0$ and $V^\top V_{\perp}=0$. By~\cite[\S7.3]{optimOnMans}, we have $\dot U\in\T_U\mathrm{St}(m,r)$ and $\dot V\in\T_V\mathrm{St}(n,r)$ if and only if 
\begin{align}
    \dot U = U\Omega_u + U_{\perp}B_u,\quad \dot V = V\Omega_v + V_{\perp}B_v,
    \label{eq:ST_tangents}
\end{align}
where $\Omega_u,\Omega_v\in\mathrm{Skew}(r)\coloneqq\{\Omega\in \RR^{r\times r}:\Omega^\top+\Omega=0\}$ and $B_u\in\RR^{(m-r)\times r}$, $B_v\in\RR^{(n-r)\times r}$ are arbitrary. 
Using this parametrization,
\begin{align}\label{eq:svd_Lmap}
    \Lmap_{(U,\sigma,V)}(\dot U,\dot \sigma,\dot V) = \begin{bmatrix} U & U_{\perp}\end{bmatrix}\begin{bmatrix} \Omega_u\diag(\sigma) - \diag(\sigma)\Omega_v + \diag(\dot\sigma) & \diag(\sigma)B_v^\top\\ B_u\diag(\sigma) & 0\end{bmatrix}\begin{bmatrix} V & V_{\perp}\end{bmatrix}^\top. 
\end{align}
Since $\T_X\calX$ is a linear space and $\dim\T_{(U,\sigma,V)}\calM=\dim\T_X\calX$, we have $\im\Lmap_{(U,\sigma,V)}=\T_X\calX$ if and only if $\Lmap_{(U,\sigma,V)}$ is injective. Suppose therefore that $\Lmap_{(U,\sigma,V)}(\dot U,\dot \sigma,\dot V)=0$. By~\eqref{eq:svd_Lmap} this is equivalent to
\begin{align}
    B_u = 0,\quad B_v = 0,\quad \dot\sigma = 0,\quad \Omega_u\diag(\sigma)-\diag(\sigma)\Omega_v = 0,
    \label{eq:kerL_svd}
\end{align}
where $\dot\sigma=0$ follows by considering the diagonal of the top left block in~\eqref{eq:svd_Lmap}. The fourth equality in~\eqref{eq:kerL_svd} together with the skew-symmetry of $\Omega_u,\Omega_v$ gives for all $i,j$ that
\begin{align*}
    (\Omega_u)_{i,j}=\frac{\sigma_i}{\sigma_j}(\Omega_v)_{i,j} \implies -\frac{\sigma_i}{\sigma_j}(\Omega_v)_{i,j} = (\Omega_u)_{j,i}=-\frac{\sigma_j}{\sigma_i}(\Omega_v)_{i,j}\implies \frac{\sigma_i^2-\sigma_j^2}{\sigma_i\sigma_j}(\Omega_v)_{i,j}=0.
\end{align*}
Since $|\sigma_i|\neq|\sigma_j|$ whenever $i\neq j$, we get $(\Omega_v)_{i,j}=0$ and $(\Omega_u)_{i,j}=0$ for all $i\neq j$, hence $\Omega_u=\Omega_v=0$. We conclude that $(\dot U,\dot\sigma,\dot V)=0$ so $\Lmap_{(U,\sigma,V)}$ is injective and {\oneimpliesone} holds. By Proposition~\ref{prop:1imp1_implies_locimploc_at_smooth}, {\localimplieslocal} holds there as well.

All cases have been checked, so the first bullet in Proposition~\ref{prop:svd_saxs} is proved.

\vspace{-1.25em}~\paragraph{Modified SVD lift:} We now consider the modified SVD lift $\mathrm{St}(m,r)\times\mbb S^r\times\mathrm{St}(n,r)\to\Rmnlr$ defined by $\varphi(U,M,V)=UMV^\top$.
Fix $(U,M,V)\in\calM$ such that $\rank(M)=r$ and $\lambda_k(M)+\lambda_{\ell}(M)=0$ for some $k<\ell$. We show {\localimplieslocal} does not hold at $(U,M,V)$ by constructing a sequence converging to $X=\varphi(U,M,V)$ such that no subsequence of it can be lifted to a sequence converging to $(U,M,V)$. Let $\alpha_i\in\RR^r$ be a sequence converging to zero such that $|\lambda(M)+\alpha_i|$ has distinct entries that are all nonzero. Let $M = W\mathrm{diag}(\lambda(M))W^\top$ be an eigendecomposition of $M$, where $W\in\mathrm{O}(r)$. Define
\begin{align}
    U^{(k,\ell)}=UWT^{(k,\ell)}S^{(k,\ell)}W^\top\in\mathrm{St}(m,r),\quad V^{(k,\ell)}=VWT^{(k,\ell)}W^\top\in\mathrm{St}(n,r),
    \label{eq:UV_mod_svd}
\end{align}
where $T^{(k,\ell)}\in \mathrm{O}(r)$ is the permutation interchanging the $k$th and $\ell$th entries of a vector while fixing all the others, and $S^{(k,\ell)}\in \mathrm{O}(r)$ flips the signs of these entries (so it's a diagonal matrix with all 1's on the diagonal except for the $k$th and $\ell$th entries which are $-1$). Note that $T^{(k,\ell)}$ and $S^{(k,\ell)}$ are symmetric and commute.
Let 
\begin{align*}
    X_i = \varphi(U^{(k,\ell)}, M + W\mathrm{diag}(\alpha_i)W^\top, V^{(k,\ell)}),
\end{align*}
whose singular values are $|\lambda(M)+\alpha_i|$ and are all distinct. Note that $X_i\to X$, because 
\begin{align*}
    \lim_{i\to\infty} X_i &= U^{(k,\ell)}M(V^{(k,\ell)})^\top = (UW)T^{(k,\ell)}(S^{(k,\ell)}\mathrm{diag}(\lambda(M)))(T^{(k,\ell)})^\top(VW)^\top\\ &= (UW)\mathrm{diag}(\lambda(M))(VW)^\top = UMV^\top = X,
\end{align*}
where the first equality on the second line follows because $S^{(k,\ell)}$ flips the signs of $\lambda_k(M)$ and of $\lambda_{\ell}(M)=-\lambda_k(M)$, and conjugation by $T^{(k,\ell)}$ interchanges them again.

Suppose $(U_{i_j},M_{i_j},V_{i_j})$ is a lift of a subsequence of $(X_i)$ converging to $(U,M,V)$. The singular values of $X_{i_j}$ are the entries $|\lambda(M_{i_j})|$, which are also the entries of $|\lambda(M) + \alpha_{i_j}|$ (possibly permuted).
Combining this with $\lambda(M_{i_j})\to \lambda(M)$ (since $M_{i_j}\to M$), $\alpha_{i_j}\to0$, and the fact that $\lambda(M)$ has no zero entries, we further get that the entries of $\lambda(M_{i_j})$ are equal to the entries of $\lambda(M)+\alpha_{i_j}$ for all large $j$. After passing to a subsequence, we may assume that this equality holds for all $j$.
Let $M_{i_j}=W_{i_j}\mathrm{diag}(\lambda(M)+\alpha_{i_j})W_{i_j}^\top$ be an eigendecomposition of $M_{i_j}$. 
Since $\mathrm{O}(r)$ is compact, we may also assume after further passing to a subsequence that 
\begin{align*}
    \lim_{j\to\infty}W_{i_j} = \widetilde W\in\mathrm{O}(r),
\end{align*}
exists.
At this point, we have two SVDs of $X_{i_j}$, namely
\begin{align*}
    X_{i_j} &= [U^{(k,\ell)}W\mathrm{sign}(\lambda(M) + \alpha_{i_j})]\mathrm{diag}(|\lambda(M)+\alpha_{i_j}|)[V^{(k,\ell)}W]^\top\\ 
    &= [U_{i_j}W_{i_j}\mathrm{sign}(\lambda(M) + \alpha_{i_j})]\mathrm{diag}(|\lambda(M)+\alpha_{i_j}|)[V_{i_j}W_{i_j}]^\top. 
\end{align*}
Because the singular values $|\lambda(M)+\alpha_{i_j}|$ of $X_{i_j}$ are distinct, its singular vectors are unique up to sign~\cite[Thm.~4.1]{trefethen1997numerical}. Specifically, there exists $S_{i_j}\in\mathrm{diag}(\{\pm 1\}^r)$ satisfying 
\begin{align*}
    &U^{(k,\ell)}W\mathrm{sign}(\lambda(M) + \alpha_{i_j}) = U_{i_j}W_{i_j}\mathrm{sign}(\lambda(M) + \alpha_{i_j})S_{i_j}\quad \textrm{implying}\quad U^{(k,\ell)}W = U_{i_j}W_{i_j}S_{i_j},\\
    &V^{(k,\ell)}W=V_{i_j}W_{i_j}S_{i_j},
\end{align*}
where in the first line we used the fact that diagonal matrices commute. Because $S_{i_j}$ takes values in a finite set, after passing to a subsequence again we may assume $S_{i_j}=S$ is fixed. Then $U^{(k,\ell)} = U_{i_j}W_{i_j}SW^\top$ and $V^{(k,\ell)} = V_{i_j}W_{i_j}SW^\top$ for all $j$, and taking $j\to\infty$ we conclude that $U^{(k,\ell)} = U\widetilde W SW^\top$ and $V^{(k,\ell)} = V\widetilde WSW^\top$. Equating this to~\eqref{eq:UV_mod_svd}, we obtain 
\begin{align*}
    U\widetilde WSW^\top = UWT^{(k,\ell)}S^{(k,\ell)}W^\top\quad \textrm{and}\quad V\widetilde WSW^\top = VWT^{(k,\ell)}W^\top. 
\end{align*}
Rearranging gives $W^\top\widetilde WS = T^{(k,\ell)}S^{(k,\ell)}$ and $W^\top\widetilde WS = T^{(k,\ell)}$ so in fact, $T^{(k,\ell)}S^{(k,\ell)}=T^{(k,\ell)}$, a contradiction. Thus, no subsequence of $(X_i)$ can be lifted, so the lift does not satisfy {\localimplieslocal} at $(U,M,V)$. By Proposition~\ref{prop:1imp1_implies_locimploc_at_smooth} and the fact that $X\in\RR^{m\times n}_{=r}=\calX^{\mathrm{smth}}$, we conclude that {\oneimpliesone} is not satisfied at $(U,M,V)$ either.

Now fix $(U,M,V)$ such that $\lambda_k(M)=0$ for some $k$, and let $X,W$ be as above. Since $r<\min\{m,n\}$, there exist unit vectors $u_k'\in\RR^m$ and $v_k'\in\RR^n$ such that $(UW)^\top u_k'=0$ and $(VW)^\top v_k'=0$.
Let $Y\in \mathrm{O}(m)$ and $Z\in\mathrm{O}(n)$ send the $k$th columns of $UW$ and $VW$ to $u_k'$ and $v_k'$, respectively, and act by the identity on their orthogonal complements. 
Let $\alpha_i\in\RR^r$ converge to zero such that $|\lambda(M)+\alpha_i|$ are distinct and nonzero. 
Define $X_i = \varphi(YU,M + \alpha_i,ZV)$, which converge to $YUM(ZV)^\top = X$. 
Suppose $(U_{i_j},M_{i_j},V_{i_j})$ is a lift of a subsequence of $(X_i)$ converging to $(U,M,V)$. 
Let $M_{i_j}=W_{i_j}\mathrm{diag}(\lambda(M) + \alpha_i)W_{i_j}^\top$ be an eigendecomposition. 
After passing to a subsequence, we may assume $W_{i_j}\to\widetilde W$ in $\mathrm{O}(r)$.
Because the singular vectors of $X_{i_j}$ are unique up to sign, there exists $S_{i_j}\in\mathrm{diag}(\{\pm 1\}^r)$ satisfying 
\begin{align*}
    YU = U_{i_j}W_{i_j}S_{i_j}W^\top,
\end{align*}
and similarly for $ZV$. Because $S_{i_j}$ takes values in a finite set, after passing to a subsequence again we may assume $S_{i_j} = S$ is fixed for all $j$, in which case we get 
\begin{align*}
    YU = U_{i_j}W_{i_j}SW^\top \to U\widetilde WSW^\top.    
\end{align*}
This is a contradiction since $\mathrm{col}(YU)\neq \mathrm{col}(U) = \mathrm{col}(U\widetilde WSW^\top)$ by construction of $Y$. Thus, no subsequence of $(X_i)$ can be lifted so {\localimplieslocal} does not hold at such $(U,M,V)$. Since $\T_X\calX$ is not a linear space, Theorem~\ref{thm:oneimpliesone_char} shows that {\oneimpliesone} does not hold there either.

Finally, fix $(U,M,V)\in\calM$ such that $\lambda_i(M)+\lambda_j(M)\neq 0$ for all $i,j$ and let $X=\varphi(U,M,V)$. We show that {\oneimpliesone} holds at $(U,M,V)$ by showing $\im\Lmap_{(U,M,V)}=\T_X\calX$ and appealing to Theorem~\ref{thm:oneimpliesone_char}. Note that $\T_{(U,M,V)}\calM = \T_U\mathrm{St}(m,r)\times \mbb S^r\times \T_V\mathrm{St}(n,r)$. Then~\eqref{eq:L_and_Q_embedded} gives
\begin{align*}
    \Lmap_{(U,M,V)}(\dot U,\dot M,\dot V) = \dot UMV^\top + U\dot MV^\top + UM\dot V^\top. 
\end{align*}
Writing $\dot U,\dot V$ as in~\eqref{eq:ST_tangents}, we get similarly to~\eqref{eq:svd_Lmap} that
\begin{align}
    \Lmap_{(U,M,V)}(\dot U,\dot M,\dot V) = \begin{bmatrix} U & U_{\perp}\end{bmatrix}\begin{bmatrix} \Omega_uM - M\Omega_v + \dot M & MB_v^\top \\ B_uM & 0 \end{bmatrix}\begin{bmatrix} V & V_{\perp}\end{bmatrix}^\top. 
    \label{eq:svd_mod_Lmap}
\end{align}
Since $\rank(M)=r$, we have $\mathrm{col}(U)=\mathrm{col}(X)$ and $\mathrm{col}(V)=\mathrm{col}(X^\top)$. By~\cite[Thm.~3.1]{schneider2015convergence}, we have
\begin{align*}
    \T_X\Rmnlr = \left\{\dot X=\begin{bmatrix} U & U_{\perp}\end{bmatrix}\begin{bmatrix} \dot X_1 & \dot X_2 \\ \dot X_3 & 0 \end{bmatrix}\begin{bmatrix} V & V_{\perp}\end{bmatrix}^\top: \dot X_1\in\RR^{r\times r}, \dot X_2\in\RR^{r\times(n-r)}, \dot X_3\in\RR^{(m-r)\times r}\right\}. 
\end{align*}
Since $\lambda_i(M)+\lambda_j(M)\neq 0$ for all $i,j$, for any $\dot X_1\in\RR^{r\times r}$ we can pick $\Omega\in\mathrm{Skew}(r)$ such that $\Omega M + M\Omega = \mathrm{skew}(\dot X_1)=\frac{\dot X_1-\dot X_1^\top}{2}$. Indeed, if $M=W\Lambda W^\top$ is an eigendecomposition of $M$, define $\Omega$ by setting
\begin{align*}
    (W^\top\Omega W)_{i,j} = \frac{(W^\top \mathrm{skew}(\dot X_1) W)_{i,j}}{\lambda_i(M)+\lambda_j(M)},
\end{align*}
which is clearly skew-symmetric and solves the above equation. We set $\Omega_u=-\Omega_v=\Omega$ and $\dot M = \mathrm{sym}(\dot X_1)=\frac{\dot X_1+\dot X_1^\top}{2}\in\mbb S^r$. Finally, we set
\begin{align*}
    B_v = \dot X_2^\top M^{-\top},\quad B_u = \dot X_3 M^{-1}.
\end{align*}
With these choices, we get $\Lmap_{(U,M,V)}(\dot U,\dot M,\dot V) = \dot X$, showing that $\Lmap_{(U,M,V)}$ is surjective and {\oneimpliesone} holds. By Proposition~\ref{prop:1imp1_implies_locimploc_at_smooth}, the lift satisfies {\localimplieslocal} at $(U,M,V)$ as well. 

\section{Notation and basic definitions}\label{apdx:notation}
Here we collect notation and standard definitions used throughout the paper.

\begin{itemize}
    \item $\calE$ is a linear space endowed with an inner product $\langle\cdot,\cdot\rangle$ and induced norm $\|\cdot\|$.
    
    \item $\calM$ is a smooth Riemannian manifold endowed with a Riemannian metric, also denoted $\langle\cdot,\cdot\rangle$, with its induced norm $\|\cdot\|$. The tangent space to $\calM$ at $y\in\calM$ is denoted $\T_y\calM$. 
    
    \item If $S\subseteq\calE$ is a subspace of an inner product space $\calE$, we denote by $\Proj_S(x)$ the orthogonal projection of $x\in\calE$ onto $S$.
    
    \item $\calX$ is endowed with its subspace topology from $\calE$, and $\calM$ is endowed with its manifold topology.
    
    \item A neighborhood of a point $x$ is a set that contains $x$ in its interior. (We do not require neighborhoods to be open.)
    
    \item A map $\varphi$ between two topological spaces is \emph{open at $y$} if it maps neighborhoods of $y$ in $\calM$ to neighborhoods of $\varphi(y)$ in $\calX$.
    Globally, a map is \emph{open} if it is open at all points, or equivalently if it maps open subsets of $\calM$ to open subsets of $\calX$.
    
    \item A smooth curve on $\calM$ passing through $y\in\calM$ with velocity $v\in\T_y\calM$ is a smooth map $c\colon \RR\to\calM$ satisfying $c(0)=y$ and $c'(0)=v$. 
    
    \item If $c\colon \RR\to\calM$ is a smooth curve on a Riemannian manifold $\calM$ such that $c(0)=y$, then $c''(0)\in \T_y\calM$ denotes its intrinsic (Riemannian) acceleration. Accordingly, if $\gamma=\varphi\circ c$ then $\gamma''(0)$ denotes its (standard) acceleration in the Euclidean space $\calE$.
    
    \item A cone is a set $K\subseteq\calE$ such that $v\in K$ implies $\alpha v\in K$ for all $\alpha>0$.
    \item The dual cone $K^*$ of a cone $K$ is $K^* = \{u\in \calE: \inner{u}{v}\geq 0 \textrm{ for all } v\in K\}$.
	We use the following properties throughout (see~\cite[Prop.~4.5]{deutsch2012best} for proofs):
	\begin{itemize}
	    \item The dual cone is always a closed convex cone.
	    \item If $K_1\subseteq K_2$, then $K_2^*\subseteq K_1^*$.
	    \item The bidual cone $K^{**}$ of $K$ is equal to the closure of its convex hull $K^{**}=\overline{\mathrm{conv}}(K)$. In particular, $K^{**}\supseteq K$.
	    \item If $K$ is a linear space, then its dual $K^*$ is equal to its orthogonal complement $K^{\perp}$.
	\end{itemize}
	
	\item If $\psi\colon\mc N\to\calM$ is a smooth map between smooth manifolds, its differential at $z\in\mc N$ is denoted $\D\psi(z)\colon \T_z\mc N\to\T_{\psi(z)}\calM$.
	\item If $g\colon \calM\to\RR$ is a smooth function, its Riemannian gradient at $y\in\calM$ is denoted $\nabla g(y)\in\T_y\calM$. It is the unique element of $\T_y\calM$ satisfying $\langle\nabla g(y),v\rangle = \D g(y)[v]$ for all $v\in\T_y\calM$. 
	
	The Riemannian Hessian of $g$ at $y$ is a self-adjoint linear map denoted $\nabla^2g(y)\colon \T_y\calM\to\T_y\calM$, see~\cite[\S8.11]{optimOnMans} for the definition.
	
	If $f\colon \calE\to\RR$ is a smooth function on a linear space, the usual definitions of the directional derivative $\D f(x)[v]$, the gradient $\nabla f(x)$, and the Hessian $\nabla^2f(x)$ coincide with the above definitions specialized to $\calM=\calE$. 
	
	\item We write $S\succeq0$ to denote a positive-semidefinite (PSD) matrix or a PSD self-adjoint linear map $S$, and we write $S\succ0$ to indicate $S$ is positive-definite.
\end{itemize}
}
\end{document}